\documentclass[a4paper,11pt]{amsart}
\usepackage{amscd,amsmath,amsfonts,amssymb,amsthm,enumerate,array,graphicx,stmaryrd}
\usepackage[margin=1.25in]{geometry}

\def\CC{\mathbb{C}}
\def\RR{\mathbb{R}}
\def\NN{\mathbb{N}}
\def\FF{\mathbb{F}}
\def\ZZ{\mathbb{Z}}

\theoremstyle{plain}
\newtheorem{theorem}{\bf Theorem}[section]
\newtheorem{lemma}[theorem]{\bf Lemma}
\theoremstyle{remark}
\newtheorem{definition}[theorem]{\bf Definition}
\newtheorem{problem}[theorem]{\bf Problem}

\newtheorem{example}[theorem]{\bf Example}
\newtheorem{remark}[theorem]{\bf Remark}
\newtheorem{proposition}[theorem]{\bf Proposition}
\newtheorem{corollary}[theorem]{\bf Corollary}
\numberwithin{equation}{section}

\newcommand{\Hilb}{\mathsf{Hilb}}
\newcommand{\Sem}{\mathsf{Sem}}

\newcommand{\Cat}{\mathsf{Cat}}
\DeclareMathOperator{\ran}{ran}
\DeclareMathOperator{\rk}{rank}

\title{Coarse comparison of Bessel families}
\author{Nizar El Idrissi}
\address{Department of Mathematics, Ibn Tofail University, Kenitra, Morocco}
\email{nizar.elidrissi@uit.ac.ma}

\begin{document}
\begin{abstract}
We classify Bessel families in separable Hilbert spaces under coarse equivalence relations: $f$ and $\tilde f$ are equivalent if the seminorms $\|T_f\,\cdot\|$ and $\|T_{\tilde f}\,\cdot\|$ of their analysis operators are mutually dominated, up to constants, after pullback along operators $w,\tilde w$ of a prescribed class $\mathcal O$. In contrast with orbit relations such as unitary or similarity equivalence, these relations ignore redundancy and the geometry of the index set, and compare frames, nonframes and compact families on the same footing. They are level $-1$ of the comparison categories of the companion article, and we compute the levels $0$ and $k\ge1$: they are trivial for bounded, injective, bounded below, isometric and upper semi-Fredholm comparison operators, and for invertible or unitary ones level $1$ is a stabilizer-type group of the family. For bounded $\mathcal O$ the classes are given by the singular numbers up to constants; injective operators add the kernel dimension, bounded below operators or isometries add upper semi-Fredholmness, and invertible, unitary or surjective operators give the kernel dimension and the dyadic multiplicity sequence of the spectrum up to bounded displacement on the logarithmic scale, which splits each remaining class into continuum many. Upper semi-Fredholm operators give only three noncompact classes, whereas Fredholm operators give continuum many, so that the Fredholm relation is strictly finer. For each relation we determine exactly which families are equivalent to an orthonormal basis, and we compute an operator Banach--Mazur distance, $\log(\sup_ns_n/\tilde s_n\cdot\sup_n\tilde s_n/s_n)$. As applications, we obtain a relative Paley--Wiener stability theorem for Bessel families and the invariance of Schatten membership and of $\int\|f\|^2$ under coarse equivalence.
\end{abstract}
\maketitle

\section{Introduction}
\label{section-introduction}

In functional analysis, equivalence relations are most often orbit relations: two functions or families are identified when their inputs and outputs can be interchanged by a group action. Unitary equivalence of frames \cite{Christensen2016,HanLarson2000}, its extension to $\mathrm{GL}(H)$ \cite{Balan1999}, and the generalized equivalence under $\mathrm L^0(\Omega,\mathbb S^1)\ltimes\mathbb U(H)$ \cite{NovikovSevostyanova2022} are of this kind, and Riesz bases are the sequences invertibly equivalent to an orthonormal basis. Such relations are rigid: they preserve every spectral invariant of the analysis operator and they only compare families living in isomorphic spaces.\par\bigskip
\noindent In this article, we investigate a family of equivalence relations on Bessel families that do not preserve every spectral invariant but can be classified completely by spectral data. These relations are an instance of the general construction of the companion article \cite{ElIdrissiCC}: they are level $-1$ of the comparison category attached to the transport datum of bounded seminorms on Hilbert spaces (Section \ref{section-categories}). Level $0$ and the levels $k\ge1$ of the same category are computed in Section \ref{section-categories} as well: they carry no information for the classes of comparison operators that contain the operators from the zero space, and they record a stabilizer-type group of the family for invertible and unitary comparison operators. The classification itself is self-contained and does not use that theory.\par\bigskip
\noindent\textbf{Why coarse relations.} The relations $\simeq_{\mathcal O}$ forget everything about a family except the quadratic form $x\mapsto\|T_fx\|^2$ up to constants, and up to reparametrization of the space by operators in $\mathcal O$. They are therefore insensitive to redundancy, to the index set and to the norms of the vectors, and they treat frames, Riesz bases, nonframes and compact (for instance Schatten class) families uniformly: for bounded comparison operators the only invariant is the decay of the singular numbers, as in the theory of operator ideals \cite{GohbergKrein1969,Pietsch1980}, whereas orbit relations see the whole spectral measure. Varying $\mathcal O$ produces a hierarchy whose successive refinements measure the kernel (injective operators), the closed range (bounded below operators) and the multiplicity of the spectrum at each dyadic scale (invertible operators), the last invariant being a coarse, bounded-displacement equivalence on the logarithmic scale. The construction is not tied to Bessel families: it applies to any system of observations in the sense of \cite[Definition 5.2]{ElIdrissiCC}, so that replacing bounded seminorms by other structures on Hilbert or Banach spaces gives further relations (Section \ref{section-complements}).\par\bigskip
\noindent\textbf{What the coarse classes are good for.} The invariants are the ones that survive the operations one actually performs on families, and they are stable under the perturbations that arise in practice. We give three consequences (Section \ref{section-applications}).
\begin{enumerate}[(a)]
\item \emph{Relative Paley--Wiener stability} (Proposition \ref{proposition-perturbation}). If $f,g$ are Bessel families in the same space with $\|T_gx-T_fx\|\le\lambda\|T_fx\|$ for all $x$ and some $\lambda<1$, then $f\simeq_\mathrm{uni}g$; hence $k(f)=k(g)$, the multiplicity sequences are related by a bijection of bounded displacement, $(1-\lambda)s_n(f)\le s_n(g)\le(1+\lambda)s_n(f)$ and $D(f,g)\le\log\frac{1+\lambda}{1-\lambda}$. No assumption of frame, Riesz or completeness type is needed.
\item \emph{Operator ideals and weighted sums} (Corollary \ref{corollary-ideals}). Membership of $T_f$ in every Schatten class is invariant under $\simeq$, and so is the finiteness of $\int_\Omega\|f(\omega)\|^2d\mu=\operatorname{tr}S_f$; the ratio of these quantities for equivalent families is controlled by the constant in Theorem \ref{theorem-singular}. In particular a family can be certified Hilbert--Schmidt by comparing it with any equivalent family with known norms, living in a different space and over a different index set.
\item \emph{Similarity invariance} (Proposition \ref{proposition-similarity}). For $w\in\mathrm{GL}(H)$ the family $w^*f$ is $\simeq_\mathrm{inv}$-equivalent to $f$, so the kernel dimension and the dyadic multiplicity sequence up to bounded displacement are invariants of the $\mathrm{GL}(H)$-orbit of $f$, whereas the exact spectral measure is not.
\end{enumerate}
\par\bigskip
\noindent\textbf{Convention.} All Hilbert spaces are separable (of finite or infinite dimension, the zero space included) and are taken over the same field $\FF \in\{\RR ,\CC \}$. Dimensions therefore take values in $\{0,1,2,\dots,\infty\}$. We write $\NN =\{1,2,3,\dots\}$.\par\bigskip
\noindent\textbf{Notation.} Let $f:\Omega\to H$ be a Bessel family in a Hilbert space $H$ over a measure space $\Omega$, with analysis operator $T_f:H\to L^2(\Omega)$, frame operator $S_f=T_f^*T_f$ and seminorm $\rho(f)=\|T_f\,\cdot\|$ (Definition \ref{definition-bessel}). For a positive operator $S$ let $E_S$ be its spectral measure (Section \ref{section-spectral}), let
\[N_S(\lambda)=\dim E_S(\lambda,\infty)H\qquad(\lambda>0)\]
be its spectral counting function and, for $n\ge1$, let $\lambda_n(S)=\sup\{\lambda>0:N_S(\lambda)\ge n\}$ be its $n$-th eigenvalue counted with multiplicity. The singular numbers of $T_f$ are $s_n(f)=\lambda_n(S_f)^{1/2}$, and we write $k(f)=\dim\ker T_f$. We use six classes of bounded operators between Hilbert spaces (Definition \ref{definition-classes}): all bounded operators, injective operators, bounded below operators, isometries, invertible operators and unitary operators. For a class $\mathcal O$ of such operators we write $f\simeq_{\mathcal O}\tilde f$ if there are operators $w:\tilde H\to H$ and $\tilde w:H\to\tilde H$ in $\mathcal O$ such that $\rho(f)\circ w$ and $\rho(\tilde f)$ are mutually dominated up to constants, and so are $\rho(\tilde f)\circ\tilde w$ and $\rho(f)$ (Definition \ref{definition-relation}); in the language of \cite{ElIdrissiCC} this is level $-1$ of the comparison category (Proposition \ref{proposition-level-minus}). We write $\simeq$ for the relation $\simeq_\mathrm{all}$ given by the class of all bounded operators, and $\simeq_\mathrm{inj}$, $\simeq_\mathrm{bb}$, $\simeq_\mathrm{iso}$, $\simeq_\mathrm{inv}$, $\simeq_\mathrm{uni}$ for the other five (further classes are introduced in Section \ref{section-further}). The analysis operator $T_f$ is \emph{upper semi-Fredholm} if $\ker T_f$ is finite dimensional and $\ran T_f$ is closed; for a sequence $f$ with $\dim\ker T_f<\infty$ this means that $f$ is a frame for its closed linear span, which has finite codimension.\par\bigskip
\noindent\textbf{Main results.} The paper has four parts: the classification of Bessel families under the relations defined by the six classes $\mathcal O_\mathrm{all},\dots,\mathcal O_\mathrm{uni}$ (Sections \ref{section-all}--\ref{section-inv}), an operator distance (Section \ref{section-distance}), further classes of comparison operators and the orthonormal bases (Sections \ref{section-further} and \ref{section-onb}), and applications together with the relation to comparison categories (Sections \ref{section-applications} and \ref{section-complements}). Table 1 and the diagram below collect the relations, their complete invariants and the number of classes of noncompact families.\par\bigskip
\noindent\emph{Arbitrary comparison operators} (Section \ref{section-all}). The relation $f\simeq\tilde f$ holds if and only if there are bounded operators with $w^*S_fw=S_{\tilde f}$ and $\tilde w^*S_{\tilde f}\tilde w=S_f$, if and only if the counting functions are comparable, $N_{S_{\tilde f}}(\lambda)\le N_{S_f}(c\lambda)$ and $N_{S_f}(\lambda)\le N_{S_{\tilde f}}(c\lambda)$ for some $c\in(0,1]$ (Theorem \ref{theorem-all}), if and only if the singular numbers satisfy $C^{-1}s_n(\tilde f)\le s_n(f)\le C\,s_n(\tilde f)$ (Theorem \ref{theorem-singular}). The rank, the compactness and the membership of $T_f$ in each Schatten class are invariants; the families of rank $r$ form one class; the classes of compact families of infinite rank are the nonincreasing null sequences modulo term-by-term comparability, so there are continuum many; and all families with noncompact $T_f$ form one class, which contains the orthonormal bases of infinite-dimensional spaces, all their frames, and many nonframes (Corollary \ref{corollary-classes}).\par\bigskip
\noindent\emph{Injective and bounded below comparison operators} (Sections \ref{section-inj} and \ref{section-closed}). The relation $\simeq_\mathrm{inj}$ is $\simeq$ together with equality of the kernel dimensions $k(f)=k(\tilde f)$ (Theorem \ref{theorem-inj}), so the noncompact families split into one class for each $k\in\{0,1,\dots,\infty\}$. Bounded below operators and isometries give the same relation $\simeq_\mathrm{bb}=\simeq_\mathrm{iso}$, which is $\simeq_\mathrm{inj}$ together with the requirement that $T_f$ is upper semi-Fredholm if and only if $T_{\tilde f}$ is (Theorem \ref{theorem-closed}); the proof rests on a reservoir lemma that embeds a bounded operator isometrically into a noncompact one whose spectrum accumulates at $0$ (Lemma \ref{lemma-reservoir}). On compact families $\simeq_\mathrm{bb}$ coincides with $\simeq_\mathrm{inj}$, and the noncompact families split into the classes $(k,\mathrm{sF})$ for finite $k$ and $(k,\mathrm{nsF})$ for $k\le\infty$ (Corollary \ref{corollary-closed}). A family is $\simeq_\mathrm{bb}$-equivalent to an orthonormal basis if and only if it is a frame of a space of the same dimension (Proposition \ref{proposition-frames}).\par\bigskip
\noindent\emph{Invertible and unitary comparison operators} (Section \ref{section-inv}). For a positive operator $S$ put $W_S(\alpha,\beta)=\dim E_S(\alpha,\beta]H$ and $d_j(S)=W_S(2^{-j-1},2^{-j})$, the multiplicity of the spectrum in the $j$-th dyadic band; we write $d_j(f)$ for $d_j(S_f)$. Invertible and unitary comparison operators give the same relation, characterized by the equality $k(f)=k(\tilde f)$ together with any of the following: the existence of a unitary $u$ with $c\,S_{\tilde f}\le u^*S_fu\le C\,S_{\tilde f}$; the comparability $W_{S_{\tilde f}}(\alpha,\beta)\le W_{S_f}(\alpha/K,K\beta)$ and symmetrically, of the window multiplicities; or the existence of a bijection of bounded displacement between the multiplicity sequences $(d_j(f))_{j\in\ZZ}$ and $(d_j(\tilde f))_{j\in\ZZ}$ (Theorem \ref{theorem-inv}). Consequently $\simeq_\mathrm{inv}$ coincides with $\simeq_\mathrm{inj}$ on compact families and on noncompact families with closed range, but for every $k$ the noncompact families with $\dim\ker T_f=k$ and nonclosed range split into exactly $2^{\aleph_0}$ classes of $\simeq_\mathrm{inv}$, although they form a single class of $\simeq_\mathrm{bb}$ (Corollary \ref{corollary-inv}). The deepest ingredient is the window lemma (Lemma \ref{lemma-window}), which converts comparability of quadratic forms through an invertible operator into comparability of multiplicities; it holds for Fredholm operators up to an additive constant, and this version is used in Section \ref{section-further}. The strategy is described at the beginning of Section \ref{section-inv}.\par\bigskip
\noindent\emph{An operator distance} (Section \ref{section-distance}). Call a pair $(w,\tilde w)$ exact if $w^*S_fw=S_{\tilde f}$ and $\tilde w^*S_{\tilde f}\tilde w=S_f$, and put $D(f,\tilde f)=\log\inf\max(1,\|w\|\|\tilde w\|)$ over exact pairs. Then $D$ is an extended pseudometric with $D<\infty$ exactly on $\simeq$, and, when $S_f\ne0\ne S_{\tilde f}$, $D(f,\tilde f)=\log\bigl(\sup_ns_n(f)/s_n(\tilde f)\cdot\sup_ns_n(\tilde f)/s_n(f)\bigr)$; the infimum is attained for compact families (Proposition \ref{proposition-distance-metric}, Theorem \ref{theorem-distance}).\par\bigskip
\noindent\emph{Further classes of comparison operators} (Section \ref{section-further}). Surjective and co-isometric comparison operators give nothing new: $\simeq_\mathrm{sur}=\simeq_\mathrm{coi}=\simeq_\mathrm{inv}$ (Theorem \ref{theorem-sur}). Upper semi-Fredholm comparison operators allow the kernel to change by a finite amount and give exactly three classes of noncompact families (Theorem \ref{theorem-usF}, Corollary \ref{corollary-usF}). Fredholm comparison operators give a relation with $\simeq_\mathrm{inv}\subseteq\simeq_\mathrm{F}\subseteq\simeq_\mathrm{usF}$, equal to $\simeq_\mathrm{usF}$ on pairs of compact families and on pairs of upper semi-Fredholm families (Proposition \ref{proposition-fredholm}); but, unlike $\simeq_\mathrm{usF}$, it still sees the multiplicities of the spectrum: a Fredholm comparison operator can change the window multiplicities only by an additive constant (Theorem \ref{theorem-fredholm}), so that $\simeq_\mathrm{F}$ is strictly finer than $\simeq_\mathrm{usF}$ and has $2^{\aleph_0}$ classes of noncompact families, in each kernel dimension (Corollary \ref{corollary-fredholm}). Dense range comparison operators preserve the closedness of the range and coincide with $\simeq_\mathrm{inj}$ on families with closed range (Proposition \ref{proposition-dense}, Corollary \ref{corollary-closed-range}).\par\bigskip
\noindent\emph{Orthonormal bases} (Section \ref{section-onb}). For each relation we determine the families equivalent to an orthonormal basis $e$ of $\tilde H$: $f\simeq e$ iff $\rk T_f=\dim\tilde H$ when $\dim\tilde H<\infty$, and iff $T_f$ is noncompact when $\dim\tilde H=\infty$ (a compact family of infinite rank is not equivalent to $e$); $f\simeq_\mathrm{inj}e$ iff moreover $f$ is complete; $f\simeq_\mathrm{usF}e$ iff $f\simeq_\mathrm{F}e$ iff $f\simeq e$ and $T_f$ is upper semi-Fredholm; and for $\simeq_\mathrm{bb},\simeq_\mathrm{iso},\simeq_\mathrm{inv},\simeq_\mathrm{uni},\simeq_\mathrm{sur},\simeq_\mathrm{coi},\simeq_\mathrm{dr}$ iff $f$ is a frame for $H$ and $\dim H=\dim\tilde H$ (Theorem \ref{theorem-onb}).\par\bigskip
\noindent\emph{Applications and comparison categories} (Sections \ref{section-applications} and \ref{section-complements}). Relative Paley--Wiener stability, invariance of Schatten membership and of $\int\|f\|^2$, and similarity invariance are proved in Section \ref{section-applications}. In Section \ref{section-categories} we show that $\simeq_\mathcal O$ is level $-1$ of the comparison datum of bounded seminorms (Proposition \ref{proposition-level-minus}); that level $0$ and the levels $k\ge1$ are trivial whenever $\mathcal O$ contains the operators from the zero space, as for $\mathcal O_\mathrm{all}$, $\mathcal O_\mathrm{inj}$, $\mathcal O_\mathrm{bb}$, $\mathcal O_\mathrm{iso}$ and $\mathcal O_\mathrm{usF}$; and that for invertible and unitary operators the comparison category is a groupoid, so that level $0$ equals level $-1$, level $1$ at $f$ is the group $\{w\in\mathrm{GL}(H):\rho(f)\circ w\asymp\rho(f)\}$, which is a parabolic subgroup in finite dimension, and the higher levels vanish (Proposition \ref{proposition-levels}). Open problems are formulated in Section \ref{section-problems}.\par\bigskip
\noindent\textbf{Table 1.} The relations, the comparison operators defining them, their complete invariants, and the number of classes of noncompact families. Here $\mathrm{sF}$ stands for upper semi-Fredholm.
\begin{center}
\footnotesize
\begin{tabular}{>{\raggedright\arraybackslash}p{2.7cm}>{\raggedright\arraybackslash}p{2.7cm}>{\raggedright\arraybackslash}p{4.5cm}>{\raggedright\arraybackslash}p{2.0cm}}
\hline
Relation & Comparison operators & Complete invariants & Noncompact classes \\ \hline
$\simeq$ & bounded & $s_n$ up to constants (Theorem \ref{theorem-singular}) & $1$ \\
$\simeq_\mathrm{inj}$ & injective & $s_n$ up to constants, $k$ (Theorem \ref{theorem-inj}) & $\aleph_0$ \\
$\simeq_\mathrm{bb}=\simeq_\mathrm{iso}$ & bounded below, isometric & as $\simeq_\mathrm{inj}$, and whether $T_f$ is sF (Theorem \ref{theorem-closed}) & $\aleph_0$ \\
$\simeq_\mathrm{usF}$ & upper semi-Fredholm & $s_n$ up to constants, finiteness of $k$, whether $T_f$ is sF (Theorem \ref{theorem-usF}) & $3$ \\
$\simeq_\mathrm{F}$ & Fredholm & partial: window multiplicities up to an additive constant (Theorem \ref{theorem-fredholm}) & $2^{\aleph_0}$ \\
$\simeq_\mathrm{inv}=\simeq_\mathrm{uni}$ $=\simeq_\mathrm{sur}=\simeq_\mathrm{coi}$ & invertible, unitary, surjective, co-isometric & $k$ and the multiplicities $d_j$ up to bounded displacement (Theorems \ref{theorem-inv}, \ref{theorem-sur}) & $2^{\aleph_0}$ \\
$\simeq_\mathrm{dr}$ & dense range & partial (Proposition \ref{proposition-dense}) & $\ge2$ \\ \hline
\end{tabular}
\end{center}
\noindent\textbf{Diagram of the inclusions.} Every relation is a relation on positive operators (Proposition \ref{proposition-reduction}). The inclusions between them, and the number of classes of noncompact families in square brackets, are the following; all inclusions are strict except possibly the one marked by a question mark (Problem \ref{problem-dr}). All rows end in the same relation $\simeq$.
\[
\begin{array}{ccccccccc}
\underset{[2^{\aleph_0}]}{\simeq_\mathrm{inv}} & \subsetneq & \underset{[\aleph_0]}{\simeq_\mathrm{bb}} & \subsetneq & \underset{[\aleph_0]}{\simeq_\mathrm{inj}} & \subsetneq & \underset{[1]}{\simeq} & & \\[2pt]
\rotatebox[origin=c]{270}{$\subsetneq$} & & \rotatebox[origin=c]{270}{$\subsetneq$} & & & & & & \\[2pt]
\underset{[2^{\aleph_0}]}{\simeq_\mathrm{F}} & \subsetneq & \underset{[3]}{\simeq_\mathrm{usF}} & & & & & & \\
\end{array}
\]
\noindent Here $\simeq_\mathrm{inv}\subsetneq\simeq_\mathrm{F}\subsetneq\simeq_\mathrm{usF}\subsetneq\simeq$ in the bottom row ($\simeq_\mathrm{F}\subsetneq\simeq_\mathrm{usF}$ is Corollary \ref{corollary-fredholm}), $\simeq_\mathrm{bb}\subsetneq\simeq_\mathrm{usF}$ by the kernel dimension, and
\[
\simeq_\mathrm{inv}\ \subseteq^{?}\ \simeq_\mathrm{dr}\ \subsetneq\ \simeq .
\]
The strictness of $\simeq_\mathrm{inv}\subsetneq\simeq_\mathrm{bb}$ is Corollary \ref{corollary-inv}(d), that of $\simeq_\mathrm{bb}\subsetneq\simeq_\mathrm{inj}$ is Corollary \ref{corollary-closed}, that of $\simeq_\mathrm{dr}\subsetneq\simeq$ is Proposition \ref{proposition-dense}(b), and $\simeq_\mathrm{inv}\subsetneq\simeq_\mathrm{F}$ follows from Lemma \ref{lemma-extension}. The relations $\simeq_\mathrm{inj}$ and $\simeq_\mathrm{usF}$ are incomparable, and so are $\simeq_\mathrm{F}$ and $\simeq_\mathrm{bb}$.\par\bigskip
\noindent\textbf{A running example.} The hierarchy is best understood on one pair of families, which will be traced through every relation. We return to it in each section.

\begin{example}\label{example-running}
On $H=\ell^2\oplus\ell^2$ let $f$ be a Bessel family with frame operator $S_f=S_1:=\mathrm{id}\oplus\operatorname{diag}(1/n)_{n\ge1}$, and on $\tilde H=\bigoplus_{j\ge0}\ell^2$ let $\tilde f$ have $S_{\tilde f}=S_2:=\bigoplus_{j\ge0}2^{-j}\,\mathrm{id}_{\ell^2}$ (such families exist by Lemma \ref{lemma-basic}(c)). Both operators are noncompact, injective, and not bounded below, since $0$ is an accumulation point of their spectra: neither family is a frame and neither is upper semi-Fredholm. Their dyadic multiplicities are
\[d_0(S_1)=\infty,\quad d_j(S_1)=2^j\ (j\ge1),\qquad d_j(S_2)=\infty\ (j\ge0),\]
since the eigenvalue $1/n$ lies in the band $j$ if and only if $2^j\le n<2^{j+1}$. The multiplicity of the spectrum is infinite in every band of $S_2$ but only in the band $0$ for $S_1$; this is the whole difference between the two families.
\begin{center}
\footnotesize
\begin{tabular}{>{\raggedright\arraybackslash}p{3.0cm}>{\raggedright\arraybackslash}p{1.3cm}>{\raggedright\arraybackslash}p{7.3cm}}
\hline
Relation & $f\sim\tilde f$? & Reason \\ \hline
$\simeq$ & yes & both noncompact (Corollary \ref{corollary-classes}(d)); $\lambda_n(S_1)=\lambda_n(S_2)=1$ for all $n$ \\
$\simeq_\mathrm{inj}$ & yes & noncompact and $k=0$ for both (Theorem \ref{theorem-inj}) \\
$\simeq_\mathrm{bb}=\simeq_\mathrm{iso}$ & yes & $\simeq_\mathrm{inj}$ and neither is upper semi-Fredholm (Theorem \ref{theorem-closed}); an isometry is built with the reservoir lemma (Lemma \ref{lemma-reservoir}) \\
$\simeq_\mathrm{usF}$ & yes & the same invariants (Theorem \ref{theorem-usF}) \\
$\simeq_\mathrm{F}$ & no & $W_{S_2}(2^{-j-1},2^{-j})=\infty$ for all $j\ge0$ whereas $W_{S_1}$ is finite on every window below $1$ (Corollary \ref{corollary-fredholm}) \\
$\simeq_\mathrm{inv}=\simeq_\mathrm{uni}$ $=\simeq_\mathrm{sur}=\simeq_\mathrm{coi}$ & no & $d_j(S_2)=\infty>d_{[j-m,j+m]}(S_1)$ for large $j$ (Corollary \ref{corollary-inv}(d)) \\
$\simeq_\mathrm{dr}$ & open & Problem \ref{problem-dr} \\ \hline
\end{tabular}
\end{center}
Moreover $D(f,\tilde f)=0$ although $f\not\simeq_\mathrm{inv}\tilde f$ (Theorem \ref{theorem-distance}): the distance does not see the dyadic multiplicities. Finally $f\simeq e$ and $f\simeq_\mathrm{inj}e$ for every orthonormal basis $e$ of an infinite-dimensional space, while $f\not\simeq_\mathrm{bb}e$ because $f$ is not a frame (Proposition \ref{proposition-frames}); the same holds for $\tilde f$.
\end{example}\par\bigskip
\noindent\textbf{Plan.} Section \ref{section-prelim} collects the preliminaries: spectral measures, spectral invariants of positive operators, Bessel families and the comparison relations. Section \ref{section-all} proves the classification for bounded comparison operators. Section \ref{section-inj} treats injective operators. Section \ref{section-closed} treats bounded below operators and isometries. Section \ref{section-inv} treats invertible and unitary operators. Section \ref{section-distance} treats the operator distance. Section \ref{section-further} introduces further classes of comparison operators, including the Fredholm ones. Section \ref{section-onb} treats the orthonormal bases. Section \ref{section-applications} contains the applications. Section \ref{section-complements} contains the relation with comparison categories, an extension to operators between Banach spaces, a noncompatibility example, a remark on operator ideals and open problems.

\section{Preliminaries}
\label{section-prelim}

\subsection{Spectral measures}
\label{section-spectral}

\begin{definition}\label{definition-spectral-measure}
Let $H$ be a Hilbert space. A \textbf{spectral measure} (or projection-valued measure) on $\RR $ is a map $E$ from the Borel subsets of $\RR $ to the orthogonal projections of $H$ such that $E(\RR )=1$, $E(B\cap B')=E(B)E(B')$ for all Borel sets $B,B'$, and, for all $x,y\in H$, the map $B\mapsto\langle E(B)x,y\rangle$ is a countably additive complex measure. In particular $E(\emptyset)=0$, $E(B)E(B')=0$ for disjoint $B,B'$, and $E(B\cup B')=E(B)+E(B')$ for disjoint $B,B'$. For $x\in H$ the \textbf{scalar spectral measure} of $x$ is the positive measure $\mu_x(B)=\langle E(B)x,x\rangle=\|E(B)x\|^2$, of total mass $\mu_x(\RR )=\|x\|^2$. The range $E(B)H$ is called the \textbf{spectral subspace} of $B$.
\end{definition}

We use the spectral theorem for bounded self-adjoint operators in the following form, see e.g.\ \cite{ReedSimon1980}.

\begin{theorem}\label{theorem-spectral}
For every bounded self-adjoint operator $S$ on $H$ there is a unique spectral measure $E_S$ on $\RR $ with $E_S(\RR \setminus\sigma(S))=0$ and
\[\langle Sx,y\rangle=\int_{\RR }\lambda\,d\langle E_S(\lambda)x,y\rangle\qquad(x,y\in H),\]
where $\sigma(S)$ is the spectrum of $S$. For a bounded Borel function $\varphi$ one has $\varphi(S)=\int\varphi\,dE_S$, in particular $E_S(B)=\mathbf 1_B(S)$, and every bounded operator commuting with $S$ commutes with all $E_S(B)$. If $S\ge0$ then $E_S$ is carried by $[0,\|S\|]$, and $\langle Sx,x\rangle=\int\lambda\,d\mu_x(\lambda)$.
\end{theorem}

\noindent Here $S\ge0$ means $\langle Sx,x\rangle\ge0$ for all $x$, and $A\le B$ means $B-A\ge0$ (the Loewner order). A bounded self-adjoint operator is determined by its quadratic form $x\mapsto\langle Sx,x\rangle$ (by polarization over $\CC $, and by symmetry over $\RR $). Throughout, $E=E_S$ denotes the spectral measure of a positive bounded operator $S$.

We record the basic properties of spectral subspaces.

\begin{lemma}\label{lemma-spectral-subspaces}
Let $S\ge0$ be bounded and $E=E_S$.
\begin{enumerate}[(a)]
\item Each spectral subspace $E(B)H$ reduces $S$ (that is, $SE(B)=E(B)S$), and spectral subspaces of disjoint Borel sets are orthogonal. If $B_1,B_2,\dots$ is a countable partition of $[0,\infty)$ into Borel sets then $H=\bigoplus_iE(B_i)H$.
\item $E(\{0\})H=\ker S$ and $E(0,\infty)H=(\ker S)^\perp$.
\item If $B\subseteq[0,\mu]$ then $\langle Sx,x\rangle\le\mu\|x\|^2$ for all $x\in E(B)H$.
\item If $B\subseteq(\lambda,\infty)$ then $\langle Sx,x\rangle>\lambda\|x\|^2$ for all nonzero $x\in E(B)H$.
\end{enumerate}
\end{lemma}
\begin{proof}
(a) $SE(B)=E(B)S$ is the commutation statement of Theorem \ref{theorem-spectral} applied to $S$ itself; orthogonality follows from $E(B)E(B')=0$, and the decomposition from $\sum_iE(B_i)=E([0,\infty))=1$ (strongly), which holds because $E$ is carried by $[0,\|S\|]$.
(b) If $Sx=0$ then $\int\lambda\,d\mu_x=\langle Sx,x\rangle=0$ with $\lambda\ge0$ $\mu_x$-a.e., so $\mu_x$ is carried by $\{0\}$ and $\|E(\RR \setminus\{0\})x\|^2=\mu_x(\RR \setminus\{0\})=0$, hence $x=E(\{0\})x$. Conversely, for $x\in E(\{0\})H$ one has $Sx=SE(\{0\})x=\int\lambda\mathbf 1_{\{0\}}(\lambda)\,dE=0$. The second equality follows from $E(0,\infty)=1-E(\{0\})$ (since $E$ is carried by $[0,\infty)$) and from $\ker S$ being closed.
(c)--(d) Let $x\in E(B)H$. Then $\mu_x(\RR \setminus B)=\|E(\RR \setminus B)E(B)x\|^2=0$, so $\mu_x$ is carried by $B$ and $\langle Sx,x\rangle=\int_B\lambda\,d\mu_x(\lambda)$. If $B\subseteq[0,\mu]$ this is at most $\mu\,\mu_x(\RR )=\mu\|x\|^2$. If $B\subseteq(\lambda,\infty)$ then $\lambda'-\lambda>0$ on $B$, so $\int(\lambda'-\lambda)\,d\mu_x(\lambda')>0$ as soon as $\mu_x\ne0$, that is as soon as $x\ne0$; hence $\langle Sx,x\rangle-\lambda\|x\|^2>0$.
\end{proof}

\subsection{Spectral invariants of positive operators}
\label{section-invariants}

We now introduce the spectral invariants of a positive operator that will be used throughout: counting functions, multiplicities and eigenvalues.

\begin{definition}\label{definition-invariants}
Let $S\ge0$ be a bounded operator on a separable Hilbert space $H$, and $E=E_S$. All dimensions below lie in $\{0,1,2,\dots,\infty\}$.
\begin{itemize}
\item The \textbf{counting function} is $N_S(\lambda)=\dim E(\lambda,\infty)H$ for $\lambda>0$. It is nonincreasing in $\lambda$, and $N_S(\lambda)=0$ for $\lambda\ge\|S\|$.
\item The \textbf{small-spectrum count} is $M_S(\mu)=\dim E[0,\mu]H$ for $\mu>0$; it is nondecreasing in $\mu$.
\item The \textbf{window count} is $W_S(\alpha,\beta)=\dim E(\alpha,\beta]H$ for $0<\alpha<\beta$.
\item The \textbf{kernel dimension} is $k(S)=\dim\ker S=\dim E\{0\}H$.
\item The \textbf{dyadic bands} are $H_j(S)=E(2^{-j-1},2^{-j}]H$ and $d_j(S)=\dim H_j(S)$, for $j\in\ZZ $. For integers $a\le b$ put $d_{[a,b]}(S)=\sum_{j=a}^bd_j(S)=W_S(2^{-b-1},2^{-a})$, the dimension of $E(2^{-b-1},2^{-a}]H$; since $S$ is bounded, $d_j(S)=0$ for all sufficiently negative $j$.
\item The \textbf{eigenvalues counted with multiplicity} are $\lambda_n(S)=\sup\{\lambda>0:N_S(\lambda)\ge n\}$ for $n\ge1$, with $\sup\emptyset=0$.
\end{itemize}
\end{definition}

\begin{example}\label{example-counting}
If $S$ is diagonal with eigenvalues $s_1\ge s_2\ge\dots$ in an orthonormal basis, then $N_S(\lambda)=\#\{n:s_n>\lambda\}$. If $S=\mathrm{id}$ on an infinite-dimensional space then $N_S(\lambda)=\infty$ for $\lambda<1$ and $N_S(\lambda)=0$ for $\lambda\ge1$. If $S$ has finite rank $r$ then $N_S(\lambda)\le r$ for all $\lambda>0$, with equality for $\lambda$ small. Comparing counting functions after a rescaling $\lambda\mapsto c\lambda$ is thus comparing the spectra of $S$ and $\tilde S$ up to constant factors, once their eigenvalues are matched in decreasing order.
\end{example}

The dyadic bands decompose the space orthogonally and control the quadratic form on each piece.

\begin{lemma}\label{lemma-bands}
Let $S\ge0$ be bounded. Then $H=\ker S\oplus\bigoplus_{j\in\ZZ }H_j(S)$ (orthogonal sum), every $H_j(S)$ reduces $S$, and for $0\ne x\in H_j(S)$
\[2^{-j-1}\|x\|^2<\langle Sx,x\rangle\le2^{-j}\|x\|^2 .\]
Moreover $E(2^{-n-1},\infty)H=\bigoplus_{j\le n}H_j(S)$ and $E[0,2^{-n}]H=\ker S\oplus\bigoplus_{j\ge n}H_j(S)$ for every $n\in\ZZ $, and $\rk S:=\dim(\ker S)^\perp=\sum_jd_j(S)$.
\end{lemma}
\begin{proof}
The sets $\{0\}$ and $(2^{-j-1},2^{-j}]$, $j\in\ZZ $, form a countable partition of $[0,\infty)$, so Lemma \ref{lemma-spectral-subspaces}(a),(b) gives the orthogonal decomposition and the reducing property. The two inequalities are Lemma \ref{lemma-spectral-subspaces}(c) with $B\subseteq[0,2^{-j}]$ and (d) with $B\subseteq(2^{-j-1},\infty)$. The unions $\bigcup_{j\le n}(2^{-j-1},2^{-j}]=(2^{-n-1},\infty)$ and $\{0\}\cup\bigcup_{j\ge n}(2^{-j-1},2^{-j}]=[0,2^{-n}]$ give the two identities. The last formula holds because $(\ker S)^\perp=\bigoplus_jH_j(S)$ and dimensions of orthogonal sums add.
\end{proof}

The eigenvalues counted with multiplicity are related to the counting function as follows.

\begin{lemma}\label{lemma-lambda}
Let $S\ge0$ be bounded on a separable space.
\begin{enumerate}[(a)]
\item For every finite $n\ge1$ and $\lambda>0$: $N_S(\lambda)\ge n\iff\lambda_n(S)>\lambda$.
\item $\lambda_1(S)\ge\lambda_2(S)\ge\dots\ge0$ and $N_S(\lambda)=\#\{n\ge1:\lambda_n(S)>\lambda\}$ for all $\lambda>0$.
\item $\rk S=\#\{n:\lambda_n(S)>0\}$.
\end{enumerate}
\end{lemma}
\begin{proof}
(a) ($\Leftarrow$) If $\lambda_n(S)>\lambda$ there is $\lambda''>\lambda$ with $N_S(\lambda'')\ge n$, and $N_S(\lambda)\ge N_S(\lambda'')\ge n$ since $N_S$ is nonincreasing. ($\Rightarrow$) Let $V\subseteq E(\lambda,\infty)H$ have dimension $n$. The function $x\mapsto\langle Sx,x\rangle$ is continuous and, by Lemma \ref{lemma-spectral-subspaces}(d), strictly larger than $\lambda$ on the unit sphere of $V$, which is compact; so its minimum there is some $\lambda''>\lambda$. Let $\mu<\lambda''$. If $\dim E(\mu,\infty)H<n$, the orthogonal projection of $V$ onto $E(\mu,\infty)H$ has a nonzero kernel element $y\in V$; then $y\in E[0,\mu]H$ and $\langle Sy,y\rangle\le\mu\|y\|^2<\lambda''\|y\|^2$, a contradiction. So $N_S(\mu)\ge n$ for all $\mu<\lambda''$, which gives $\lambda_n(S)\ge\lambda''>\lambda$.
(b) By (a), $\{\lambda:N_S(\lambda)\ge n+1\}\subseteq\{\lambda:N_S(\lambda)\ge n\}$, so $\lambda_{n+1}\le\lambda_n$. For fixed $\lambda$, (a) says that $\{n:\lambda_n>\lambda\}=\{n\text{ finite}:N_S(\lambda)\ge n\}$, which has $N_S(\lambda)$ elements.
(c) The subspaces $E(1/p,\infty)H$ increase with $p$ and their union is dense in $E(0,\infty)H=(\ker S)^\perp$ (as $E(1/p,\infty)\uparrow E(0,\infty)$ strongly). Hence $\rk S=\sup_pN_S(1/p)$, and by (b) this is $\sup_p\#\{n:\lambda_n>1/p\}=\#\{n:\lambda_n>0\}$.
\end{proof}

For compact operators, the counting function and the eigenvalues have their usual meaning.

\begin{lemma}\label{lemma-compact}
For a bounded operator $S\ge0$ on a separable space the following are equivalent: (i) $S$ is compact; (ii) $N_S(\lambda)<\infty$ for every $\lambda>0$; (iii) $\lambda_n(S)\to0$. In this case there is an orthonormal basis $(x_n)_{1\le n\le r}$ of $(\ker S)^\perp$, $r=\rk S$, with $Sx_n=\lambda_n(S)x_n$.
\end{lemma}
\begin{proof}
(ii)$\iff$(iii) holds by Lemma \ref{lemma-lambda}(b), since $(\lambda_n)$ is nonincreasing. (i)$\Rightarrow$(ii): if $N_S(\lambda)=\infty$ then $V=E(\lambda,\infty)H$ is infinite-dimensional and reduces $S$, and $S|_V$ satisfies $\|Sx\|\ge\lambda\|x\|$ (by $\langle Sx,x\rangle\ge\lambda\|x\|^2$), so $S|_V$ is compact and bounded below on an infinite-dimensional space, which is impossible (the unit ball of $V$ would be relatively compact). (ii)$\Rightarrow$(i): $S\,E(\lambda,\infty)$ has finite rank and $\|S-S\,E(\lambda,\infty)\|=\|S\,E[0,\lambda]\|\le\lambda$, so $S$ is a norm limit of finite-rank operators. Finally, for compact $S$ the spectral theorem for compact self-adjoint operators provides an orthonormal basis of $(\ker S)^\perp$ of eigenvectors $x_n$ with eigenvalues $\mu_1\ge\mu_2\ge\dots>0$ (with multiplicity); then $E(\lambda,\infty)H=\operatorname{span}\{x_n:\mu_n>\lambda\}$, so $N_S(\lambda)=\#\{n:\mu_n>\lambda\}$ and, by Lemma \ref{lemma-lambda}(a), $\lambda_n(S)=\mu_n$.
\end{proof}

\subsection{Bessel families}
\label{section-bessel}

We recall the definitions of Bessel families and frames, in the generality of measure spaces.

\begin{definition}\label{definition-bessel}
Let $H$ be a separable Hilbert space and $(\Omega,\mu)$ a measure space. A map $f:\Omega\to H$ is \textbf{weakly measurable} if $\omega\mapsto\langle x,f(\omega)\rangle$ is measurable for every $x\in H$. It is a \textbf{Bessel family}, $f\in\mathcal B(\Omega,H)$, if the \textbf{analysis operator}
\[T_f:H\to L^2(\Omega,\mu),\qquad T_fx=(\langle x,f(\omega)\rangle)_{\omega\in\Omega}\]
is a well-defined bounded operator. For $\Omega=\NN $ with counting measure one recovers the Bessel sequences: $\sum_n|\langle x,f_n\rangle|^2\le B\|x\|^2$. The \textbf{frame operator} is $S_f=T_f^*T_f$, the \textbf{seminorm} of $f$ is $\rho(f)(x)=\|T_fx\|$, and $f$ is a \textbf{frame} if $T_f$ is bounded below, that is, $A\|x\|^2\le\int_\Omega|\langle x,f(\omega)\rangle|^2d\mu\le B\|x\|^2$ for some $0<A\le B$. We write $k(f)=\dim\ker T_f$, $N_f=N_{S_f}$, $d_j(f)=d_j(S_f)$, $\lambda_n(f)=\lambda_n(S_f)$ and $s_n(f)=\lambda_n(f)^{1/2}$, the \textbf{singular numbers} of $T_f$. We say that $f$ is \textbf{compact}, of \textbf{rank} $r$, \textbf{upper semi-Fredholm}, if $T_f$ is.
\end{definition}

We collect the basic properties of Bessel families.

\begin{lemma}\label{lemma-basic}
Let $f\in\mathcal B(\Omega,H)$.
\begin{enumerate}[(a)]
\item $\rho(f)$ is a seminorm with $\rho(f)(x)^2=\langle S_fx,x\rangle=\int_\Omega|\langle x,f(\omega)\rangle|^2d\mu(\omega)$ and $\rho(f)=\|S_f^{1/2}\,\cdot\|$; it is bounded by $\|T_f\|\,\|\cdot\|$ and vanishes exactly on $\ker T_f=\ker S_f$.
\item $T_f$ is compact iff $S_f$ is compact; $\rk T_f=\rk S_f$ and $s_n(f)$ are the singular numbers of $T_f$ in the usual sense, so that $\sum_ns_n(f)^p=\|T_f\|_{p}^p$ is the $p$-th Schatten norm to the power $p$.
\item Every bounded positive operator $S$ on $H$ is the frame operator of a Bessel family: for an orthonormal basis $(e_i)_{i\in I}$ of $H$ ($I$ countable, with counting measure) the family $f(i)=S^{1/2}e_i$ satisfies $S_f=S$.
\end{enumerate}
\end{lemma}
\begin{proof}
(a) Since $\langle S_fx,x\rangle=\|T_fx\|^2=\int_\Omega|\langle x,f(\omega)\rangle|^2d\mu(\omega)$, the function $\rho(f)$ is the $L^2$-norm composed with the linear map $T_f$, hence a seminorm; it is bounded by $\|T_f\|$, equals $\|S_f^{1/2}\cdot\|$, and vanishes iff $T_fx=0$ iff $S_f^{1/2}x=0$ iff $S_fx=0$ (that is, $x$ is orthogonal to $f(\omega)$ for almost every $\omega$).
(b) If $T_f$ is compact so is $S_f=T_f^*T_f$. Conversely, if $S_f$ is compact and $x_n\to0$ weakly then $\|T_fx_n\|^2=\langle S_fx_n,x_n\rangle\to0$ (because $S_fx_n\to0$ in norm and $(x_n)$ is bounded), so $T_f$ is compact. The singular numbers of $T_f$ are the square roots of the eigenvalues of $T_f^*T_f$, which is the definition used here (Lemma \ref{lemma-compact} gives the spectral decomposition in the compact case).
(c) $\langle x,f(i)\rangle=\langle S^{1/2}x,e_i\rangle$, so by Parseval $\|T_fx\|^2=\sum_i|\langle S^{1/2}x,e_i\rangle|^2=\|S^{1/2}x\|^2=\langle Sx,x\rangle$; hence $T_f$ is bounded and $S_f=S$.
\end{proof}

\subsection{The comparison relations}
\label{section-relations}

We now define the classes of comparison operators.

\begin{definition}\label{definition-classes}
Let $H,\tilde H$ be separable Hilbert spaces. A bounded operator $w:\tilde H\to H$ is called
\begin{itemize}
\item \textbf{injective} if $\ker w=0$;
\item \textbf{bounded below} if $\|wy\|\ge\beta\|y\|$ for some $\beta>0$ and all $y$ (equivalently, $w$ is injective with closed range, by the open mapping theorem);
\item an \textbf{isometry} if $\|wy\|=\|y\|$ for all $y$;
\item \textbf{invertible} if $w$ is bijective (then $w^{-1}$ is bounded);
\item \textbf{unitary} if $w$ is a surjective isometry.
\end{itemize}
We write $\mathcal O_\mathrm{all}\supseteq\mathcal O_\mathrm{inj}\supseteq\mathcal O_\mathrm{bb}\supseteq\mathcal O_\mathrm{iso}$ for the classes of all bounded, injective, bounded below and isometric operators, and $\mathcal O_\mathrm{bb}\supseteq\mathcal O_\mathrm{inv}\supseteq\mathcal O_\mathrm{uni}$ for the invertible and unitary operators ($\mathcal O_\mathrm{uni}\subseteq\mathcal O_\mathrm{iso}$ too). Each class contains the identity operators and is closed under composition.
\end{definition}

\begin{remark}\label{remark-cl}
The subscript ``bb'' stands for \emph{bounded below}: by the open mapping theorem the bounded below operators are exactly the injective operators with closed range, that is, the isomorphisms onto closed subspaces. The class $\mathcal O_\mathrm{bb}$ is closed under composition, whereas the class of all operators with closed range is not, which is why injectivity is built in.
\end{remark}

These classes give rise to the comparison relations.

\begin{definition}\label{definition-relation}
For seminorms $\phi,\psi$ on a space write $\phi\asymp\psi$ if $c\psi\le\phi\le C\psi$ for some constants $0<c\le C<\infty$. Let $\mathcal O$ be one of the six classes above, or more generally any class of bounded operators between Hilbert spaces which contains the identities and is closed under composition (such classes are studied in Section \ref{section-further}). For Bessel families $f\in\mathcal B(\Omega,H)$ and $\tilde f\in\mathcal B(\tilde\Omega,\tilde H)$ we write $f\simeq_{\mathcal O}\tilde f$ if there are operators $w:\tilde H\to H$ and $\tilde w:H\to\tilde H$ in $\mathcal O$ with
\[\rho(f)\circ w\asymp\rho(\tilde f)\qquad\text{and}\qquad\rho(\tilde f)\circ\tilde w\asymp\rho(f).\]
Here $\rho(f)\circ w$ is the seminorm $y\mapsto\rho(f)(wy)=\|T_fwy\|$ on $\tilde H$. We write $\simeq=\simeq_\mathrm{all}$, and $\simeq_\mathrm{inj},\simeq_\mathrm{bb},\simeq_\mathrm{iso},\simeq_\mathrm{inv},\simeq_\mathrm{uni}$ for the other relations. Only the spaces $H,\tilde H$ and the operators $S_f,S_{\tilde f}$ matter; the index spaces $\Omega,\tilde\Omega$ are unrelated.
\end{definition}

The comparison relations depend only on the frame operators, and they are equivalence relations.

\begin{proposition}\label{proposition-reduction}
Let $\mathcal O$ be one of the six classes, or any class of bounded operators containing the identities and closed under composition.
\begin{enumerate}[(a)]
\item $f\simeq_{\mathcal O}\tilde f$ iff there are $w:\tilde H\to H$ and $\tilde w:H\to\tilde H$ in $\mathcal O$ and constants $0<c\le C$ with
\[c\,S_{\tilde f}\le w^*S_fw\le C\,S_{\tilde f}\qquad\text{and}\qquad c\,S_f\le\tilde w^*S_{\tilde f}\tilde w\le C\,S_f.\]
In particular $\simeq_{\mathcal O}$ depends only on the frame operators; every bounded positive operator occurs as a frame operator (Lemma \ref{lemma-basic}(c)), so $\simeq_{\mathcal O}$ is a relation on bounded positive operators on separable spaces.
\item $\simeq_{\mathcal O}$ is an equivalence relation.
\end{enumerate}
\end{proposition}
\begin{proof}
(a) Since $\rho(f)(wy)^2=\langle S_fwy,wy\rangle=\langle w^*S_fwy,y\rangle$ and $\rho(\tilde f)(y)^2=\langle S_{\tilde f}y,y\rangle$, the condition $c_1\rho(\tilde f)\le\rho(f)\circ w\le C_1\rho(\tilde f)$ is equivalent to $c_1^2S_{\tilde f}\le w^*S_fw\le C_1^2S_{\tilde f}$. Put $c=c_1^2$, $C=C_1^2$, taking the minimum of the two lower constants and the maximum of the two upper constants over the two comparisons.
(b) Reflexivity: take $w=\tilde w=\mathrm{id}$. Symmetry: the definition is symmetric in $(f,w)$ and $(\tilde f,\tilde w)$. Transitivity: let $w,\tilde w$ realize $f\simeq_{\mathcal O}\tilde f$ and $v,\tilde v$ realize $\tilde f\simeq_{\mathcal O}\hat f$, with $v:\hat H\to\tilde H$ and $\tilde v:\tilde H\to\hat H$. Pulling back along $v$ preserves $\asymp$, so $\rho(f)\circ(wv)=(\rho(f)\circ w)\circ v\asymp\rho(\tilde f)\circ v\asymp\rho(\hat f)$, and likewise $\rho(\hat f)\circ(\tilde v\tilde w)\asymp\rho(\tilde f)\circ\tilde w\asymp\rho(f)$. Since $\mathcal O$ is closed under composition, $wv$ and $\tilde v\tilde w$ lie in $\mathcal O$, so $f\simeq_{\mathcal O}\hat f$.
\end{proof}

The inclusions between the classes of comparison operators give inclusions between the relations.

\begin{proposition}\label{proposition-inclusions}
${\simeq_\mathrm{uni}}\subseteq{\simeq_\mathrm{inv}}\subseteq{\simeq_\mathrm{bb}}\subseteq{\simeq_\mathrm{inj}}\subseteq{\simeq}$ and ${\simeq_\mathrm{uni}}\subseteq{\simeq_\mathrm{iso}}\subseteq{\simeq_\mathrm{bb}}$.
\end{proposition}
\begin{proof}
If $\mathcal O\subseteq\mathcal O'$, any pair $(w,\tilde w)$ realizing $\simeq_{\mathcal O}$ also realizes $\simeq_{\mathcal O'}$. The inclusions of classes are clear from Definition \ref{definition-classes}: unitary operators are invertible isometries, isometries and invertible operators are bounded below, and bounded below operators are injective.
\end{proof}

\section{Arbitrary comparison operators: counting functions and singular numbers}
\label{section-all}

\subsection{A necessary condition}

We first establish a necessary condition, in terms of counting functions, for the existence of a comparison operator.

\begin{proposition}\label{proposition-counting}
Let $S\ge0$ and $\tilde S\ge0$ be bounded operators on $H$ and $\tilde H$, and let $w:\tilde H\to H$ be bounded, $w\ne0$. If $K_1^2\tilde S\le w^*Sw$ for some $K_1>0$ (that is, $K_1\rho(\tilde f)\le\rho(f)\circ w$ for $S=S_f$, $\tilde S=S_{\tilde f}$), then
\[N_{\tilde S}(\lambda)\le N_S\bigl(K_1^2\lambda/\|w\|^2\bigr)\qquad\text{for all }\lambda>0.\]
\end{proposition}
\begin{proof}
Put $W=E_{\tilde S}(\lambda,\infty)\tilde H$, $\mu=K_1^2\lambda/\|w\|^2$ and $Q=E_S(\mu,\infty)$. For $0\ne y\in W$ one has $\langle\tilde Sy,y\rangle>\lambda\|y\|^2$ by Lemma \ref{lemma-spectral-subspaces}(d): the spectral measure of $y$ is a nonzero measure carried by $(\lambda,\infty)$, so the inequality is strict. For $x=wy$ we have $\|x\|\le\|w\|\|y\|$, hence
\[\langle Sx,x\rangle\ge K_1^2\langle\tilde Sy,y\rangle>K_1^2\lambda\|y\|^2\ge\mu\|x\|^2 .\]
In particular $x\ne0$. Moreover $Qx\ne0$: otherwise $x=(1-Q)x\in E_S[0,\mu]H$ (as $E_S$ is carried by $[0,\infty)$), and Lemma \ref{lemma-spectral-subspaces}(c) would give $\langle Sx,x\rangle\le\mu\|x\|^2$. So $Qw|_W:W\to E_S(\mu,\infty)H$ is injective and linear, and an injective linear map cannot increase dimension: $N_{\tilde S}(\lambda)=\dim W\le N_S(\mu)$.
\end{proof}

\subsection{Two operator lemmas}

We shall use the following factorization lemma of Douglas \cite{Douglas1966}; we include its short proof.

\begin{lemma}\label{lemma-douglas}
If $AA^*\le\kappa BB^*$ for bounded operators $A:Y\to X$, $B:Z\to X$ and $\kappa>0$, then $A=BC$ for some bounded $C:Y\to Z$ with $\|C\|^2\le\kappa$.
\end{lemma}
\begin{proof}
The hypothesis gives $\|A^*x\|^2=\langle AA^*x,x\rangle\le\kappa\langle BB^*x,x\rangle=\kappa\|B^*x\|^2$ for all $x\in X$. Hence $B^*x\mapsto A^*x$ is a well-defined linear map on $\ran B^*$ (if $B^*x=B^*x'$ then $\|A^*(x-x')\|\le\sqrt\kappa\|B^*(x-x')\|=0$) of norm at most $\sqrt\kappa$. Extend it by continuity to $\overline{\ran B^*}$ and by $0$ to the orthogonal complement: this gives $D:Z\to Y$ with $\|D\|\le\sqrt\kappa$ and $DB^*=A^*$. Taking adjoints, $A=BD^*$, and $C=D^*$ works.
\end{proof}

Recall the notion of partial isometry.

\begin{definition}\label{definition-partial-isometry}
A bounded operator $J:\tilde H\to H$ is a \textbf{partial isometry} if $\|Jy\|=\|y\|$ for all $y\in(\ker J)^\perp$; this space is the \textbf{initial space} of $J$. Then $J$ vanishes on $\ker J$ and maps $(\ker J)^\perp$ unitarily onto its (closed) range.
\end{definition}

The next lemma embeds the spectral bands of one operator into those of another; it is the heart of the sufficiency in Theorem \ref{theorem-all}. The idea is a greedy allocation: the vectors of the band $k$ of $\tilde S$ are sent isometrically into bands of $S$ of index at most $k+m$, where $S$ is larger, and the hypothesis says exactly that there is always enough room; the only care is needed with infinite-dimensional bands, where room must be kept in reserve.

\begin{lemma}\label{lemma-embedding}
Let $\tilde S\ge0$ and $S\ge0$ be bounded operators on separable spaces $\tilde H$ and $H$, let $q>1$ and $m\in\ZZ$, and assume $N_{\tilde S}(\lambda)\le N_S(q^{-m}\lambda)$ for all $\lambda>0$. Then there is a partial isometry $J:\tilde H\to H$ with initial space $(\ker\tilde S)^\perp$ and $J\tilde SJ^*\le q^{m+1}S$.
\end{lemma}
\begin{proof}
\emph{Step 1: bands and the hypothesis.} For $j\in\ZZ$ let $H_j=E_S(q^{-j-1},q^{-j}]H$ and $\tilde H^k=E_{\tilde S}(q^{-k-1},q^{-k}]\tilde H$, with dimensions $d_j,\tilde d_k$. As in Lemma \ref{lemma-bands} (with $2$ replaced by $q$), $H=\ker S\oplus\bigoplus_jH_j$ and $\tilde H=\ker\tilde S\oplus\bigoplus_k\tilde H^k$ orthogonally, all bands reduce the operators, $q^{-j-1}\|x\|^2<\langle Sx,x\rangle\le q^{-j}\|x\|^2$ on $H_j$, and $\langle\tilde Sy,y\rangle\le q^{-k}\|y\|^2$ on $\tilde H^k$. Since $S,\tilde S$ are bounded, $d_j=\tilde d_j=0$ for all $j$ below some $j_0$. Put $G_n=\bigoplus_{j\le n}H_j=E_S(q^{-n-1},\infty)H$. Then $\dim G_n=N_S(q^{-n-1})$ and $\sum_{j\le k}\tilde d_j=N_{\tilde S}(q^{-k-1})$. The hypothesis at $\lambda=q^{-k-1}$ reads
\[\sum_{j\le k}\tilde d_j\le\dim G_{k+m}\qquad\text{for all }k\in\ZZ.\tag{$*$}\]
\emph{Step 2: allocation.} We choose, by induction on $k$, pairwise orthogonal subspaces $X_k\subseteq G_{k+m}$ with $\dim X_k=\tilde d_k$, together with decompositions $X_k=\bigoplus_{j\le k+m}X_{k,j}$, $X_{k,j}\subseteq H_j$. For $k<j_0$ put $X_k=0$. Let $k\ge j_0$ and assume $X_{k'}$ has been chosen for $k'<k$. Put $R_j=H_j\ominus\bigoplus_{k'<k}X_{k',j}$ (the part of $H_j$ not yet used) and assume the invariant
\[(I_k)\qquad H_j\text{ infinite-dimensional}\implies R_j\text{ infinite-dimensional, for every }j.\]
It holds for $k=j_0$ since then $R_j=H_j$. Let $\nu=\dim G_{k+m}=\sum_{j\le k+m}d_j$.

\emph{Case $\nu=\infty$.} A finite sum of finite dimensions is finite, so some band $H_{j^*}$ with $j^*\le k+m$ is infinite-dimensional, and $R_{j^*}$ is infinite-dimensional by $(I_k)$. Split $R_{j^*}=A\oplus B$ with $A,B$ infinite-dimensional, and choose $X_{k,j^*}\subseteq A$ of dimension $\tilde d_k$ (possible since $\tilde d_k\le\infty=\dim A$ in a separable space), and $X_{k,j}=0$ for $j\ne j^*$. Then $R_{j^*}$ is replaced by a subspace containing $B$, so $(I_{k+1})$ holds.

\emph{Case $\nu<\infty$.} Then all bands $H_j$ with $j\le k+m$ are finite-dimensional, and all earlier $X_{k'}$ lie in $G_{k'+m}\subseteq G_{k+m}$, so the unused part $\bigoplus_{j\le k+m}R_j$ of $G_{k+m}$ has dimension $\nu-\sum_{k'<k}\tilde d_{k'}$, which is at least $\tilde d_k$ by $(*)$. Choose $X_{k,j}\subseteq R_j$ ($j\le k+m$) of total dimension $\tilde d_k$. Only finite-dimensional bands are touched, so $(I_{k+1})$ holds.

By construction the $X_k$ are pairwise orthogonal: $X_{k,j}\subseteq R_j$ is orthogonal to the earlier $X_{k',j}$, and different bands are orthogonal.

\emph{Step 3: the operator $J$.} Since $\dim\tilde H^k=\dim X_k$, there are unitaries $J_k:\tilde H^k\to X_k$. Put $J=\bigoplus_kJ_k$ on $\bigoplus_k\tilde H^k=(\ker\tilde S)^\perp$ and $J=0$ on $\ker\tilde S$. It is a partial isometry with initial space $(\ker\tilde S)^\perp$, and $J^*x=\sum_kJ_k^{-1}P_{X_k}x$.

\emph{Step 4: the estimate.} Let $x\in H$, $E_j=E_S(q^{-j-1},q^{-j}]$, and $\Pi_{k,j}$ the orthogonal projection onto $X_{k,j}$. As $\tilde S$ is bounded by $q^{-k}$ on $\tilde H^k$ and these spaces are mutually orthogonal and reduce $\tilde S$,
\[\langle J\tilde SJ^*x,x\rangle=\langle\tilde SJ^*x,J^*x\rangle=\sum_k\langle\tilde SJ_k^{-1}P_{X_k}x,J_k^{-1}P_{X_k}x\rangle\le\sum_kq^{-k}\|P_{X_k}x\|^2 .\]
Now $\|P_{X_k}x\|^2=\sum_{j\le k+m}\|\Pi_{k,j}E_jx\|^2$, because $X_{k,j}\subseteq H_j$, and for $j\le k+m$ we have $q^{-k}\le q^{m-j}=q^{m+1}q^{-j-1}$. Exchanging the sums, and using that for fixed $j$ the spaces $X_{k,j}$ ($k\in\ZZ$) are pairwise orthogonal subspaces of $H_j$, so $\sum_k\|\Pi_{k,j}E_jx\|^2\le\|E_jx\|^2$, we get
\[\langle J\tilde SJ^*x,x\rangle\le q^{m+1}\sum_jq^{-j-1}\|E_jx\|^2\le q^{m+1}\sum_j\langle SE_jx,E_jx\rangle\le q^{m+1}\langle Sx,x\rangle,\]
where we used $\langle Sz,z\rangle\ge q^{-j-1}\|z\|^2$ on $H_j$, and that the bands are orthogonal and reduce $S$ (so $\sum_j\langle SE_jx,E_jx\rangle\le\langle Sx,x\rangle$, with equality since $\ker S$ contributes $0$). Hence $J\tilde SJ^*\le q^{m+1}S$.
\end{proof}

\subsection{Classification by counting functions and singular numbers}

We now prove the classification for bounded comparison operators.

\begin{theorem}\label{theorem-all}
For Bessel families $f\in\mathcal B(\Omega,H)$, $\tilde f\in\mathcal B(\tilde\Omega,\tilde H)$ with $S=S_f$, $\tilde S=S_{\tilde f}$, the following are equivalent:
\begin{enumerate}[(i)]
\item $f\simeq\tilde f$;
\item there are bounded $w:\tilde H\to H$ and $\tilde w:H\to\tilde H$ with $\rho(f)\circ w=\rho(\tilde f)$ and $\rho(\tilde f)\circ\tilde w=\rho(f)$, that is, $w^*Sw=\tilde S$ and $\tilde w^*\tilde S\tilde w=S$;
\item there is $c\in(0,1]$ with $N_{\tilde S}(\lambda)\le N_S(c\lambda)$ and $N_S(\lambda)\le N_{\tilde S}(c\lambda)$ for all $\lambda>0$.
\end{enumerate}
Moreover, if $N_{\tilde S}(\lambda)\le N_S(2^{-m}\lambda)$ for all $\lambda>0$ ($m\ge0$ an integer), then $w$ in \textup{(ii)} can be chosen with $\|w\|^2\le2^{m+1}$.
\end{theorem}
\begin{proof}
\emph{(i)$\Rightarrow$(iii).} Let $w,\tilde w$ be bounded with $c_1\rho(\tilde f)\le\rho(f)\circ w\le C_1\rho(\tilde f)$ and $\tilde c_1\rho(f)\le\rho(\tilde f)\circ\tilde w\le\tilde C_1\rho(f)$. If $w=0$ then $\rho(\tilde f)=0$, so $\tilde S=0$ and $N_{\tilde S}\equiv0\le N_S(c\lambda)$ for every $c$; otherwise Proposition \ref{proposition-counting} gives $N_{\tilde S}(\lambda)\le N_S(K_1^2\lambda/\|w\|^2)$ with $K_1=c_1$. The same holds for $\tilde w$. Take $c=\min\{1,c_1^2/\|w\|^2,\tilde c_1^2/\|\tilde w\|^2\}$ (omitting the terms with $w=0$ or $\tilde w=0$) and use that $N_S$ is nonincreasing.

\emph{(iii)$\Rightarrow$(ii).} Choose an integer $m\ge0$ with $2^{-m}\le c$. Since $N_S$ is nonincreasing, $N_{\tilde S}(\lambda)\le N_S(c\lambda)\le N_S(2^{-m}\lambda)$. Lemma \ref{lemma-embedding} with $q=2$ gives a partial isometry $J$ with initial space $(\ker\tilde S)^\perp$ and $J\tilde SJ^*\le2^{m+1}S$. Put $A=J\tilde S^{1/2}:\tilde H\to H$ and $B=S^{1/2}$. Then $AA^*=J\tilde SJ^*\le2^{m+1}BB^*$, so Lemma \ref{lemma-douglas} gives a bounded $w:\tilde H\to H$ with $S^{1/2}w=J\tilde S^{1/2}$ and $\|w\|^2\le2^{m+1}$. Since $\ran\tilde S^{1/2}\subseteq(\ker\tilde S)^\perp$ and $J$ is isometric there, $\rho(f)(wy)=\|S^{1/2}wy\|=\|J\tilde S^{1/2}y\|=\|\tilde S^{1/2}y\|=\rho(\tilde f)(y)$, that is $w^*Sw=\tilde S$ (a bounded self-adjoint operator is determined by its quadratic form). The same argument with the roles of $f,\tilde f$ exchanged gives $\tilde w$.

\emph{(ii)$\Rightarrow$(i).} Take $c=C=1$.
\end{proof}

Condition (iii) of Theorem \ref{theorem-all} translates into a statement about singular numbers.

\begin{theorem}\label{theorem-singular}
$f\simeq\tilde f$ if and only if there is $C\ge1$ with $C^{-1}s_n(\tilde f)\le s_n(f)\le C\,s_n(\tilde f)$ for all $n\ge1$.
\end{theorem}
\begin{proof}
Write $\lambda_n=\lambda_n(S)$, $\tilde\lambda_n=\lambda_n(\tilde S)$, so $s_n(f)=\lambda_n^{1/2}$.

\emph{($\Rightarrow$)} Let $c$ be as in Theorem \ref{theorem-all}(iii), and let $n$ be finite with $\tilde\lambda_n>0$. For $0<\lambda<\tilde\lambda_n$, Lemma \ref{lemma-lambda}(a) gives $N_{\tilde S}(\lambda)\ge n$, hence $N_S(c\lambda)\ge n$, hence $\lambda_n>c\lambda$. Letting $\lambda\uparrow\tilde\lambda_n$ gives $\lambda_n\ge c\tilde\lambda_n$ (trivially true if $\tilde\lambda_n=0$). By symmetry $\tilde\lambda_n\ge c\lambda_n$. Take square roots, with $C=c^{-1/2}$.

\emph{($\Leftarrow$)} We have $\tilde\lambda_n\le C^2\lambda_n$ for all $n$. If $N_{\tilde S}(\lambda)\ge n$ for a finite $n$, then $\tilde\lambda_n>\lambda$, so $\lambda_n\ge C^{-2}\tilde\lambda_n>C^{-2}\lambda$ and $N_S(C^{-2}\lambda)\ge n$ by Lemma \ref{lemma-lambda}(a). As the counts take values in $\{0,1,\dots,\infty\}$ (separability), this shows $N_{\tilde S}(\lambda)\le N_S(C^{-2}\lambda)$, and symmetrically. So Theorem \ref{theorem-all}(iii) holds with $c=C^{-2}$.
\end{proof}

We deduce the description of the classes of $\simeq$.

\begin{corollary}\label{corollary-classes}
Let $f,\tilde f$ be Bessel families.
\begin{enumerate}[(a)]
\item If $f\simeq\tilde f$ then $\rk T_f=\rk T_{\tilde f}$, $T_f$ is compact iff $T_{\tilde f}$ is, and $\sum_ns_n(f)^p<\infty$ iff $\sum_ns_n(\tilde f)^p<\infty$ for every $p>0$.
\item For each $r\in\{0,1,2,\dots\}$ the families of rank $r$ form exactly one class.
\item The classes of compact families of infinite rank are in bijection with the nonincreasing sequences of positive numbers tending to $0$, modulo term-by-term comparability (comparability up to a multiplicative constant). There are exactly $2^{\aleph_0}$ of them.
\item All families with noncompact $T_f$ form exactly one class. It contains the orthonormal bases of the infinite-dimensional separable spaces, all frames of such spaces, and also nonframes.
\end{enumerate}
\end{corollary}
\begin{proof}
(a) By Theorem \ref{theorem-singular} the sequences $(s_n(f))$ and $(s_n(\tilde f))$ are comparable. In particular $s_n(f)>0$ iff $s_n(\tilde f)>0$, so the ranks agree (Lemma \ref{lemma-lambda}(c)); $s_n(f)\to0$ iff $s_n(\tilde f)\to0$, so the compactness agrees (Lemmas \ref{lemma-compact} and \ref{lemma-basic}(b)); and the series $\sum_ns_n^p$ converge simultaneously (for the Schatten classes and $s$-numbers see \cite{GohbergKrein1969,Pietsch1980}).

(b) If both families have rank $r$, then $s_n>0$ exactly for $n\le r$ in both, and the finitely many ratios $s_n(f)/s_n(\tilde f)$ are bounded above and below, so $f\simeq\tilde f$ by Theorem \ref{theorem-singular}. Rank $r$ occurs: the diagonal operator $\operatorname{diag}(1,\dots,1,0,0,\dots)$ with $r$ ones is a frame operator by Lemma \ref{lemma-basic}(c) (on $\FF^r$ or on $\ell^2$).

(c) By Theorem \ref{theorem-singular}, the class of a compact family $f$ of infinite rank is determined by the sequence $(s_n(f))_{n\ge1}$ modulo comparability, and this sequence is positive, nonincreasing and tends to $0$ (Lemma \ref{lemma-compact}). Conversely every such sequence $(s_n)$ occurs, as the frame operator $\operatorname{diag}(s_n^2)$ on $\ell^2$ (Lemma \ref{lemma-basic}(c)). There are at most $2^{\aleph_0}$ sequences of reals; and the sequences $s_n=n^{-a}$, $a>0$, are pairwise noncomparable, since $n^{-a}/n^{-b}=n^{b-a}$ is unbounded in one of the two directions when $a\ne b$.

(d) If $T_f$ is noncompact, then $\lambda_n(S_f)\not\to0$ (Lemma \ref{lemma-compact}). As $(\lambda_n(S_f))_n$ is nonincreasing, $\lambda_\infty(f):=\inf_n\lambda_n(S_f)>0$, and $\lambda_n(S_f)\le\|S_f\|$. So for two noncompact families $f,\tilde f$ the ratios $\lambda_n(S_f)/\lambda_n(S_{\tilde f})$ lie in $[\lambda_\infty(f)/\|S_{\tilde f}\|,\|S_f\|/\lambda_\infty(\tilde f)]$, hence the singular numbers are comparable and $f\simeq\tilde f$. Noncompact operators live on infinite-dimensional spaces. An orthonormal basis $e$ has $S_e=\mathrm{id}$, which is noncompact; a frame is bounded below, hence noncompact on an infinite-dimensional space; and a family with $S_f=\mathrm{id}\oplus\operatorname{diag}(1/n)$ (Lemma \ref{lemma-basic}(c)) is noncompact but not a frame, since $T_f$ is not bounded below.
\end{proof}

\begin{example}[Running example, continued]\label{example-running-all}
For the families of Example \ref{example-running}, $\lambda_n(S_1)=\lambda_n(S_2)=1$ for all $n$, so $s_n(f)=s_n(\tilde f)=1$ and $f\simeq\tilde f$ by Theorem \ref{theorem-singular}. The singular numbers cannot see the difference between the two families: they only register that $1$ is an eigenvalue of infinite multiplicity, and not that the spectrum accumulates at $0$ in very different ways.
\end{example}

\section{Injective comparison operators}
\label{section-inj}

Injectivity of the comparison operators adds exactly the kernel dimension.

\begin{theorem}\label{theorem-inj}
$f\simeq_\mathrm{inj}\tilde f$ if and only if $f\simeq\tilde f$ and $k(f)=k(\tilde f)$.
\end{theorem}
\begin{proof}
Write $S=S_f$, $\tilde S=S_{\tilde f}$; recall $\ker T_f=\ker S$.

\emph{($\Rightarrow$)} The relation $\simeq$ holds by Proposition \ref{proposition-inclusions}. Let $w:\tilde H\to H$ be injective with $\rho(f)\circ w\asymp\rho(\tilde f)$. For $y\in\ker\tilde S$ we have $\rho(\tilde f)(y)=0$, hence $\rho(f)(wy)=0$ and $wy\in\ker S$. Thus $w$ maps $\ker\tilde S$ injectively into $\ker S$, so $k(\tilde f)\le k(f)$. Symmetrically $k(f)\le k(\tilde f)$, and since both lie in $\{0,1,\dots,\infty\}$, $k(f)=k(\tilde f)$.

\emph{($\Leftarrow$)} \emph{Step 1: an exact operator.} By Theorem \ref{theorem-all} there is a bounded $w_0:\tilde H\to H$ with $w_0^*Sw_0=\tilde S$. \emph{Step 2: remove the kernel of $S$.} Let $P$ be the orthogonal projection onto $(\ker S)^\perp$; since $S^{1/2}(1-P)=0$ we have $\rho(f)\circ P=\rho(f)$. For $0\ne y\in(\ker\tilde S)^\perp$ one has $\rho(f)(Pw_0y)=\rho(f)(w_0y)=\rho(\tilde f)(y)>0$, so $Pw_0y\ne0$. Hence $Pw_0$ is injective on $(\ker\tilde S)^\perp$, with values in $(\ker S)^\perp$. \emph{Step 3: match the kernels.} Since $\dim\ker\tilde S=\dim\ker S$, there is a unitary $\iota:\ker\tilde S\to\ker S$ (both are separable of the same dimension, finite or infinite). Define $w=Pw_0|_{(\ker\tilde S)^\perp}\oplus\iota$. It is bounded and injective, because its two components have values in the orthogonal spaces $(\ker S)^\perp$ and $\ker S$ and are injective. \emph{Step 4: check.} For $y=y_1+y_2$ with $y_1\in(\ker\tilde S)^\perp$, $y_2\in\ker\tilde S$ we have $\rho(f)(wy)=\rho(f)(Pw_0y_1)=\rho(\tilde f)(y_1)=\rho(\tilde f)(y)$, since $\rho(f)$ vanishes on $\ker S$ and $\rho(\tilde f)$ on $\ker\tilde S$. Exchanging the roles of $f,\tilde f$ gives $\tilde w$.
\end{proof}

Theorem \ref{theorem-inj} yields the following description of the classes of $\simeq_\mathrm{inj}$.

\begin{corollary}\label{corollary-inj}
The noncompact families split into exactly one $\simeq_\mathrm{inj}$-class for each $k\in\{0,1,2,\dots,\infty\}$. The compact families are classified by the pair (class of $(s_n)$ modulo comparability, $k$).
\end{corollary}
\begin{proof}
By Theorem \ref{theorem-inj}, $\simeq_\mathrm{inj}$ is the intersection of $\simeq$ with the equality of $k$. All noncompact families are $\simeq$-equivalent (Corollary \ref{corollary-classes}(d)), so on them the classes are the level sets of $k$. Every $k$ occurs, for instance for $S_f=\mathrm{id}_{\ell^2}\oplus0_{\FF^k}$ ($0_{\ell^2}$ for $k=\infty$), a frame operator by Lemma \ref{lemma-basic}(c). For compact families, the $\simeq$-classes are described in Corollary \ref{corollary-classes}.
\end{proof}

\begin{example}[Running example, continued]
For the families of Example \ref{example-running}, $k(f)=k(\tilde f)=0$, so $f\simeq_\mathrm{inj}\tilde f$.
\end{example}

\section{Comparison operators bounded below}
\label{section-closed}

\subsection{Semi-Fredholm families}

We begin with several characterizations of upper semi-Fredholm families.

\begin{lemma}\label{lemma-usF}
Let $f$ be a Bessel family and $S=S_f$. The following are equivalent:
\begin{enumerate}[(i)]
\item $f$ is upper semi-Fredholm, that is, $k(f)<\infty$ and $\ran T_f$ is closed;
\item $\ker S$ is finite-dimensional and $\ran S$ is closed;
\item $M_S(\mu)<\infty$ for some $\mu>0$;
\item there is $\delta>0$ with $M_S(\mu)=k(f)<\infty$ for all $0<\mu<\delta$.
\end{enumerate}
In this case $\delta P\le S\le\|S\|P$, where $P$ is the orthogonal projection onto $(\ker S)^\perp$ and $\delta>0$ is as in \textup{(iv)}.
\end{lemma}
\begin{proof}
(i)$\Leftrightarrow$(ii). $\ran T_f$ is closed iff $T_f$ is bounded below on $(\ker T_f)^\perp=(\ker S)^\perp$ (a standard consequence of the open mapping theorem), iff $\langle Sx,x\rangle=\|T_fx\|^2\ge\delta\|x\|^2$ on $(\ker S)^\perp$ for some $\delta>0$, iff $S|_{(\ker S)^\perp}$ is bounded below, iff $\ran S$ is closed (note that $(\ker S)^\perp$ reduces $S$ and $S|_{(\ker S)^\perp}$ is injective with the same range as $S$).

(ii)$\Rightarrow$(iv). By the proof of (i)$\Leftrightarrow$(ii) we have $S\ge\delta$ on $(\ker S)^\perp=E(0,\infty)H$ for some $\delta>0$. If $0\ne x\in E(0,\delta/2]H$, then $x\in(\ker S)^\perp$ and $\langle Sx,x\rangle\le\frac\delta2\|x\|^2<\delta\|x\|^2$ by Lemma \ref{lemma-spectral-subspaces}(c), a contradiction. Hence $E[0,\mu]H=E\{0\}H=\ker S$ for $0<\mu\le\delta/2$, that is, $M_S(\mu)=k(f)<\infty$ (with $\delta/2$ in place of $\delta$ in (iv)).

(iv)$\Rightarrow$(iii) is clear.

(iii)$\Rightarrow$(ii). If $M_S(\mu)<\infty$ then $E[0,\mu]H$ is finite-dimensional and reduces $S$, so $\ker S\subseteq E[0,\mu]H$ is finite-dimensional, and $S$ restricted to $E[0,\mu]H$ is a self-adjoint operator on a finite-dimensional space, with finitely many eigenvalues. Let $\delta>0$ be smaller than $\mu$ and than every positive eigenvalue of this restriction. Then $E(0,\delta]H=0$, so $S\ge\delta$ on $E(\delta,\infty)H=E(0,\infty)H=(\ker S)^\perp$, and $S|_{(\ker S)^\perp}$ is bounded below by $\delta$, which gives (ii).

The last assertion: $S\ge\delta$ on $(\ker S)^\perp$ and $S=0$ on $\ker S$ give $S\ge\delta P$, and $S\le\|S\|P$ is clear since $S=SP$.
\end{proof}

The next proposition gives a necessary condition for the existence of a bounded below comparison operator.

\begin{proposition}\label{proposition-necessary-cl}
Let $f,\tilde f$ be Bessel families and let $w:\tilde H\to H$ be bounded below, say $\|wy\|\ge\beta\|y\|$, with $\rho(f)\circ w\le C_1\rho(\tilde f)$. Then $w(\ker S_{\tilde f})\subseteq\ker S_f$, $k(\tilde f)\le k(f)$, $\dim\tilde H\le\dim H$, and for all $\mu>0$
\[M_{S_{\tilde f}}(\mu)\le M_{S_f}(C_1^2\beta^{-2}\mu).\]
In particular, if $f$ is upper semi-Fredholm then so is $\tilde f$.
\end{proposition}
\begin{proof}
Write $S=S_f$, $\tilde S=S_{\tilde f}$. If $y\in\ker\tilde S$ then $\rho(f)(wy)\le C_1\rho(\tilde f)(y)=0$, so $wy\in\ker S$; as $w$ is injective, $k(\tilde f)\le k(f)$, and $\dim\tilde H\le\dim H$. Let $\mu>0$ and $\nu=C_1^2\beta^{-2}\mu$. For $y\in E_{\tilde S}[0,\mu]\tilde H$ we have $\langle\tilde Sy,y\rangle\le\mu\|y\|^2$ (Lemma \ref{lemma-spectral-subspaces}(c)), so for $x=wy$
\[\langle Sx,x\rangle=\rho(f)(wy)^2\le C_1^2\langle\tilde Sy,y\rangle\le C_1^2\mu\|y\|^2\le C_1^2\mu\beta^{-2}\|x\|^2=\nu\|x\|^2 .\]
Let $V=wE_{\tilde S}[0,\mu]\tilde H$; since $w$ is injective, $\dim V=M_{\tilde S}(\mu)$, and $\langle Sx,x\rangle\le\nu\|x\|^2$ for all $x\in V$. By Lemma \ref{lemma-spectral-subspaces}(d), nonzero vectors of $E_S(\nu,\infty)H$ satisfy $\langle Sx,x\rangle>\nu\|x\|^2$, so $V\cap E_S(\nu,\infty)H=0$, and the orthogonal projection onto $E_S[0,\nu]H=(E_S(\nu,\infty)H)^\perp$ is injective on $V$. Hence $M_{\tilde S}(\mu)=\dim V\le M_S(\nu)$. If $f$ is upper semi-Fredholm, $M_S(\nu_0)<\infty$ for some $\nu_0>0$ (Lemma \ref{lemma-usF}(iii)); choosing $\mu=\beta^2C_1^{-2}\nu_0$ gives $M_{\tilde S}(\mu)\le M_S(\nu_0)<\infty$, so $\tilde f$ is upper semi-Fredholm by Lemma \ref{lemma-usF}.
\end{proof}

\subsection{The reservoir lemma}

The following lemma is the key to the noncompact families that are not upper semi-Fredholm. Its idea: a noncompact operator $S$ has an infinite-dimensional spectral band at a fixed height $2^{-i_0}$, the \emph{reservoir}. To realize on $\tilde S$ a low eigenvalue $2^{-j}$ ($j>i_0$) inside the form of $S$ one does not need low eigenvalues of $S$ at all: a unit vector $\sqrt{1-t}\,a+\sqrt t\,b$ with $b$ in the reservoir and $t=2^{i_0-j}$ has form value $(1-t)\langle Sa,a\rangle+2^{-j}$, which lies between $2^{-j}$ and $2\cdot2^{-j}$ as soon as $a$ sits in a band of $S$ of index $\ge j$; the hypothesis $M_S(\mu)=\infty$ for all $\mu>0$ provides infinitely many such $a$, so everything can be done injectively. This is exactly what is excluded for invertible operators (Section \ref{section-inv}).

\begin{lemma}\label{lemma-reservoir}
Let $S\ge0$ be a bounded operator on a separable space $H$ that is noncompact and satisfies $M_S(\mu)=\infty$ for every $\mu>0$, and let $\tilde S\ge0$ be a bounded operator on a separable space $\tilde H$ with $\dim\ker\tilde S\le\dim\ker S$. Then there are an isometry $U:\tilde H\to H$ and constants $0<c\le C$ with $c\,\tilde S\le U^*SU\le C\,\tilde S$.
\end{lemma}
\begin{proof}
\emph{Step 1: reduction to diagonal operators.} Let $E_i=E_S(2^{-i-1},2^{-i}]$, $H_i=E_iH$, and choose orthonormal bases $X_i$ of $H_i$ and $Z$ of $\ker S$; together they form an orthonormal basis $X=\bigsqcup_iX_i$, $X\sqcup Z$, of $H$ (Lemma \ref{lemma-bands}). Let $S^d=\sum_i2^{-i}E_i$ (zero on $\ker S$). On $H_i$ we have $2^{-i-1}\le S\le2^{-i}$ (Lemma \ref{lemma-bands}), hence $\frac12S^d\le S\le S^d$. Define $\tilde E_j,\tilde H_j,\tilde S^d$ and the orthonormal bases $Y_j$ of $\tilde H_j$, $\tilde Z$ of $\ker\tilde S$ similarly; then $\tilde S\le\tilde S^d\le2\tilde S$. It suffices to find an isometry $U$ and $c_0>0$ with
\[c_0\tilde S^d\le U^*S^dU\le2\tilde S^d,\]
because then $U^*SU\ge\frac12U^*S^dU\ge\frac12c_0\tilde S^d\ge\frac12c_0\tilde S$ and $U^*SU\le U^*S^dU\le2\tilde S^d\le4\tilde S$.

\emph{Step 2: the reservoirs.} Since $S$ is noncompact, $N_S(\lambda)=\infty$ for some $\lambda>0$ (Lemma \ref{lemma-compact}). Then $E_S(\lambda,\infty)H$ is infinite-dimensional and contained in a finite sum of bands (as $S$ is bounded), so some band is infinite-dimensional. Let $i_0$ be the smallest index with $X_{i_0}$ infinite, and split $X_{i_0}=P\sqcup Q$ with $P,Q$ infinite. If $\ker S$ is infinite-dimensional split $Z=Z_1\sqcup Z_2$ with both infinite; otherwise put $Z_1=\emptyset$ and $Z_2=Z$. For $j\in\ZZ$ let $A_j=\bigcup_{i\ge j}X_i\cup Z_1$. Each $A_j$ is infinite: $M_S(2^{-j})=\infty$ and $E_S[0,2^{-j}]H=\ker S\oplus\bigoplus_{i\ge j}H_i$ (Lemma \ref{lemma-bands}), so either $Z$ is infinite (and then $Z_1$ is) or $\bigcup_{i\ge j}X_i$ is infinite.

\emph{Step 3: definition of $U$ on the basis.} For $y\in Y=\bigsqcup_jY_j$ let $j(y)$ be its band. Let $Y_\mathrm{hi}=\{y:j(y)\le i_0\}$ and $Y_\mathrm{lo}=\{y:j(y)>i_0\}$; both are countable. Choose an injection $y\mapsto p_y$ of $Y_\mathrm{hi}$ into $P$. Enumerate $Y_\mathrm{lo}=\{y_1,y_2,\dots\}$ and choose inductively $a_{y_n}\in A_{j(y_n)}$ and $b_{y_n}\in Q$ distinct from all earlier choices of the same kind (possible since $A_j$ and $Q$ are infinite). Since $\dim\ker\tilde S\le\dim\ker S$, there is an injection $z\mapsto z'$ of $\tilde Z$ into $Z_2$ (if $\ker S$ is infinite, $Z_2$ is infinite; if it is finite, $Z_2=Z$). Put $t_y=2^{i_0-j(y)}\in(0,1)$ for $y\in Y_\mathrm{lo}$ and
\[Uy=p_y\ (y\in Y_\mathrm{hi}),\qquad Uy=\sqrt{1-t_y}\,a_y+\sqrt{t_y}\,b_y\ (y\in Y_\mathrm{lo}),\qquad Uz=z'\ (z\in\tilde Z),\]
extended linearly and continuously.

\emph{Step 4: $U$ is an isometry.} The vectors $a_y$ ($y\in Y_\mathrm{lo}$) are distinct, lie in bands of index $>i_0$ or in $Z_1$, hence differ from the $b$'s and the $p$'s (which lie in $X_{i_0}$) and from the $z'$ (which lie in $Z_2$). Hence the vectors $Uy$, $Uz$ have pairwise disjoint supports in the orthonormal basis $X\sqcup Z$ and have norm $1$ (as $(1-t_y)+t_y=1$): $U$ is an isometry.

\emph{Step 5: the estimate.} $S^d$ is diagonal in the basis $X\sqcup Z$, with value $2^{-i(x)}$ at $x\in X_{i(x)}$ and $0$ on $Z$. Since the vectors $Uv$ have disjoint supports, $\langle S^dUv,Uv'\rangle=0$ for distinct $v,v'\in Y\sqcup\tilde Z$, so $U^*S^dU$ is diagonal in $Y\sqcup\tilde Z$, with entries $e_y=\langle S^dUy,Uy\rangle$ and $0$ on $\tilde Z$ (as $Uz\in\ker S$); and $\tilde S^d$ is diagonal there with entries $2^{-j(y)}$ and $0$. Let $j_{\min}$ be the smallest band index with $Y_j\ne\emptyset$. For $y\in Y_\mathrm{hi}$ we have $e_y=2^{-i_0}\le2^{-j(y)}$ since $j(y)\le i_0$, and $e_y\ge2^{j_{\min}-i_0}2^{-j(y)}$ since $j(y)\ge j_{\min}$. For $y\in Y_\mathrm{lo}$ we have $e_y=(1-t_y)2^{-i(a_y)}+t_y2^{-i_0}$, where $t_y2^{-i_0}=2^{-j(y)}$ and $(1-t_y)2^{-i(a_y)}\le2^{-j(y)}$ because $i(a_y)\ge j(y)$ (or $a_y\in Z_1$, contributing $0$); hence $2^{-j(y)}\le e_y\le2\cdot2^{-j(y)}$. With $c_0=\min(1,2^{j_{\min}-i_0})$ we get $c_0\tilde S^d\le U^*S^dU\le2\tilde S^d$. (If $\tilde S=0$ there is no $Y$ and $U$ is an isometry of $\tilde H=\ker\tilde S$ into $\ker S$.)
\end{proof}

\subsection{The classification}

We now prove the classification for bounded below comparison operators.

\begin{theorem}\label{theorem-closed}
For Bessel families $f,\tilde f$ the following are equivalent:
\begin{enumerate}[(i)]
\item $f\simeq_\mathrm{bb}\tilde f$;
\item $f\simeq_\mathrm{iso}\tilde f$;
\item $f\simeq_\mathrm{inj}\tilde f$, and $f$ is upper semi-Fredholm if and only if $\tilde f$ is.
\end{enumerate}
In particular $\simeq_\mathrm{bb}=\simeq_\mathrm{iso}$.
\end{theorem}
\begin{proof}
Write $S=S_f$, $\tilde S=S_{\tilde f}$. (ii)$\Rightarrow$(i) holds by Proposition \ref{proposition-inclusions}.

\emph{(i)$\Rightarrow$(iii).} By Proposition \ref{proposition-inclusions} we have $f\simeq_\mathrm{inj}\tilde f$. If $w$ is bounded below with $\rho(f)\circ w\le C_1\rho(\tilde f)$ and $f$ is upper semi-Fredholm, then so is $\tilde f$ by Proposition \ref{proposition-necessary-cl}; the converse uses $\tilde w$.

\emph{(iii)$\Rightarrow$(ii).} By Theorem \ref{theorem-inj}, $f\simeq\tilde f$ and $k(f)=k(\tilde f)$. By Corollary \ref{corollary-classes}(a), $T_f$ is compact iff $T_{\tilde f}$ is. We construct isometries $U:\tilde H\to H$ and $V:H\to\tilde H$ with $c\tilde S\le U^*SU\le C\tilde S$ and $cS\le V^*\tilde SV\le CS$ (constants $0<c\le C$); by Proposition \ref{proposition-reduction}(a), this gives $f\simeq_\mathrm{iso}\tilde f$. There are three cases.

\emph{Case 1: $f,\tilde f$ compact.} Then $T_f,T_{\tilde f}$ have the same rank $r$ (Corollary \ref{corollary-classes}(a)). By Lemma \ref{lemma-compact} there are orthonormal bases $(x_n)_{n\le r}$ of $(\ker S)^\perp$ and $(\tilde x_n)_{n\le r}$ of $(\ker\tilde S)^\perp$ with $Sx_n=\lambda_nx_n$, $\tilde S\tilde x_n=\tilde\lambda_n\tilde x_n$, $\lambda_n=\lambda_n(S)$, $\tilde\lambda_n=\lambda_n(\tilde S)$. Since $k(f)=k(\tilde f)$ there is a unitary $w:\tilde H\to H$ with $w\tilde x_n=x_n$ and mapping $\ker\tilde S$ onto $\ker S$. For $y\in\tilde H$,
\[\langle w^*Swy,y\rangle=\langle Swy,wy\rangle=\sum_n\lambda_n|\langle y,\tilde x_n\rangle|^2,\qquad\langle\tilde Sy,y\rangle=\sum_n\tilde\lambda_n|\langle y,\tilde x_n\rangle|^2,\]
and $\lambda_n\asymp\tilde\lambda_n$ by Theorem \ref{theorem-singular}. So $c\tilde S\le w^*Sw\le C\tilde S$, and $U=w$ is an isometry; $V=w^{-1}$ is treated in the same way.

\emph{Case 2: noncompact and upper semi-Fredholm.} Here $k=k(f)=k(\tilde f)<\infty$ and $(\ker S)^\perp$, $(\ker\tilde S)^\perp$ are infinite-dimensional (noncompact operators have infinite rank). Hence there is a unitary $u:\tilde H\to H$ with $u(\ker\tilde S)=\ker S$ (unitary on the kernels, which have the same finite dimension, and on the complements, which are both infinite-dimensional and separable). Let $P,\tilde P$ be the projections onto $(\ker S)^\perp$, $(\ker\tilde S)^\perp$; then $u^*Pu=\tilde P$. By Lemma \ref{lemma-usF}, $\delta P\le S\le\|S\|P$ and $\tilde\delta\tilde P\le\tilde S\le\|\tilde S\|\tilde P$. Therefore $u^*Su\ge\delta u^*Pu=\delta\tilde P\ge(\delta/\|\tilde S\|)\tilde S$ and $u^*Su\le\|S\|\tilde P\le(\|S\|/\tilde\delta)\tilde S$. Take $U=u$, $V=u^{-1}$.

\emph{Case 3: noncompact and not upper semi-Fredholm.} By assumption $\tilde f$ is not upper semi-Fredholm either, and both are noncompact (Corollary \ref{corollary-classes}(a)); by Lemma \ref{lemma-usF}(iii), $M_S(\mu)=M_{\tilde S}(\mu)=\infty$ for all $\mu>0$. Since $\dim\ker\tilde S=\dim\ker S$, Lemma \ref{lemma-reservoir} provides an isometry $U:\tilde H\to H$ with $c\tilde S\le U^*SU\le C\tilde S$, and, exchanging the roles, an isometry $V:H\to\tilde H$ with $cS\le V^*\tilde SV\le CS$.
\end{proof}

We obtain the description of the classes of $\simeq_\mathrm{bb}=\simeq_\mathrm{iso}$.

\begin{corollary}\label{corollary-closed}
\begin{enumerate}[(a)]
\item On compact families, $\simeq_\mathrm{bb}=\simeq_\mathrm{iso}$ coincides with $\simeq_\mathrm{inj}$.
\item The noncompact families split into one class $(k,\mathrm{sF})$ for each finite $k$ (the upper semi-Fredholm families with $k(f)=k$) and one class $(k,\mathrm{nsF})$ for each $k\in\{0,1,\dots,\infty\}$ (the noncompact families that are not upper semi-Fredholm and have $k(f)=k$). In particular $\simeq_\mathrm{bb}$ is strictly finer than $\simeq_\mathrm{inj}$ on noncompact families.
\end{enumerate}
\end{corollary}
\begin{proof}
(a) A compact family of infinite rank is not upper semi-Fredholm: by Lemma \ref{lemma-compact} it has infinitely many orthonormal eigenvectors $x_n$ with eigenvalues $\lambda_n\to0$, so $E_S[0,\mu]H$ contains all $x_n$ with $\lambda_n\le\mu$ and $M_S(\mu)=\infty$ for every $\mu>0$. A compact family of finite rank has closed range, so it is upper semi-Fredholm iff $k<\infty$ iff $\dim H<\infty$. Now let $f\simeq_\mathrm{inj}\tilde f$ be compact. They have the same rank (Corollary \ref{corollary-classes}(a)) and the same $k$, so both are or are not upper semi-Fredholm simultaneously, and Theorem \ref{theorem-closed}(iii) holds.

(b) By Theorem \ref{theorem-closed} and Corollary \ref{corollary-inj}, the noncompact $\simeq_\mathrm{bb}$-classes are the classes of $\simeq_\mathrm{inj}$ (indexed by $k$) subdivided by whether the family is upper semi-Fredholm. A noncompact upper semi-Fredholm family has $k<\infty$. All values occur: $S_f=\mathrm{id}_{\ell^2}\oplus0_{\FF^k}$ is upper semi-Fredholm for finite $k$, and $S_f=\mathrm{id}_{\ell^2}\oplus\operatorname{diag}(1/n)\oplus0_{\FF^k}$ (with $0_{\ell^2}$ for $k=\infty$) is noncompact and not upper semi-Fredholm. For $k$ finite, the families $\mathrm{id}_{\ell^2}\oplus0_{\FF^k}$ and $\mathrm{id}_{\ell^2}\oplus\operatorname{diag}(1/n)\oplus0_{\FF^k}$ are $\simeq_\mathrm{inj}$- but not $\simeq_\mathrm{bb}$-equivalent.
\end{proof}

Finally, we characterize the families that are $\simeq_\mathrm{bb}$-equivalent to an orthonormal basis.

\begin{proposition}\label{proposition-frames}
Let $f\in\mathcal B(\Omega,H)$ and let $e$ be an orthonormal basis of $\tilde H$. Then $f\simeq_\mathrm{bb}e$ if and only if $f$ is a frame and $\dim H=\dim\tilde H$.
\end{proposition}
\begin{proof}
Here $S_e=\mathrm{id}_{\tilde H}$ and $\rho(e)=\|\cdot\|$.

\emph{($\Rightarrow$)} Let $w:\tilde H\to H$ and $\tilde w:H\to\tilde H$ be bounded below with $\rho(f)\circ w\asymp\|\cdot\|$ and $\|\tilde w\,\cdot\|\asymp\rho(f)$. Then $\rho(f)(x)\ge c\|\tilde wx\|\ge c\beta\|x\|$, so $f$ is a frame; and $w,\tilde w$ are injective in both directions, so $\dim\tilde H\le\dim H\le\dim\tilde H$.

\emph{($\Leftarrow$)} Let $S=S_f$; since $f$ is a frame, $a\le S\le\|S\|$ with $a>0$, so $S^{\pm1/2}$ are bounded and $S^{-1/2}$ is bounded below by $\|S\|^{-1/2}$. Let $U:\tilde H\to H$ be unitary (the dimensions agree). Then $w=S^{-1/2}U$ is bounded below and $\rho(f)(wy)=\|S^{1/2}S^{-1/2}Uy\|=\|y\|$, and $\tilde w=U^{-1}S^{1/2}$ is bounded below with $\|\tilde wx\|=\|S^{1/2}x\|=\rho(f)(x)$.
\end{proof}

\begin{example}[Running example, continued]\label{example-running-closed}
Neither $S_1$ nor $S_2$ is bounded below and $M_{S_1}(\mu)=M_{S_2}(\mu)=\infty$ for every $\mu>0$, so both families lie in the class $(0,\mathrm{nsF})$ and $f\simeq_\mathrm{bb}\tilde f\simeq_\mathrm{iso}f$ by Theorem \ref{theorem-closed}; the isometry $U:\tilde H\to H$ is produced by Lemma \ref{lemma-reservoir}, with $i_0=0$, the reservoir $Q$ inside the eigenspace of the eigenvalue $1$ of $S_1$, and the vectors $a_y$ among the eigenvectors $e_n$ of $\operatorname{diag}(1/n)$. By Proposition \ref{proposition-frames}, neither family is $\simeq_\mathrm{bb}$-equivalent to an orthonormal basis, since neither is a frame.
\end{example}

\section{Invertible and unitary comparison operators}
\label{section-inv}

For invertible comparison operators the counting functions are no longer sufficient: one needs to control the multiplicities of the spectrum on the logarithmic scale.

\medskip
\noindent\textbf{Strategy.} Theorem \ref{theorem-inv} is proved by the cycle (i)$\Rightarrow$(iv)$\Rightarrow$(v)$\Rightarrow$(vi)$\Rightarrow$(iii)$\Rightarrow$(ii)$\Rightarrow$(i), and each arrow is one lemma. The point is that comparing multiplicities on the logarithmic scale reduces to a bounded-displacement matching of eigenvalue bands.
\begin{itemize}
\item \emph{Window lemma} (Lemma \ref{lemma-window}): an invertible $w$ with $c\tilde S\le w^*Sw\le C\tilde S$ cannot send a spectral window $(\alpha,\beta]$ of $\tilde S$ outside a window $(\alpha/K,K\beta]$ of $S$. A vector $x=wy$ with $y$ in the window can leave the window of $S$ only by combining a part of $S$ below the window with a \emph{small} part of $S$ above it, whose form is nevertheless of the order of $\beta$. This is the leak exploited in Lemma \ref{lemma-reservoir}; it is excluded here because $w^{-1}$ is bounded, which allows one to subtract the high part of $x$ inside the range of $w$ at a controlled cost. The comparison of dimensions is then made by a perturbation argument: two almost-nested subspaces differ by at most the dimension of the extra window.
\item \emph{Matching lemma} (Lemma \ref{lemma-matching}): the window inequalities, summed over dyadic bands, say that the multiplicity sequences $(d_j)$, $(\tilde d_j)$ satisfy Hall's condition for the bipartite graph in which an eigenvalue of $\tilde S$ is joined to the eigenvalues of $S$ at distance at most $m$ in the band index. Hall's theorem (for infinite families of finite sets, with infinite multiplicities treated separately) and the Cantor--Bernstein construction give a bijection of bounded displacement.
\item \emph{Unitary lemma} (Lemma \ref{lemma-unitary}): a bijection of displacement $m$ between eigenvector bases defines a unitary $u$ which sends a band of $\tilde S$ into a window of $2m+1$ bands of $S$; the quadratic forms change by at most the factor $2^{m+1}$ band by band, and the $(2m+1)$ residue classes of bands modulo $2m+1$ are orthogonal families, which gives an additional factor $2m+1$.
\end{itemize}
Corollary \ref{corollary-inv}(c) then counts the classes, by encoding $2^{\aleph_0}$ pairwise inequivalent multiplicity sequences with sparse infinite sets of bands.

\subsection{Three lemmas}

We first state the combinatorial language.

\begin{definition}\label{definition-lambda}
Let $d:\ZZ\to\{0,1,2,\dots,\infty\}$. Its \textbf{level set} is
\[\Lambda(d)=\{(j,i):j\in\ZZ,\ i\in\NN,\ 1\le i\le d_j\}\]
(for $d_j=\infty$ all $i\ge1$ are allowed); $\Lambda(d)$ is countable. The \textbf{level} of $(j,i)$ is $j$. For integers $a\le b$ put $d[a,b]=\sum_{j=a}^bd_j\in\{0,\dots,\infty\}$. For two such functions $d,\tilde d$ and $m\in\NN$, a map $\chi:\Lambda(\tilde d)\to\Lambda(d)$ has \textbf{displacement at most $m$} if the levels of $x$ and $\chi(x)$ differ by at most $m$ for all $x$. We say that $(d_j)$ and $(\tilde d_j)$ are \textbf{related by a bijection of bounded displacement} if there are $m$ and a bijection $\chi:\Lambda(\tilde d)\to\Lambda(d)$ of displacement at most $m$. For a positive operator $S$ we apply this to $d=(d_j(S))_j$, so that $d[a,b]=d_{[a,b]}(S)$ in the notation of Definition \ref{definition-invariants}; the points of $\Lambda(d)$ are then the eigenvalues of $S$ counted with multiplicity, grouped by dyadic level (in the compact case).
\end{definition}

The first lemma controls the multiplicity of the spectrum of $\tilde S$ in a window by that of $S$. It is stated for Fredholm comparison operators, with an additive error, because the same argument is needed in Section \ref{section-further}; for invertible $w$ the error vanishes.

\begin{lemma}[Window lemma]\label{lemma-window}
Let $S\ge0$, $\tilde S\ge0$ be bounded operators on $H$, $\tilde H$, and let $w:\tilde H\to H$ be a nonzero Fredholm operator, $k=\dim\ker w$, $q=\dim(H/\ran w)$, with $c\tilde S\le w^*Sw\le C\tilde S$. Put $a=\|w\|^2$, let $b>0$ be such that $\|wy\|^2\ge b^{-1}\|y\|^2$ for $y\in(\ker w)^\perp$ (it exists since $\ran w$ is closed), and $K=\max\{1,2a/c,2bC\}$. Then
\[W_{\tilde S}(\alpha,\beta)\le W_S(\alpha/K,K\beta)+2(k+q)\qquad\text{for all }0<\alpha<\beta .\]
If $w$ is invertible, then $k=q=0$ (with $b=\|w^{-1}\|^2$), $\dim\ker\tilde S=\dim\ker S$ and $W_{\tilde S}(\alpha,\beta)\le W_S(\alpha/K,K\beta)$.
\end{lemma}
\begin{proof}
\emph{Step 0: set-up.} Fix $0<\alpha<\beta$ and put $\alpha'=c\alpha/(2a)$, $\beta'=2bC\beta$. Since $b^{-1}\le a$ and $c\le C$ we have $\alpha'<\beta'$, and $(\alpha',\beta']\subseteq(\alpha/K,K\beta]$, so $m:=\dim E_S(\alpha',\beta']H\le W_S(\alpha/K,K\beta)$; it suffices to prove $W_{\tilde S}(\alpha,\beta)\le m+2(k+q)$, and we may assume $m<\infty$. Let
\[L=E_S[0,\alpha']H,\quad M=E_S(\alpha',\beta']H,\quad T=E_S(\beta',\infty)H,\]
\[Z_1=E_{\tilde S}[0,\alpha]\tilde H,\quad V=E_{\tilde S}(\alpha,\beta]\tilde H,\quad Z_2=Z_1\oplus V=E_{\tilde S}[0,\beta]\tilde H,\]
$\tilde Z=(\ker w)^\perp$, and $G_1=w^{-1}(L)\subseteq G_2=w^{-1}(L\oplus M)$; these are closed subspaces of $\tilde H$ because $w$ is bounded. For a closed subspace $X$ write $P_X$ for the orthogonal projection onto $X$. The dimension of $V$ is $W_{\tilde S}(\alpha,\beta)$.

\emph{Step 1: $G_1$ is almost contained in $Z_1$.} Let $g\in G_1$ and $x=wg\in L$. Then
\[c\langle\tilde Sg,g\rangle\le\langle Sx,x\rangle\le\alpha'\|x\|^2\le\alpha'a\|g\|^2=\tfrac12c\alpha\|g\|^2 .\]
On the other hand, since $Z_1$ and $Z_1^\perp=E_{\tilde S}(\alpha,\infty)\tilde H$ reduce $\tilde S$ and $\tilde S\ge\alpha$ on $Z_1^\perp$, we have $\langle\tilde Sg,g\rangle\ge\alpha\|(1-P_{Z_1})g\|^2$. Hence
\[\|(1-P_{Z_1})g\|^2\le\tfrac12\|g\|^2\qquad(g\in G_1).\tag{6.1}\]

\emph{Step 2: a large part of $Z_2$ is almost contained in $G_2$.} Let $\theta:Z_2\to H/\ran w$ be the linear map $z\mapsto[P_Twz]$; it is continuous with values in a space of dimension $q$, so $\ker\theta$ is closed of codimension at most $q$ in $Z_2$. Put $Y=Z_2\cap\tilde Z\cap\ker\theta$; it is a closed subspace of $Z_2$ of codimension $e'\le k+q$. Let $z\in Y$ and $x=wz$. Then $P_Twz\in\ran w$, hence $x-P_Tx=E_S[0,\beta']x\in\ran w$, and since $z\in\tilde Z$ there is a unique $z'\in\tilde Z$ with $wz'=x-P_Tx$; as $x-P_Tx\in L\oplus M$ we have $z'\in G_2$. Moreover $z-z'\in\tilde Z$ and $w(z-z')=P_Tx$, so $\|z-z'\|^2\le b\|P_Tx\|^2$. Since $T$ reduces $S$ and $S\ge\beta'$ on $T$, and $z\in E_{\tilde S}[0,\beta]\tilde H$,
\[\beta'\|P_Tx\|^2\le\langle Sx,x\rangle\le C\langle\tilde Sz,z\rangle\le C\beta\|z\|^2 .\]
Therefore
\[\operatorname{dist}(z,G_2)^2\le\|z-z'\|^2\le\frac{bC\beta}{\beta'}\|z\|^2=\tfrac12\|z\|^2\qquad(z\in Y).\tag{6.2}\]

\emph{Step 3: two dimension counts.} First, $w$ induces an injective linear map $G_2/G_1\to(L\oplus M)/L\cong M$, so $\dim(G_2/G_1)\le m$. Second, let $G_1^0=\{g\in G_1:g\perp(Z_2\ominus Y)\}$; it is the kernel of the map $G_1\to Z_2\ominus Y$, $g\mapsto P_{Z_2\ominus Y}g$, so $\dim(G_1/G_1^0)\le e'$. Put $N'=G_2\ominus G_1^0$. Then $\dim N'=\dim(G_2/G_1^0)\le m+e'$. For $g\in G_1^0$ the vector $P_{Z_2}g\in Z_2$ has no component in $Z_2\ominus Y$, so $P_{Z_2}g\in Y$.

\emph{Step 4: a perturbation argument.} Define $J:G_1^0\to G_2$ by $Jg=P_{G_2}P_{Z_2}g$. Since $g\in G_1\subseteq G_2$ we have $g=P_{G_2}g$, hence $Jg-g=P_{G_2}(P_{Z_2}g-g)$ and, using $Z_1\subseteq Z_2$ and (6.1),
\[\|Jg-g\|\le\|(1-P_{Z_2})g\|\le\|(1-P_{Z_1})g\|\le2^{-1/2}\|g\| .\]
Let $\Psi_0,\Psi:G_1^0\oplus N'\to G_2$ be given by $\Psi_0(g,n)=g+n$ and $\Psi(g,n)=Jg+n$. Then $\Psi_0$ is unitary and $\|(\Psi-\Psi_0)(g,n)\|=\|Jg-g\|\le2^{-1/2}\|(g,n)\|$, so $\Psi=\Psi_0(1+\Psi_0^{-1}(\Psi-\Psi_0))$ is invertible by a Neumann series. In particular $\ran J+N'=G_2$, so
\[\dim(G_2/\ran J)\le\dim N'\le m+e'.\tag{6.3}\]

\emph{Step 5: conclusion.} Let $A=P_{Z_2}|_{G_1}$, so that $\ran A=P_{Z_2}G_1$, and $B=P_{G_2}|_Y$. By (6.2), $\|y-By\|=\operatorname{dist}(y,G_2)\le2^{-1/2}\|y\|$ for $y\in Y$, so $B$ is injective. By Step 3, $A(G_1^0)\subseteq Y\cap\ran A$, so $\ran J=B(A(G_1^0))\subseteq B(Y\cap\ran A)$. As $B$ is injective,
\begin{align*}\dim\bigl(Y/(Y\cap\ran A)\bigr)&=\dim\bigl(B(Y)/B(Y\cap\ran A)\bigr)\le\dim\bigl(G_2/B(Y\cap\ran A)\bigr)\\&\le\dim(G_2/\ran J)\le m+e'\end{align*}
by (6.3). Since $Y$ has codimension $e'$ in $Z_2$,
\[\dim(Z_2/\ran A)\le e'+\dim\bigl(Y/(Y\cap\ran A)\bigr)\le m+2e'.\]
Finally $V\cap\ran A=0$: if $v=P_{Z_2}g\in V$ with $g\in G_1$, then $V\perp Z_1$ gives $0=P_{Z_1}v=P_{Z_1}P_{Z_2}g=P_{Z_1}g$, so $\|g\|^2=\|(1-P_{Z_1})g\|^2\le\frac12\|g\|^2$ by (6.1), hence $g=0$ and $v=0$. Thus $V$ embeds linearly into $Z_2/\ran A$, and
\[W_{\tilde S}(\alpha,\beta)=\dim V\le m+2e'\le m+2(k+q).\]

\emph{The invertible case.} If $w$ is invertible, then $k=q=0$ and the above gives the stated inequality. For the kernels: $y\in\ker\tilde S$ iff $\langle\tilde Sy,y\rangle=0$ iff $\langle Swy,wy\rangle=0$ (by the two inequalities) iff $wy\in\ker S$; as $w$ is bijective, $w(\ker\tilde S)=\ker S$.
\end{proof}

The second lemma is purely combinatorial: it converts such window estimates into a bijection of bounded displacement.

\begin{lemma}[Matching lemma]\label{lemma-matching}
Let $d,\tilde d:\ZZ\to\{0,1,\dots,\infty\}$ and $m\in\NN$ with
\[\tilde d[a,b]\le d[a-m,b+m]\quad\text{and}\quad d[a,b]\le\tilde d[a-m,b+m]\qquad\text{for all integers }a\le b.\]
Then there is a bijection $\chi:\Lambda(\tilde d)\to\Lambda(d)$ of displacement at most $m$.
\end{lemma}
\begin{proof}
\emph{Step 1: an injection $\varphi:\Lambda(\tilde d)\to\Lambda(d)$ of displacement at most $m$; the infinite levels.} Let $Z_\infty=\{l:d_l=\infty\}$ and $\Sigma=\{j:[j-m,j+m]\cap Z_\infty\ne\emptyset\}$. If $\tilde d_j=\infty$ then $d[j-m,j+m]\ge\tilde d[j,j]=\infty$, so $j\in\Sigma$. For a point of $\Lambda(\tilde d)$ with level $j\in\Sigma$ choose the smallest $l(j)\in Z_\infty$ with $|l(j)-j|\le m$. For each $l\in Z_\infty$ the set of points of $\Lambda(\tilde d)$ whose level $j\in\Sigma$ has $l(j)=l$ is countable, so it injects into the infinite set $\{(l,i):i\ge1\}\subseteq\Lambda(d)$ of all points of level $l$. Different $l$ give different target levels, so $\varphi$ is injective on these points.

\emph{Step 2: the finite levels, by Hall's theorem.} Let $X'$ be the (countable) set of points of $\Lambda(\tilde d)$ of level $j\notin\Sigma$. Then $\tilde d_j<\infty$ (by Step 1) and the window $[j-m,j+m]$ misses $Z_\infty$. For $x\in X'$ of level $j$ let $A_x$ be the set of points of $\Lambda(d)$ with level in $[j-m,j+m]$; it is finite because the window misses $Z_\infty$. We check Hall's condition for a finite set $F\subseteq X'$. Let $U$ be the union of the windows $[j-m,j+m]$ over the levels $j$ of points of $F$, and write $U$ as a disjoint union of maximal intervals $[p_s,q_s]$ of integers ($s=1,\dots,r$). Every level of a point of $F$ lies in some $[p_s+m,q_s-m]$, so $F=F_1\sqcup\dots\sqcup F_r$ with $F_s$ the points of $F$ of level in $[p_s+m,q_s-m]$ (a nonempty $F_s$ forces $p_s+m\le q_s-m$). Then, by the hypothesis with $a=p_s+m$, $b=q_s-m$,
\[|F_s|\le\tilde d[p_s+m,q_s-m]\le d[p_s,q_s].\]
The points of $\Lambda(d)$ with level in $[p_s,q_s]$ all lie in $\bigcup_{x\in F}A_x$ (each $l\in[p_s,q_s]\subseteq U$ lies in a window of some point of $F$), and the intervals are disjoint, so $|\bigcup_{x\in F}A_x|\ge\sum_sd[p_s,q_s]\ge\sum_s|F_s|=|F|$. By Hall's theorem for countable families of finite sets \cite{Hall1948}, there is an injective choice $x\mapsto a_x\in A_x$. These values have levels outside $Z_\infty$, hence differ from the values chosen in Step 1. Altogether we get an injection $\varphi:\Lambda(\tilde d)\to\Lambda(d)$ of displacement at most $m$.

\emph{Step 3: a bijection.} By symmetry of the hypotheses there is also an injection $\psi:\Lambda(d)\to\Lambda(\tilde d)$ of displacement at most $m$. The Cantor--Bernstein construction produces a bijection $\chi$ with $\chi(x)\in\{\varphi(x)\}\cup\{y:\psi(y)=x\}$ for every $x$; in both cases the levels of $x$ and $\chi(x)$ differ by at most $m$.
\end{proof}

The third lemma builds a unitary operator from such a bijection.

\begin{lemma}[Unitary lemma]\label{lemma-unitary}
Let $S\ge0$, $\tilde S\ge0$ be bounded operators on separable spaces $H$, $\tilde H$ with $\dim\ker S=\dim\ker\tilde S$, and let $\chi:\Lambda(\tilde d)\to\Lambda(d)$ be a bijection of displacement at most $m$, where $d=(d_j(S))$, $\tilde d=(d_j(\tilde S))$. Then there is a unitary $u:\tilde H\to H$ with $\gamma^{-1}\tilde S\le u^*Su\le\gamma\tilde S$, where $\gamma=(2m+1)2^{m+1}$.
\end{lemma}
\begin{proof}
\emph{Step 1: the unitary.} Let $H_j=H_j(S)$, $\tilde H^j=H_j(\tilde S)$, and choose orthonormal bases $(e_{j,i})_{1\le i\le d_j}$ of $H_j$ and $(\tilde e_{j,i})_{1\le i\le\tilde d_j}$ of $\tilde H^j$. Define $u_0:(\ker\tilde S)^\perp\to(\ker S)^\perp$ by $u_0\tilde e_{j,i}=e_{\chi(j,i)}$; it is unitary since $\chi$ is a bijection of the index sets. It maps $\tilde H^j$ into $Z_j=\bigoplus_{|l-j|\le m}H_l$. Let $u_1:\ker\tilde S\to\ker S$ be unitary (equal dimensions) and $u=u_0\oplus u_1$.

\emph{Step 2: reduction to one inequality.} It suffices to prove $\langle Suy,uy\rangle\le\gamma\langle\tilde Sy,y\rangle$ for all $y$: applying this to $u^{-1}$, whose associated bijection $\chi^{-1}$ also has displacement at most $m$, gives $\langle\tilde Su^{-1}x,u^{-1}x\rangle\le\gamma\langle Sx,x\rangle$, that is $\gamma^{-1}\tilde S\le u^*Su$. As $S^{1/2}$ and $\tilde S^{1/2}$ vanish on the kernels and $u$ maps $\ker\tilde S$ into $\ker S$, we may take $y\in(\ker\tilde S)^\perp$, and by continuity (both sides are continuous in $y$) we may take $y$ in the dense set of finite sums $y=\sum_jy_j$, $y_j\in\tilde H^j$.

\emph{Step 3: the estimate band by band.} Put $z_j=uy_j\in Z_j$. Since $S\le2^{-(j-m)}$ on $Z_j$ (it consists of bands of index $l\ge j-m$) and $\|z_j\|=\|y_j\|$,
\[\langle Sz_j,z_j\rangle\le2^{m-j}\|y_j\|^2\le2^{m-j}2^{j+1}\langle\tilde Sy_j,y_j\rangle=2^{m+1}\langle\tilde Sy_j,y_j\rangle,\]
where $\langle\tilde Sy_j,y_j\rangle>2^{-j-1}\|y_j\|^2$ was used.

\emph{Step 4: summing over bands.} For a residue class $r$ modulo $2m+1$ the spaces $Z_j$, $j\equiv r$, are sums of bands over pairwise disjoint windows $[j-m,j+m]$ (consecutive such $j$ differ by at least $2m+1$), hence are mutually orthogonal and reduce $S$; so $\langle S\sum_{j\equiv r}z_j,\sum_{j\equiv r}z_j\rangle=\sum_{j\equiv r}\langle Sz_j,z_j\rangle$. With $v_r=\sum_{j\equiv r}z_j$ we have $uy=\sum_{r=0}^{2m}v_r$, and by the Cauchy--Schwarz inequality $\|S^{1/2}\sum_rv_r\|^2\le(2m+1)\sum_r\|S^{1/2}v_r\|^2$. Altogether
\[\langle Suy,uy\rangle\le(2m+1)\sum_j\langle Sz_j,z_j\rangle\le(2m+1)2^{m+1}\sum_j\langle\tilde Sy_j,y_j\rangle=\gamma\langle\tilde Sy,y\rangle,\]
the last equality because the $\tilde H^j$ are mutually orthogonal and reduce $\tilde S$.
\end{proof}

\subsection{The classification}

We now prove the classification for invertible and unitary comparison operators, together with some additional equivalent conditions.

\begin{theorem}\label{theorem-inv}
Let $f,\tilde f$ be Bessel families, $S=S_f$, $\tilde S=S_{\tilde f}$, $d_j=d_j(S)$, $\tilde d_j=d_j(\tilde S)$. The following are equivalent:
\begin{enumerate}[(i)]
\item $f\simeq_\mathrm{inv}\tilde f$;
\item $f\simeq_\mathrm{uni}\tilde f$;
\item there are a unitary $u:\tilde H\to H$ and $0<c\le C$ with $c\,\tilde S\le u^*Su\le C\,\tilde S$;
\item $k(f)=k(\tilde f)$, and there is $K\ge1$ with $W_{\tilde S}(\alpha,\beta)\le W_S(\alpha/K,K\beta)$ and $W_S(\alpha,\beta)\le W_{\tilde S}(\alpha/K,K\beta)$ for all $0<\alpha<\beta$;
\item $k(f)=k(\tilde f)$, and there is $m\in\NN$ with $\tilde d[a,b]\le d[a-m,b+m]$ and $d[a,b]\le\tilde d[a-m,b+m]$ for all integers $a\le b$;
\item $k(f)=k(\tilde f)$, and $(d_j)$, $(\tilde d_j)$ are related by a bijection of bounded displacement (Definition \ref{definition-lambda}).
\end{enumerate}
\end{theorem}
\begin{proof}
\emph{(iii)$\Rightarrow$(ii).} Let $u$ be as in (iii) and put $w=u$, $\tilde w=u^{-1}=u^*$. Then $c\tilde S\le w^*Sw\le C\tilde S$. Conjugating the two inequalities of (iii) by $u$ gives $c\,u\tilde Su^*\le S\le C\,u\tilde Su^*$, that is $C^{-1}S\le\tilde w^*\tilde S\tilde w\le c^{-1}S$ for $\tilde w=u^*$ (as $\tilde w^*\tilde S\tilde w=u\tilde Su^*$). By Proposition \ref{proposition-reduction}(a), $f\simeq_\mathrm{uni}\tilde f$.

\emph{(ii)$\Rightarrow$(i).} Proposition \ref{proposition-inclusions}.

\emph{(i)$\Rightarrow$(iv).} Let $w:\tilde H\to H$ and $\tilde w:H\to\tilde H$ be invertible operators realizing $\simeq_\mathrm{inv}$, with $c\tilde S\le w^*Sw\le C\tilde S$ and $c\,S\le\tilde w^*\tilde S\tilde w\le C\,S$ (Proposition \ref{proposition-reduction}(a)). Lemma \ref{lemma-window} applied to $w$ and to $\tilde w$ gives $\dim\ker\tilde S=\dim\ker S$ and the two window inequalities, with $K$ the larger of the two constants.

\emph{(iv)$\Rightarrow$(v).} Choose $m\in\NN$ with $2^m\ge K$. For integers $a\le b$ we have
\begin{align*}\tilde d[a,b]&=W_{\tilde S}(2^{-b-1},2^{-a})\le W_S(2^{-b-1}/K,K2^{-a})\\&\le W_S(2^{-b-m-1},2^{-a+m})=d[a-m,b+m],\end{align*}
because $(2^{-b-1}/K,K2^{-a}]\subseteq(2^{-(b+m)-1},2^{-(a-m)}]$; and symmetrically.

\emph{(v)$\Rightarrow$(vi).} Lemma \ref{lemma-matching}.

\emph{(vi)$\Rightarrow$(iii).} Lemma \ref{lemma-unitary} gives a unitary $u$ with $\gamma^{-1}\tilde S\le u^*Su\le\gamma\tilde S$.
\end{proof}

We deduce the description of the classes of $\simeq_\mathrm{inv}=\simeq_\mathrm{uni}$.

\begin{corollary}\label{corollary-inv}
Let $f,\tilde f$ be Bessel families.
\begin{enumerate}[(a)]
\item ${\simeq_\mathrm{uni}}={\simeq_\mathrm{inv}}\subseteq{\simeq_\mathrm{bb}}$, and on compact families $\simeq_\mathrm{inv}$ coincides with $\simeq_\mathrm{inj}$.
\item Let $f,\tilde f$ be noncompact with closed range of $T_f$, $T_{\tilde f}$. Then $f\simeq_\mathrm{inv}\tilde f$ iff $k(f)=k(\tilde f)$. For finite $k$ these classes are the classes $(k,\mathrm{sF})$ of Corollary \ref{corollary-closed}.
\item For each $k\in\{0,1,\dots,\infty\}$ the noncompact families with $\dim\ker T_f=k$ and nonclosed range split into exactly $2^{\aleph_0}$ classes of $\simeq_\mathrm{inv}$, although all of them lie in a single class of $\simeq_\mathrm{bb}$.
\item The relation $\simeq_\mathrm{inv}$ is strictly finer than $\simeq_\mathrm{bb}$.
\end{enumerate}
\end{corollary}
\begin{proof}
(a) ${\simeq_\mathrm{uni}}={\simeq_\mathrm{inv}}$ is Theorem \ref{theorem-inv} and the inclusion is Proposition \ref{proposition-inclusions}. Let $f\simeq_\mathrm{inj}\tilde f$ be compact. Then $k(f)=k(\tilde f)$ and the ranks $r$ agree, and by Theorem \ref{theorem-singular} the eigenvalues $\lambda_n=\lambda_n(S)$, $\tilde\lambda_n=\lambda_n(\tilde S)$, $n\le r$, satisfy $L^{-1}\tilde\lambda_n\le\lambda_n\le L\tilde\lambda_n$ for some $L\ge1$. By Lemma \ref{lemma-compact} the points of $\Lambda(d)$ are in bijection with the indices $n\le r$ (the $n$-th eigenvalue counted with multiplicity lies in the band $j_n=\lfloor-\log_2\lambda_n\rfloor$, that is $2^{-j_n-1}<\lambda_n\le2^{-j_n}$), and likewise for $\tilde d$ with $\tilde j_n$. The map $n\mapsto n$ is a bijection $\Lambda(\tilde d)\to\Lambda(d)$, and since $|\log_2\lambda_n-\log_2\tilde\lambda_n|\le\log_2L$ and $|\lfloor s\rfloor-\lfloor t\rfloor|<|s-t|+1$, its displacement is at most $m=\lceil\log_2L\rceil+1$. By Theorem \ref{theorem-inv}(vi), $f\simeq_\mathrm{inv}\tilde f$.

(b) If $T_f$ has closed range, then $E_S(0,\delta)=0$ for some $\delta>0$ (Lemma \ref{lemma-usF}: $S\ge\delta$ on $(\ker S)^\perp$ implies $E(0,\delta/2]H=0$), that is, $d_j=0$ for all $j$ above some $j_1$; also $d_j=0$ below some $j_0$ as $S$ is bounded. The same holds for $\tilde f$, with a common range $[j_0,j_1]$. Both level sets are then countably infinite (noncompact operators have infinite rank), so there is a bijection $\Lambda(\tilde d)\to\Lambda(d)$, and any such bijection has displacement at most $j_1-j_0$. Hence $f\simeq_\mathrm{inv}\tilde f$ iff $k(f)=k(\tilde f)$ by Theorem \ref{theorem-inv}(vi), the necessity of $k(f)=k(\tilde f)$ being part of it. Every $k$ occurs, for instance for $S=\mathrm{id}_{\ell^2}\oplus0_{\FF^k}$. For finite $k$, such families are exactly the noncompact upper semi-Fredholm families with $k(f)=k$ (Lemma \ref{lemma-usF}).

(c) \emph{Step 1: at most $2^{\aleph_0}$ classes.} A class is determined by the sequence $(d_j)$ up to the equivalence, and there are $2^{\aleph_0}$ sequences.

\emph{Step 2: the test family.} Identify the finite $0$-$1$ strings with $\NN$. For each infinite $0$-$1$ sequence $\sigma$ let $A_\sigma=\{\sigma|_n:n\ge1\}\subseteq\NN$ (the set of its finite initial segments); each $A_\sigma$ is infinite and $A_\sigma\cap A_\tau$ is finite for $\sigma\ne\tau$ (it consists of the common initial segments). Put $J_\sigma=\{t^2:t\in A_\sigma\}$ and let
\[S_\sigma=\bigoplus_{j\in J_\sigma}2^{-j}\,\mathrm{id}_{\ell^2}\ \oplus\ 0_{\FF^k}\]
($0_{\ell^2}$ if $k=\infty$), a bounded positive operator on a separable space (Lemma \ref{lemma-basic}(c)). It is noncompact (eigenvalue $2^{-j}$ of infinite multiplicity), its kernel has dimension $k$, and its range is not closed ($J_\sigma$ is infinite, so $0$ is an accumulation point of the spectrum). Moreover $d_j(S_\sigma)=\infty$ for $j\in J_\sigma$ and $0$ otherwise.

\emph{Step 3: the test families are pairwise inequivalent.} Suppose $S_\sigma\simeq_\mathrm{inv}S_\tau$ with $\sigma\ne\tau$ and let $m$ be as in Theorem \ref{theorem-inv}(v); applying it with $a=b$ gives $d_{[a,a]}(S_\tau)\le d_{[a-m,a+m]}(S_\sigma)$ for all $a$. Choose $t\in A_\tau\setminus A_\sigma$ with $2t-1>m$ (possible since $A_\tau\setminus A_\sigma$ is infinite). Then $d_{t^2}(S_\tau)=\infty$, so $d_{[t^2-m,t^2+m]}(S_\sigma)=\infty$, hence there is $s^2\in J_\sigma$ with $|s^2-t^2|\le m$. But distinct squares $s^2\ne t^2$ differ by at least $2t-1>m$, so $s=t$ and $t\in A_\sigma$, a contradiction. So the $S_\sigma$ lie in $2^{\aleph_0}$ different classes.

\emph{Step 4: a single $\simeq_\mathrm{bb}$-class.} All these families are noncompact with kernel dimension $k$, hence in one class of $\simeq_\mathrm{inj}$, and not upper semi-Fredholm, hence in one class $(k,\mathrm{nsF})$ of $\simeq_\mathrm{bb}$ (Corollary \ref{corollary-closed}).

(d) Let $S_1=\mathrm{id}_{\ell^2}\oplus\operatorname{diag}(1/n)$ and $S_2=\bigoplus_{j\ge0}2^{-j}\mathrm{id}_{\ell^2}$. Both are noncompact with $k=0$ and $M_{S_1}(\mu)=M_{S_2}(\mu)=\infty$ for all $\mu>0$, hence $S_1\simeq_\mathrm{bb}S_2$ by Theorem \ref{theorem-closed}. But $d_j(S_1)<\infty$ for $j\ge1$ (the eigenvalues $1/n$ lie in band $j$ iff $2^j\le n<2^{j+1}$, so $d_j(S_1)=2^j$), whereas $d_j(S_2)=\infty$ for every $j\ge0$. For any $m$ and $a>m+1$ one has $d_{[a,a]}(S_2)=\infty>d_{[a-m,a+m]}(S_1)$, so Theorem \ref{theorem-inv}(v) fails and $S_1\not\simeq_\mathrm{inv}S_2$.
\end{proof}

\begin{example}[Running example, continued]\label{example-running-inv}
For the families of Example \ref{example-running}, Corollary \ref{corollary-inv}(d) shows that $f\not\simeq_\mathrm{inv}\tilde f$, hence $f\not\simeq_\mathrm{uni}\tilde f$: the multiplicity of the spectrum in the band $j$ is finite for $f$ and infinite for $\tilde f$ once $j\ge1$, and no bounded-displacement bijection can repair this. The pair $(f,\tilde f)$ is the simplest example separating $\simeq_\mathrm{bb}$ from $\simeq_\mathrm{inv}$: $\simeq_\mathrm{bb}$ allows the reservoir leak, while $\simeq_\mathrm{inv}$ does not.
\end{example}

\section{An operator Banach--Mazur distance}
\label{section-distance}

The classical Banach--Mazur distance between two isomorphic Banach spaces $X,Y$ is $\log\inf\{\|T\|\,\|T^{-1}\|\}$ over all linear isomorphisms $T:X\to Y$ (bounded with bounded inverse); it measures how far an isomorphism is from being an isometry. The distance $D$ below is the analogue for the seminorms $\rho(f)$: an isomorphism $T$ is replaced by the pair $(w,\tilde w)$ of comparison operators, and the condition that $T$ be invertible is replaced by the exact equalities $w^*S_fw=S_{\tilde f}$, $\tilde w^*S_{\tilde f}\tilde w=S_f$. Since two families that are not equivalent should be at infinite distance, we allow the value $\infty$.

\begin{definition}\label{definition-pseudometric}
An \textbf{extended pseudometric} on a set $X$ is a function $d:X\times X\to[0,\infty]$ with $d(x,x)=0$, $d(x,y)=d(y,x)$ and $d(x,z)\le d(x,y)+d(y,z)$; the value $\infty$ is allowed and distinct points may have distance $0$.
\end{definition}

The distance $D$ is defined through the pairs of comparison operators realizing the relation exactly.

\begin{definition}\label{definition-exact}
Call a pair $(w,\tilde w)$ of bounded operators \textbf{exact} for $(f,\tilde f)$ if $w^*S_fw=S_{\tilde f}$ and $\tilde w^*S_{\tilde f}\tilde w=S_f$, that is, $\rho(f)\circ w=\rho(\tilde f)$ and $\rho(\tilde f)\circ\tilde w=\rho(f)$; a single operator $w$ with $w^*S_fw=S_{\tilde f}$ is called \textbf{exact} for $(f,\tilde f)$. For Bessel families $f,\tilde f$ put
\[D(f,\tilde f):=\log\inf\{\max(1,\|w\|\,\|\tilde w\|):(w,\tilde w)\text{ exact for }(f,\tilde f)\}\in[0,\infty],\]
with $\inf\emptyset=\infty$. By Theorem \ref{theorem-all}(ii), exact pairs exist iff $f\simeq\tilde f$.
\end{definition}

\begin{proposition}\label{proposition-distance-metric}
$D$ is an extended pseudometric on Bessel families, $D<\infty\iff f\simeq\tilde f$, $D(f,cf)=0$ for every scalar $c\ne0$, and hence $D(cf,g)=D(f,g)$. If $S_f\ne0$ the maximum with $1$ is inactive.
\end{proposition}
\begin{proof}
Symmetry is clear and $D(f,f)=0$ (take $w=\tilde w=1$).

\emph{Triangle inequality.} If $D(f,g)$ or $D(g,h)$ is infinite there is nothing to prove. Otherwise, for any exact pairs $(w_1,\tilde w_1)$ for $(f,g)$ and $(w_2,\tilde w_2)$ for $(g,h)$, put $w=w_1w_2$, $\tilde w=\tilde w_2\tilde w_1$. Then $w^*S_fw=w_2^*S_gw_2=S_h$ and $\tilde w^*S_h\tilde w=\tilde w_1^*S_g\tilde w_1=S_f$, so $(w,\tilde w)$ is exact for $(f,h)$, and $\|w\|\|\tilde w\|\le\|w_1\|\|w_2\|\|\tilde w_2\|\|\tilde w_1\|$, so $\max(1,\|w\|\|\tilde w\|)\le\max(1,\|w_1\|\|\tilde w_1\|)\max(1,\|w_2\|\|\tilde w_2\|)$. Taking the infimum over the two pairs and then logarithms turns the product into a sum.

\emph{Invariance.} Since $S_{cf}=|c|^2S_f$, the pair $(|c|\,1,|c|^{-1}1)$ is exact for $(f,cf)$ with product of norms $1$, so $D(f,cf)=0$; the second assertion follows from the triangle inequality.

\emph{The maximum with $1$.} If $(w,\tilde w)$ is exact for $(f,\tilde f)$ then $S_f=\tilde w^*(w^*S_fw)\tilde w=(w\tilde w)^*S_f(w\tilde w)$, so $\|S_f\|\le\|w\tilde w\|^2\|S_f\|$; if $S_f\ne0$ this gives $1\le\|w\tilde w\|\le\|w\|\|\tilde w\|$.
\end{proof}

For positive operators $S,\tilde S$ put
\[\kappa(\tilde S|S):=\inf\{c>0:N_{\tilde S}(\lambda)\le N_S(\lambda/c)\text{ for all }\lambda>0\}\in[0,\infty].\]
The set of admissible $c$ is closed upwards (as $N_S$ is nonincreasing), so every $c>\kappa(\tilde S|S)$ is admissible; and $\kappa(\tilde S|S)=0$ iff $\tilde S=0$ (if every small $c$ is admissible then $N_{\tilde S}(\lambda)=0$ for $\lambda>c\|S\|$ and all $c>0$).

\begin{theorem}\label{theorem-distance}
Let $f,\tilde f$ be Bessel families, $S=S_f$, $\tilde S=S_{\tilde f}$.
\begin{enumerate}[(a)]
\item Every exact $w$ has $\|w\|^2\ge\kappa(\tilde S|S)$, and for every $c>\kappa(\tilde S|S)$ there is an exact $w$ with $\|w\|^2<c$. Thus $\inf\|w\|^2=\kappa(\tilde S|S)$ over exact $w$.
\item If $S\ne0\ne\tilde S$, then $D(f,\tilde f)=\frac12\log\bigl(\kappa(\tilde S|S)\,\kappa(S|\tilde S)\bigr)$, and it is finite iff $f\simeq\tilde f$. If $S=\tilde S=0$ then $D=0$, and if exactly one of $S,\tilde S$ is $0$ then $D=\infty$.
\item $\kappa(\tilde S|S)=\sup_n\lambda_n(\tilde S)/\lambda_n(S)$ (with $0/0=0$ and $a/0=\infty$ for $a>0$), so that, if $S\ne0\ne\tilde S$,
\[D(f,\tilde f)=\log\Bigl(\sup_n\frac{s_n}{\tilde s_n}\cdot\sup_n\frac{\tilde s_n}{s_n}\Bigr),\qquad s_n=s_n(f),\ \tilde s_n=s_n(\tilde f).\]
\item If $T_f,T_{\tilde f}$ are compact and $f\simeq\tilde f$, the infimum in $D$ is attained. For example $s_n=n^{-a}$, $\tilde s_n=(n+1)^{-a}$ gives $D=a\log2$, and $n^{-a}$ versus $n^{-b}$ gives $D=\infty$ for $a\ne b$.
\end{enumerate}
\end{theorem}
\begin{proof}
(a) \emph{Lower bound.} Let $w$ be exact. If $w=0$ then $\tilde S=0$ and $\kappa(\tilde S|S)=0$. Otherwise Proposition \ref{proposition-counting} with $K_1=1$ gives $N_{\tilde S}(\lambda)\le N_S(\lambda/\|w\|^2)$, i.e.\ $\|w\|^2$ is admissible, so $\|w\|^2\ge\kappa(\tilde S|S)$.

\emph{Upper bound.} Let $c>\kappa(\tilde S|S)$ and choose $c'\in(\kappa(\tilde S|S),c)$ and $q>1$ with $q^2c'<c$. Let $m\in\ZZ$ be the smallest integer with $q^m\ge c'$, so $q^{m-1}<c'$. Since $c'$ is admissible and $N_S$ is nonincreasing, $N_{\tilde S}(\lambda)\le N_S(\lambda/c')\le N_S(q^{-m}\lambda)$. Lemma \ref{lemma-embedding} gives a partial isometry $J$ with initial space $(\ker\tilde S)^\perp$ and $J\tilde SJ^*\le q^{m+1}S$, and Douglas' lemma applied as in the proof of Theorem \ref{theorem-all} gives $w$ with $S^{1/2}w=J\tilde S^{1/2}$, $w^*Sw=\tilde S$ and $\|w\|^2\le q^{m+1}<q^2c'<c$.

(b) Exact pairs consist of two independently chosen exact operators $w$ and $\tilde w$ (for $(f,\tilde f)$ and $(\tilde f,f)$), so by (a) $\inf\|w\|\|\tilde w\|=\sqrt{\kappa(\tilde S|S)\kappa(S|\tilde S)}$ whenever exact pairs exist, that is, whenever $f\simeq\tilde f$ (Theorem \ref{theorem-all}). If $S\ne0\ne\tilde S$ the maximum with $1$ is inactive (Proposition \ref{proposition-distance-metric}), which gives the formula; if $f\not\simeq\tilde f$ then $D=\infty$, and one of the two $\kappa$ is infinite while the other is positive, so the right-hand side is $\infty$ as well. If $S=\tilde S=0$ then $w=\tilde w=0$ is exact and $\max(1,0)=1$. If exactly one is $0$, no exact pair exists, since $w^*0\,w=0\ne\tilde S$.

(c) Let $c\ge\sup_n\tilde\lambda_n/\lambda_n$ with $\tilde\lambda_n=\lambda_n(\tilde S)$, $\lambda_n=\lambda_n(S)$, i.e.\ $\tilde\lambda_n\le c\lambda_n$ for all $n$. If $N_{\tilde S}(\lambda)\ge n$ for a finite $n$, then $\tilde\lambda_n>\lambda$ by Lemma \ref{lemma-lambda}, so $\lambda_n\ge\tilde\lambda_n/c>\lambda/c$ and $N_S(\lambda/c)\ge n$; as the counts take values in $\{0,\dots,\infty\}$ in separable spaces, $N_{\tilde S}(\lambda)\le N_S(\lambda/c)$. Conversely, if $c$ is admissible and $\lambda<\tilde\lambda_n$, then $N_{\tilde S}(\lambda)\ge n$, so $N_S(\lambda/c)\ge n$ and $\lambda_n>\lambda/c$; letting $\lambda\uparrow\tilde\lambda_n$ gives $\tilde\lambda_n\le c\lambda_n$. Hence $\kappa(\tilde S|S)=\sup_n\tilde\lambda_n/\lambda_n$, and the formula for $D$ follows from (b) and $\lambda_n=s_n^2$, since $\frac12\log(\sup\tilde s_n^2/s_n^2\cdot\sup s_n^2/\tilde s_n^2)=\log(\sup\tilde s_n/s_n\cdot\sup s_n/\tilde s_n)$.

(d) If $f\simeq\tilde f$ are compact, the ranks coincide and $\tilde s_n\le Cs_n$; with the spectral decompositions $S=\sum_ns_n^2\langle\cdot,x_n\rangle x_n$, $\tilde S=\sum_n\tilde s_n^2\langle\cdot,\tilde x_n\rangle\tilde x_n$ (Lemma \ref{lemma-compact}) put $wy=\sum_n(\tilde s_n/s_n)\langle y,\tilde x_n\rangle x_n$, extended by $0$ on $\ker\tilde S$. Then $\langle w^*Swy,y\rangle=\sum_ns_n^2(\tilde s_n/s_n)^2|\langle y,\tilde x_n\rangle|^2=\langle\tilde Sy,y\rangle$, so $w$ is exact, and $\|w\|=\sup_n\tilde s_n/s_n=\sqrt{\kappa(\tilde S|S)}$. The same construction with the roles exchanged gives $\tilde w$, so the infimum is attained. In the example $\sup_ns_n/\tilde s_n=\sup_n((n+1)/n)^a=2^a$ (at $n=1$) and $\sup_n\tilde s_n/s_n=\sup_n(n/(n+1))^a=1$, so $D=a\log2$; for $a\ne b$ the ratios $n^{b-a}$ are unbounded in one direction.
\end{proof}

\begin{example}[Running example, continued]\label{example-running-distance}
For the families of Example \ref{example-running}, $\lambda_n(S_1)=\lambda_n(S_2)=1$ for all $n$, so $\kappa(S_2|S_1)=\kappa(S_1|S_2)=1$ and $D(f,\tilde f)=0$ by Theorem \ref{theorem-distance}(c), although $f\not\simeq_\mathrm{inv}\tilde f$. Thus $D$ is blind to the dyadic multiplicities, which are the invariants that separate $\simeq_\mathrm{inv}$ from $\simeq_\mathrm{bb}$.
\end{example}

\section{Further classes of comparison operators}
\label{section-further}

The proof of Proposition \ref{proposition-reduction} uses only that $\mathcal O$ contains the identities and is closed under composition, that is, that $\mathcal O$ is a wide subcategory of $\Hilb$. Every such class defines an equivalence relation $\simeq_{\mathcal O}$ which depends only on the frame operators. In this section we study five further classes.

\begin{definition}\label{definition-further}
A bounded operator $w:\tilde H\to H$ is called
\begin{itemize}
\item \textbf{surjective} if $\ran w=H$;
\item a \textbf{co-isometry} if $w^*$ is an isometry, that is, $ww^*=1$ (so $w$ is surjective);
\item \textbf{upper semi-Fredholm} if $\dim\ker w<\infty$ and $\ran w$ is closed;
\item \textbf{Fredholm} if it is upper semi-Fredholm and $\dim H/\ran w<\infty$;
\item of \textbf{dense range} if $\overline{\ran w}=H$.
\end{itemize}
We write $\mathcal O_\mathrm{sur}$, $\mathcal O_\mathrm{coi}$, $\mathcal O_\mathrm{usF}$, $\mathcal O_\mathrm{F}$, $\mathcal O_\mathrm{dr}$ for these classes and $\simeq_\mathrm{sur}$, $\simeq_\mathrm{coi}$, $\simeq_\mathrm{usF}$, $\simeq_\mathrm{F}$, $\simeq_\mathrm{dr}$ for the corresponding relations. Each class contains the identities and is closed under composition (for the semi-Fredholm and Fredholm classes this is classical). Moreover $\mathcal O_\mathrm{uni}\subseteq\mathcal O_\mathrm{coi}\subseteq\mathcal O_\mathrm{sur}\subseteq\mathcal O_\mathrm{dr}$, $\mathcal O_\mathrm{inv}\subseteq\mathcal O_\mathrm{sur}$, $\mathcal O_\mathrm{inv}\subseteq\mathcal O_\mathrm{F}\subseteq\mathcal O_\mathrm{usF}$ and $\mathcal O_\mathrm{bb}\subseteq\mathcal O_\mathrm{usF}$, hence (as in Proposition \ref{proposition-inclusions})
\begin{gather*}{\simeq_\mathrm{uni}}\subseteq{\simeq_\mathrm{coi}}\subseteq{\simeq_\mathrm{sur}}\subseteq{\simeq_\mathrm{dr}}\subseteq{\simeq},\qquad{\simeq_\mathrm{inv}}\subseteq{\simeq_\mathrm{F}}\subseteq{\simeq_\mathrm{usF}}\subseteq{\simeq},\\{\simeq_\mathrm{inv}}\subseteq{\simeq_\mathrm{sur}},\qquad{\simeq_\mathrm{bb}}\subseteq{\simeq_\mathrm{usF}}.\end{gather*}
\end{definition}

\begin{remark}\label{remark-further}
(a) If $\mathcal O$ is closed under multiplication by positive scalars, then $\simeq_{\mathcal O}$ does not change when $\mathcal O$ is replaced by its contractions, since $\rho(f)\circ(tw)=t\,\rho(f)\circ w$ and constants are free; in particular the class of all contractions gives the relation $\simeq$. (b) The classes of compact operators, of partial isometries and of operators with closed range do not define equivalence relations: the first does not contain the identities, and the other two are not closed under composition, which is needed for transitivity in Proposition \ref{proposition-reduction}(b).
\end{remark}

We first record how comparison operators act on kernels.

\begin{lemma}\label{lemma-kernel}
Let $S\ge0$, $\tilde S\ge0$ be bounded operators on $H$, $\tilde H$ and let $w:\tilde H\to H$ be bounded with $c\tilde S\le w^*Sw\le C\tilde S$. Then $w^{-1}(\ker S)=\ker\tilde S$. In particular $w(\ker\tilde S)\subseteq\ker S$ and $\ker w\subseteq\ker\tilde S$.
\end{lemma}
\begin{proof}
If $y\in\ker\tilde S$ then $\|S^{1/2}wy\|^2=\langle w^*Swy,y\rangle\le C\langle\tilde Sy,y\rangle=0$, so $wy\in\ker S$. If $wy\in\ker S$ then $c\langle\tilde Sy,y\rangle\le\langle Swy,wy\rangle=0$, so $y\in\ker\tilde S$.
\end{proof}

\subsection{Surjective and co-isometric comparison operators}

Surjectivity brings nothing new.

\begin{theorem}\label{theorem-sur}
$\simeq_\mathrm{sur}=\simeq_\mathrm{coi}=\simeq_\mathrm{inv}=\simeq_\mathrm{uni}$.
\end{theorem}
\begin{proof}
By Theorem \ref{theorem-inv} and the inclusions of Definition \ref{definition-further} we have $\simeq_\mathrm{uni}\subseteq\simeq_\mathrm{coi}\subseteq\simeq_\mathrm{sur}$ and $\simeq_\mathrm{uni}=\simeq_\mathrm{inv}$, so it suffices to show $\simeq_\mathrm{sur}\subseteq\simeq_\mathrm{inv}$. Let $f\simeq_\mathrm{sur}\tilde f$, with $S=S_f$, $\tilde S=S_{\tilde f}$, and let $w:\tilde H\to H$, $\tilde w:H\to\tilde H$ be surjective with $c\tilde S\le w^*Sw\le C\tilde S$ and $cS\le\tilde w^*\tilde S\tilde w\le CS$ (Proposition \ref{proposition-reduction}(a)). Put $Z=(\ker S)^\perp$, $\tilde Z=(\ker\tilde S)^\perp$ and let $P$ be the orthogonal projection onto $Z$.

\emph{Step 1: kernels.} By Lemma \ref{lemma-kernel}, $w^{-1}(\ker S)=\ker\tilde S$. If $x\in\ker S$ then $x=wy$ for some $y$ since $w$ is surjective, and $y\in w^{-1}(\ker S)=\ker\tilde S$; hence $w(\ker\tilde S)=\ker S$ and $\dim\ker S\le\dim\ker\tilde S$. The same argument for $\tilde w$ gives the reverse inequality, so $k(f)=k(\tilde f)$.

\emph{Step 2: the operator on the complements.} Let $w'=Pw|_{\tilde Z}:\tilde Z\to Z$. It is injective: if $y\in\tilde Z$ and $Pwy=0$ then $wy\in\ker S$, so $y\in\ker\tilde S\cap\tilde Z=0$. It is surjective: let $x\in Z$ and $x=wy$ with $y=y_0+y_1$, $y_0\in\ker\tilde S$, $y_1\in\tilde Z$; then $wy_0\in\ker S$ and $x=Px=Pwy_1=w'y_1$. By the open mapping theorem $w'$ is invertible. Since $S^{1/2}$ vanishes on $\ker S$, $\rho(f)(w'y)=\rho(f)(wy)$ for $y\in\tilde Z$.

\emph{Step 3: conclusion.} Let $u:\ker\tilde S\to\ker S$ be unitary (equal dimensions) and define $w_1=w'\oplus u:\tilde H=\tilde Z\oplus\ker\tilde S\to H=Z\oplus\ker S$, an invertible operator. For $y=y_1+y_0$ as above, $\rho(f)(w_1y)=\rho(f)(w'y_1)=\rho(f)(wy_1)$ and $c\,\rho(\tilde f)(y_1)^2\le\rho(f)(wy_1)^2\le C\rho(\tilde f)(y_1)^2$, while $\rho(\tilde f)(y_1)=\rho(\tilde f)(y)$. Thus $\rho(f)\circ w_1\asymp\rho(\tilde f)$. The same construction applied to $\tilde w$ gives an invertible $\tilde w_1$ with $\rho(\tilde f)\circ\tilde w_1\asymp\rho(f)$, so $f\simeq_\mathrm{inv}\tilde f$.
\end{proof}

Thus nothing new arises from surjectivity: the finest relations see only the invertible structure, and Theorem \ref{theorem-inv} classifies all four. The upper semi-Fredholm class is different, because it allows the kernel to change by a finite-dimensional amount.

\subsection{Upper semi-Fredholm comparison operators}

We start with a necessary condition.

\begin{proposition}\label{proposition-usF-necessary}
Let $f,\tilde f$ be Bessel families, $S=S_f$, $\tilde S=S_{\tilde f}$, and let $w:\tilde H\to H$ be upper semi-Fredholm with $c\tilde S\le w^*Sw\le C\tilde S$. Then $\dim\ker\tilde S\le\dim\ker S+\dim\ker w$, and there is $\beta>0$ such that for all $\mu>0$
\[M_{\tilde S}(\mu)\le M_S(C\beta^{-2}\mu)+\dim\ker w.\]
In particular, if $f$ is upper semi-Fredholm then so is $\tilde f$.
\end{proposition}
\begin{proof}
By Lemma \ref{lemma-kernel}, $w$ maps $\ker\tilde S$ into $\ker S$ with kernel $\ker w\cap\ker\tilde S=\ker w$, so $\dim\ker\tilde S\le\dim\ker S+\dim\ker w$. Since $w$ has closed range, it is bounded below on $(\ker w)^\perp$: $\|wy\|\ge\beta\|y\|$ there. Fix $\mu>0$, let $L=E_{\tilde S}[0,\mu]\tilde H$ and $L'=L\cap(\ker w)^\perp$; $L'$ is the kernel of the restriction to $L$ of the orthogonal projection onto $\ker w$, so $\dim L\le\dim L'+\dim\ker w$. For $y\in L'$, by Lemma \ref{lemma-spectral-subspaces}(c), $\langle Swy,wy\rangle\le C\langle\tilde Sy,y\rangle\le C\mu\|y\|^2\le\nu\|wy\|^2$ with $\nu=C\beta^{-2}\mu$. The space $V=wL'$ has dimension $\dim L'$ ($w$ is injective on $L'$) and $\langle Sx,x\rangle\le\nu\|x\|^2$ on $V$, so $V\cap E_S(\nu,\infty)H=0$ by Lemma \ref{lemma-spectral-subspaces}(d), and the orthogonal projection onto $E_S[0,\nu]H$ is injective on $V$. Hence $\dim L'\le M_S(\nu)$, which is the inequality. If $f$ is upper semi-Fredholm, $M_S(\nu_0)<\infty$ for some $\nu_0>0$ (Lemma \ref{lemma-usF}), and $\mu=\beta^2C^{-1}\nu_0$ gives $M_{\tilde S}(\mu)<\infty$, so $\tilde f$ is upper semi-Fredholm by Lemma \ref{lemma-usF}.
\end{proof}

Upper semi-Fredholm comparison operators allow the kernel to change by a finite-dimensional amount, as the following construction shows.

\begin{lemma}\label{lemma-extension}
Let $f$ be a Bessel family in $H$, $m\in\{0,1,2,\dots\}$, and let $f_m$ be a Bessel family in $H\oplus\FF^m$ with $S_{f_m}=S_f\oplus0$ (Lemma \ref{lemma-basic}(c)). Then $f\simeq_\mathrm{F}f_m$ (hence $f\simeq_\mathrm{usF}f_m$ and $f\simeq f_m$), $k(f_m)=k(f)+m$ (with $\infty+m=\infty$), and $T_{f_m}$ is compact (respectively has closed range) iff $T_f$ is. In particular $\simeq_\mathrm{inv}\subsetneq\simeq_\mathrm{F}$.
\end{lemma}
\begin{proof}
Let $\iota:H\to H\oplus\FF^m$, $x\mapsto(x,0)$, and $\pi:H\oplus\FF^m\to H$ the projection. Then $\rho(f_m)\circ\iota=\rho(f)$ and $\rho(f)\circ\pi=\rho(f_m)$. The operator $\iota$ is an isometry with cokernel $\FF^m$ and $\pi$ is onto with kernel $\FF^m$, so both are Fredholm. The remaining assertions follow from $\ker(S\oplus0)=\ker S\oplus\FF^m$, $\ran(S\oplus0)=\ran S\oplus0$, and the fact that $S\oplus0$ is compact iff $S$ is. Finally, take $k(f)<\infty$ (for instance $S_f=\mathrm{id}_{\ell^2}$) and $m\ge1$: then $k(f_m)=k(f)+m\ne k(f)$, so $f\not\simeq_\mathrm{inv}f_m$ by Theorem \ref{theorem-inv}. (If $k(f)=\infty$ then $k(f_m)=k(f)$ and the kernel dimension does not separate $f$ from $f_m$.)
\end{proof}

We can now classify the upper semi-Fredholm relation.

\begin{theorem}\label{theorem-usF}
For Bessel families $f,\tilde f$ the following are equivalent:
\begin{enumerate}[(i)]
\item $f\simeq_\mathrm{usF}\tilde f$;
\item $f\simeq\tilde f$, $k(f)<\infty\iff k(\tilde f)<\infty$, and $T_f$ is upper semi-Fredholm iff $T_{\tilde f}$ is.
\end{enumerate}
\end{theorem}
\begin{proof}
\emph{(i)$\Rightarrow$(ii).} The relation $f\simeq\tilde f$ holds by inclusion. Let $w,\tilde w$ be upper semi-Fredholm realizing (i). By Proposition \ref{proposition-usF-necessary}, $k(\tilde f)\le k(f)+\dim\ker w$, so $k(f)<\infty$ implies $k(\tilde f)<\infty$, and symmetrically; and $f$ upper semi-Fredholm implies $\tilde f$ upper semi-Fredholm, and symmetrically.

\emph{(ii)$\Rightarrow$(i).} \emph{Case $k(f)=k(\tilde f)=\infty$.} Neither family is upper semi-Fredholm, $f\simeq\tilde f$ and $k(f)=k(\tilde f)$, so $f\simeq_\mathrm{inj}\tilde f$ (Theorem \ref{theorem-inj}) and then $f\simeq_\mathrm{bb}\tilde f$ (Theorem \ref{theorem-closed}(iii)); and $\simeq_\mathrm{bb}\subseteq\simeq_\mathrm{usF}$.

\emph{Case $k(f),k(\tilde f)<\infty$.} By symmetry assume $k(f)\le k(\tilde f)$ and put $m=k(\tilde f)-k(f)$. With $f_m$ as in Lemma \ref{lemma-extension} we have $f_m\simeq f\simeq\tilde f$ and $k(f_m)=k(\tilde f)$, so $f_m\simeq_\mathrm{inj}\tilde f$ (Theorem \ref{theorem-inj}). Both $T_{f_m}$ and $T_f$ are upper semi-Fredholm or not simultaneously (finite kernel, same range), so by (ii) $f_m$ and $\tilde f$ are upper semi-Fredholm simultaneously, and Theorem \ref{theorem-closed} gives $f_m\simeq_\mathrm{bb}\tilde f$, hence $f_m\simeq_\mathrm{usF}\tilde f$. Together with $f\simeq_\mathrm{usF}f_m$ (Lemma \ref{lemma-extension}) and transitivity, $f\simeq_\mathrm{usF}\tilde f$.
\end{proof}

Theorem \ref{theorem-usF} yields the following classification.

\begin{corollary}\label{corollary-usF}
The noncompact families split into exactly three $\simeq_\mathrm{usF}$-classes: the upper semi-Fredholm ones ($k<\infty$ and closed range), those with $k<\infty$ and nonclosed range, and those with $k=\infty$. The compact families are classified by the pair (class of the singular numbers, whether $k<\infty$).
\end{corollary}
\begin{proof}
All noncompact families are $\simeq$-equivalent (Corollary \ref{corollary-classes}(d)), so Theorem \ref{theorem-usF} leaves the invariants ``$k<\infty$'' and ``upper semi-Fredholm'', and the latter implies the former. The three possibilities occur: $S=\mathrm{id}_{\ell^2}$, $S=\mathrm{id}_{\ell^2}\oplus\operatorname{diag}(1/n)$ and $S=\mathrm{id}_{\ell^2}\oplus0_{\ell^2}$ (Lemma \ref{lemma-basic}(c)). For compact families, upper semi-Fredholmness is determined by the other invariants: a compact family of infinite rank is not upper semi-Fredholm (proof of Corollary \ref{corollary-closed}(a)), and one of finite rank is upper semi-Fredholm iff $k<\infty$.
\end{proof}

In particular the two $\simeq_\mathrm{bb}$-classes $(k,\mathrm{nsF})$ and $(k',\mathrm{nsF})$ with finite $k\ne k'$, as well as $(k,\mathrm{sF})$ and $(k',\mathrm{sF})$, are merged, and the relations $\simeq_\mathrm{inj}$ and $\simeq_\mathrm{usF}$ are incomparable: two noncompact upper semi-Fredholm families with $k=0$ (a frame) and $k=1$ (a frame for its closed span, of codimension $1$) are $\simeq_\mathrm{usF}$- but not $\simeq_\mathrm{inj}$-equivalent, while a frame and a nonframe with $k=0$ are $\simeq_\mathrm{inj}$- but not $\simeq_\mathrm{usF}$-equivalent.

\subsection{Fredholm comparison operators}

A Fredholm comparison operator changes the kernel by a finite amount, like an upper semi-Fredholm one, but it also has a finite-dimensional cokernel, and this is enough to keep the multiplicities of the spectrum. We first record the easy part of the classification.

\begin{proposition}\label{proposition-fredholm}
\begin{enumerate}[(a)]
\item ${\simeq_\mathrm{inv}}\subseteq{\simeq_\mathrm{F}}\subseteq{\simeq_\mathrm{usF}}$.
\item If $f\simeq_\mathrm{usF}\tilde f$ and either $T_f,T_{\tilde f}$ are both compact or both upper semi-Fredholm, then $f\simeq_\mathrm{F}\tilde f$.
\end{enumerate}
\end{proposition}
\begin{proof}
(a) is the inclusion of the classes. (b) By Theorem \ref{theorem-usF}, $f\simeq\tilde f$ and $k(f)<\infty\iff k(\tilde f)<\infty$. If $k(f)=k(\tilde f)=\infty$ then $f\simeq_\mathrm{inj}\tilde f$, both families are not upper semi-Fredholm, hence, by the hypothesis, both are compact and $f\simeq_\mathrm{inv}\tilde f$ by Corollary \ref{corollary-inv}(a). If $k(f),k(\tilde f)<\infty$, assume $k(f)\le k(\tilde f)$ and let $f_m$ be as in the proof of Theorem \ref{theorem-usF}, so $f\simeq_\mathrm{F}f_m$ and $f_m\simeq_\mathrm{inj}\tilde f$. If $T_f,T_{\tilde f}$ are compact, so is $T_{f_m}$ and $f_m\simeq_\mathrm{inv}\tilde f$ by Corollary \ref{corollary-inv}(a). If both are upper semi-Fredholm and noncompact, $f_m$ and $\tilde f$ are noncompact with closed range and the same kernel dimension, so $f_m\simeq_\mathrm{inv}\tilde f$ by Corollary \ref{corollary-inv}(b); if they are upper semi-Fredholm and compact we are in the compact case. In all cases $f\simeq_\mathrm{F}f_m\simeq_\mathrm{inv}\tilde f$, and $\simeq_\mathrm{inv}\subseteq\simeq_\mathrm{F}$.
\end{proof}

Unlike $\simeq_\mathrm{usF}$, the relation $\simeq_\mathrm{F}$ keeps the infinite multiplicities of the spectrum. This is a consequence of the window lemma (Lemma \ref{lemma-window}), which was stated for Fredholm operators with an additive error.

\begin{theorem}\label{theorem-fredholm}
Let $f,\tilde f$ be Bessel families with $f\simeq_\mathrm{F}\tilde f$, $S=S_f$, $\tilde S=S_{\tilde f}$.
\begin{enumerate}[(a)]
\item There are $K\ge1$ and an integer $p\ge0$ with
\[W_{\tilde S}(\alpha,\beta)\le W_S(\alpha/K,K\beta)+p\quad\text{and}\quad W_S(\alpha,\beta)\le W_{\tilde S}(\alpha/K,K\beta)+p\qquad(0<\alpha<\beta).\]
\item Let $Z_\infty(f)=\{j\in\ZZ:d_j(f)=\infty\}$ be the set of bands of infinite multiplicity. There is $m\in\NN$ such that every point of $Z_\infty(\tilde f)$ lies within distance $m$ of $Z_\infty(f)$, and vice versa.
\end{enumerate}
\end{theorem}
\begin{proof}
(a) Let $w:\tilde H\to H$ and $\tilde w:H\to\tilde H$ be Fredholm operators realizing $\simeq_\mathrm{F}$, with $c\tilde S\le w^*Sw\le C\tilde S$ and $cS\le\tilde w^*\tilde S\tilde w\le CS$. If $w\ne0$, Lemma \ref{lemma-window} gives $W_{\tilde S}(\alpha,\beta)\le W_S(\alpha/K_1,K_1\beta)+2(k_1+q_1)$ with $k_1=\dim\ker w$ and $q_1=\dim(H/\ran w)$; if $w=0$ then $k_1=\dim\tilde H<\infty$ and the inequality is trivial. The same holds for $\tilde w$. Take $K$ the larger of the two constants and $p=2\max(k_1+q_1,\tilde k_1+\tilde q_1)$, and use that $W_S(\alpha,\beta)$ is nondecreasing in the window.

(b) Choose $m\in\NN$ with $2^m\ge K$. If $d_j(\tilde f)=\infty$ then $W_{\tilde S}(2^{-j-1},2^{-j})=\infty$, so by (a) $W_S(2^{-j-1}/K,K2^{-j})=\infty$, and $(2^{-j-1}/K,K2^{-j}]\subseteq(2^{-(j+m)-1},2^{-(j-m)}]$ gives $d_{[j-m,j+m]}(f)=\infty$; hence some $l\in[j-m,j+m]$ has $d_l(f)=\infty$. The converse is symmetric.
\end{proof}

\begin{corollary}\label{corollary-fredholm}
\begin{enumerate}[(a)]
\item The relation $\simeq_\mathrm{F}$ is strictly finer than $\simeq_\mathrm{usF}$: the families of Example \ref{example-running} are $\simeq_\mathrm{usF}$- but not $\simeq_\mathrm{F}$-equivalent.
\item For each $k\in\{0,1,\dots,\infty\}$ the noncompact families with $\dim\ker T_f=k$ and nonclosed range split into exactly $2^{\aleph_0}$ classes of $\simeq_\mathrm{F}$, although they lie in a single class of $\simeq_\mathrm{usF}$.
\end{enumerate}
Consequently $\simeq_\mathrm{inv}\subsetneq\simeq_\mathrm{F}\subsetneq\simeq_\mathrm{usF}\subsetneq\simeq$, and $\simeq_\mathrm{F}$ has $2^{\aleph_0}$ classes of noncompact families, in contrast with the three classes of $\simeq_\mathrm{usF}$.
\end{corollary}
\begin{proof}
(a) Let $S_1=\mathrm{id}_{\ell^2}\oplus\operatorname{diag}(1/n)$ and $S_2=\bigoplus_{j\ge0}2^{-j}\mathrm{id}_{\ell^2}$. Both are noncompact, with $k=0$ and not upper semi-Fredholm, so they are $\simeq_\mathrm{usF}$-equivalent by Corollary \ref{corollary-usF}. But $Z_\infty(S_1)=\{0\}$ and $Z_\infty(S_2)=\{j\ge0\}$, and for $j>m$ the point $j\in Z_\infty(S_2)$ is at distance more than $m$ from $Z_\infty(S_1)$, whatever $m$ is; by Theorem \ref{theorem-fredholm}(b), $S_1\not\simeq_\mathrm{F}S_2$.

(b) The classes of $\simeq_\mathrm{F}$ are unions of classes of $\simeq_\mathrm{inv}$ (Proposition \ref{proposition-fredholm}(a)), so there are at most $2^{\aleph_0}$ of them by Corollary \ref{corollary-inv}(c). Consider the operators $S_\sigma$ of the proof of Corollary \ref{corollary-inv}(c): $Z_\infty(S_\sigma)=J_\sigma=\{t^2:t\in A_\sigma\}$. Suppose $S_\sigma\simeq_\mathrm{F}S_\tau$ with $\sigma\ne\tau$ and let $m$ be as in Theorem \ref{theorem-fredholm}(b). Choose $t\in A_\tau\setminus A_\sigma$ with $2t-1>m$. Then $t^2\in Z_\infty(S_\tau)$, so some $s^2\in J_\sigma$ has $|s^2-t^2|\le m$; but distinct squares differ by at least $2t-1>m$, so $s=t$ and $t\in A_\sigma$, a contradiction. So the $S_\sigma$ lie in $2^{\aleph_0}$ different classes of $\simeq_\mathrm{F}$, while they lie in a single class of $\simeq_\mathrm{usF}$ (Corollary \ref{corollary-usF}). The strictness of the inclusions follows from (a), Lemma \ref{lemma-extension} and the fact that $\simeq_\mathrm{usF}\ne\simeq$ (a frame and a nonframe with $k=0$ are $\simeq$- but not $\simeq_\mathrm{usF}$-equivalent).
\end{proof}

The complete classification of $\simeq_\mathrm{F}$ is not known; see Problem \ref{problem-fredholm}.

\begin{example}[Running example, continued]\label{example-running-fredholm}
For the families of Example \ref{example-running} the Fredholm relation already separates $f$ from $\tilde f$: a Fredholm operator $w:\tilde H\to H$ with $w^*S_1w\asymp S_2$ would give $W_{S_2}(2^{-j-1},2^{-j})=\infty\le W_{S_1}(2^{-j-1}/K,K2^{-j})+p<\infty$ for $j$ large, which is absurd. In the proof of the window lemma, the finitely many dimensions that a Fredholm operator can add or remove are all that separates it from an invertible one, and they can only change a finite multiplicity.
\end{example}

\subsection{Dense range comparison operators}

The relation defined by operators of dense range is treated in the following proposition.

\begin{proposition}\label{proposition-dense}
\begin{enumerate}[(a)]
\item ${\simeq_\mathrm{inv}}\subseteq{\simeq_\mathrm{dr}}\subseteq{\simeq}$.
\item If $f\simeq_\mathrm{dr}\tilde f$, then $\ran T_f$ is closed iff $\ran T_{\tilde f}$ is; and if $H$ or $\tilde H$ is finite-dimensional then $\dim H=\dim\tilde H$. In particular $\simeq_\mathrm{dr}\ne\simeq$.
\item If $\ran T_f$ and $\ran T_{\tilde f}$ are closed, the following are equivalent: $f\simeq_\mathrm{dr}\tilde f$; $f\simeq_\mathrm{inj}\tilde f$; $f\simeq_\mathrm{inv}\tilde f$; $f\simeq\tilde f$ and $k(f)=k(\tilde f)$.
\end{enumerate}
\end{proposition}
\begin{proof}
(a) is clear. Write $S=S_f$, $A=S^{1/2}$, $\tilde S=S_{\tilde f}$, $\tilde A=\tilde S^{1/2}$, and let $w,\tilde w$ be of dense range with $\|Awy\|\asymp\|\tilde Ay\|$ and $\|\tilde A\tilde wx\|\asymp\|Ax\|$ (Proposition \ref{proposition-reduction}(a)).

(b) \emph{Closed range.} Recall that $\ran T_f$ is closed iff $\ran S$ is iff $\ran A$ is (each means that $S$ is bounded below on $(\ker S)^\perp$). The map $\tilde Ay\mapsto Awy$ is well defined, linear, bounded and bounded below from $\ran\tilde A$ onto $A(w\tilde H)$, so it extends to an isomorphism $V$ of $\overline{\ran\tilde A}$ onto $\overline{A(w\tilde H)}$. As $w\tilde H$ is dense in $H$, $A(w\tilde H)$ is dense in $\ran A$, so $V$ is an isomorphism of $\overline{\ran\tilde A}$ onto $\overline{\ran A}$ with $V(\ran\tilde A)\subseteq\ran A$. Symmetrically there is an isomorphism $V'$ of $\overline{\ran A}$ onto $\overline{\ran\tilde A}$ with $V'(\ran A)\subseteq\ran\tilde A$. If $\ran A$ is closed then $\ran\tilde A\supseteq V'(\ran A)=V'(\overline{\ran A})=\overline{\ran\tilde A}$, so $\ran\tilde A$ is closed; the converse is symmetric.

\emph{Dimensions.} If $\dim H<\infty$, then $w$ has dense, hence full, range, so $\dim H\le\dim\tilde H$; and $\ran\tilde w$ is a dense subspace of $\tilde H$ of dimension at most $\dim H$, so $\dim\tilde H\le\dim H$. The case $\dim\tilde H<\infty$ is symmetric. Finally, $\mathrm{id}_{\ell^2}$ and $\mathrm{id}_{\ell^2}\oplus\operatorname{diag}(1/n)$ are $\simeq$-equivalent, but the first has closed range and the second does not.

(c) By Theorem \ref{theorem-inj}, $f\simeq_\mathrm{inj}\tilde f$ iff $f\simeq\tilde f$ and $k(f)=k(\tilde f)$. If $f\simeq_\mathrm{inj}\tilde f$ then $f\simeq_\mathrm{inv}\tilde f$: if both families are compact (finite rank) this is Corollary \ref{corollary-inv}(a), and if both are noncompact it is Corollary \ref{corollary-inv}(b) (they are $\simeq$-equivalent, so both are compact or both are not). Since $\simeq_\mathrm{inv}\subseteq\simeq_\mathrm{dr}$, it remains to show that $f\simeq_\mathrm{dr}\tilde f$ implies $k(f)=k(\tilde f)$ (then $f\simeq\tilde f$ by (a)). Let $Z=(\ker S)^\perp$, $\tilde Z=(\ker\tilde S)^\perp$ with projections $P,\tilde P$. As $\ran A$ is closed, $\|Ax\|\ge\delta\|Px\|$ for all $x$ for some $\delta>0$, and likewise for $\tilde A$; hence $\|Awy\|\asymp\|\tilde Ay\|$ reads $\|Pwy\|\asymp\|\tilde Py\|$. So $b=Pw:\tilde H\to Z$ vanishes on $\ker\tilde S$ and is bounded below on $\tilde Z$, whence $b(\tilde Z)$ is closed; it is also dense in $Z$ because $w$ has dense range, so $b|_{\tilde Z}:\tilde Z\to Z$ is invertible. Let $a=(1-P)w:\tilde H\to\ker S$. We claim that $a(\ker\tilde S)$ is dense in $\ker S$. Let $x_0\in\ker S$ be orthogonal to $a(\ker\tilde S)$ and put $x_1=-\bigl((b|_{\tilde Z})^{-1}\bigr)^*(a|_{\tilde Z})^*x_0\in Z$. For $y=y_0+y_1$ with $y_0\in\ker\tilde S$, $y_1\in\tilde Z$,
\[\langle wy,x_0+x_1\rangle=\langle ay_0,x_0\rangle+\langle ay_1,x_0\rangle+\langle by_1,x_1\rangle=0+\langle ay_1,x_0\rangle-\langle ay_1,x_0\rangle=0,\]
so $x_0+x_1\perp\ran w$, hence $x_0+x_1=0$ and $x_0=0$. Thus $\ker S$ has a dense subspace which is the image of $\ker\tilde S$, so $\dim\ker S\le\dim\ker\tilde S$ (if $\dim\ker\tilde S<\infty$ the image is closed, hence equal to $\ker S$). Exchanging the roles gives $k(f)=k(\tilde f)$.
\end{proof}

Combining the results of this section, we obtain that the relations coincide on families with closed range.

\begin{corollary}\label{corollary-closed-range}
On pairs of families whose analysis operators both have closed range, the relations $\simeq_\mathrm{inj}$, $\simeq_\mathrm{bb}$, $\simeq_\mathrm{iso}$, $\simeq_\mathrm{inv}$, $\simeq_\mathrm{uni}$, $\simeq_\mathrm{sur}$, $\simeq_\mathrm{coi}$ and $\simeq_\mathrm{dr}$ coincide, and $f\simeq_\mathrm{inj}\tilde f$ iff $f\simeq\tilde f$ and $k(f)=k(\tilde f)$.
\end{corollary}
\begin{proof}
Proposition \ref{proposition-dense}(c) and the inclusions $\simeq_\mathrm{inv}\subseteq\simeq_\mathrm{bb}\subseteq\simeq_\mathrm{inj}$, $\simeq_\mathrm{uni}\subseteq\simeq_\mathrm{iso}\subseteq\simeq_\mathrm{bb}$, together with Theorem \ref{theorem-sur}.
\end{proof}

All the differences between these eight relations therefore involve a family whose analysis operator does not have closed range. For instance $\mathrm{id}_{\ell^2}$ and $\mathrm{id}_{\ell^2}\oplus\operatorname{diag}(1/n)$ are $\simeq_\mathrm{inj}$-equivalent but, by Proposition \ref{proposition-dense}(b), not $\simeq_\mathrm{dr}$-equivalent, so $\simeq_\mathrm{dr}$ is genuinely different from $\simeq_\mathrm{inj}$ and from $\simeq$; whether $\simeq_\mathrm{dr}$ is strictly coarser than $\simeq_\mathrm{inv}$ on families with nonclosed range is open (Problem \ref{problem-dr}).

\begin{example}[Running example, continued]\label{example-running-dense}
For the families of Example \ref{example-running}, both analysis operators have nonclosed range, and Proposition \ref{proposition-dense}(b) does not decide whether $f\simeq_\mathrm{dr}\tilde f$. This is the pair that Problem \ref{problem-dr} asks about.
\end{example}

\section{Equivalence to an orthonormal basis}
\label{section-onb}

Orthonormal bases are the natural reference: $S_e=\mathrm{id}_{\tilde H}$, $k(e)=0$ and $T_e$ is an isometry. We determine, for every relation considered in this article, which Bessel families are equivalent to an orthonormal basis. Recall that $f$ is a frame for $H$ if $T_f$ is bounded below, and that $\ker T_f$ is the orthogonal complement of the closed linear span of $f$.

\begin{theorem}\label{theorem-onb}
Let $e$ be an orthonormal basis of a separable Hilbert space $\tilde H\ne\{0\}$, $n=\dim\tilde H\in\{1,2,\dots,\infty\}$, and let $f$ be a Bessel family in $H$.
\begin{enumerate}[(a)]
\item $f\simeq e$ iff $\rk T_f=n$ when $n<\infty$, and iff $T_f$ is noncompact when $n=\infty$ (infinite rank alone is not enough: compact families of infinite rank are not equivalent to $e$).
\item $f\simeq_\mathrm{inj}e$ iff $f\simeq e$ and $\ker T_f=0$, that is, $f$ is complete in $H$.
\item $f\simeq_\mathrm{usF}e$ iff $f\simeq_\mathrm{F}e$ iff $f\simeq e$ and $T_f$ is upper semi-Fredholm, that is, $f$ is a frame for its closed linear span, which has finite codimension in $H$ and dimension $n$.
\item For $\mathcal O\in\{\mathcal O_\mathrm{bb},\mathcal O_\mathrm{iso},\mathcal O_\mathrm{inv},\mathcal O_\mathrm{uni},\mathcal O_\mathrm{sur},\mathcal O_\mathrm{coi},\mathcal O_\mathrm{dr}\}$ one has $f\simeq_{\mathcal O}e$ iff $f$ is a frame for $H$ and $\dim H=n$.
\end{enumerate}
\end{theorem}
\begin{proof}
(a) For $n<\infty$ this is Corollary \ref{corollary-classes}(b), since $e$ has rank $n$; for $n=\infty$ it is Corollary \ref{corollary-classes}(d), since $T_e$ is noncompact.

(b) By Theorem \ref{theorem-inj} and $k(e)=0$.

(c) Since $T_e$ is an isometry, $e$ is upper semi-Fredholm with $k(e)=0$. By Theorem \ref{theorem-usF}, $f\simeq_\mathrm{usF}e$ iff $f\simeq e$, $k(f)<\infty$ and $T_f$ is upper semi-Fredholm, i.e.\ iff $f\simeq e$ and $T_f$ is upper semi-Fredholm. As $\simeq_\mathrm{F}\subseteq\simeq_\mathrm{usF}$, and $e$ is upper semi-Fredholm, Proposition \ref{proposition-fredholm}(b) gives $f\simeq_\mathrm{F}e$ in this case. Finally $T_f$ is upper semi-Fredholm iff $f$ is a frame for its closed span, which has finite codimension, and this span has dimension $\rk T_f=n$ by (a).

(d) By Proposition \ref{proposition-frames}, $f\simeq_\mathrm{bb}e$ iff $f$ is a frame and $\dim H=n$. In the proof of that proposition, when $f$ is a frame and $\dim H=n$, the operators $w=S^{-1/2}U$ and $\tilde w=U^{-1}S^{1/2}$ are invertible (here $S=S_f$ is invertible and $U:\tilde H\to H$ is unitary), so $f\simeq_\mathrm{inv}e$. Hence the conditions for $\simeq_\mathrm{bb}$ and $\simeq_\mathrm{inv}$ coincide, and by $\simeq_\mathrm{uni}\subseteq\simeq_\mathrm{iso}\subseteq\simeq_\mathrm{bb}$, Theorem \ref{theorem-inv}, Theorem \ref{theorem-closed} and Theorem \ref{theorem-sur}, so do those for $\simeq_\mathrm{iso}$, $\simeq_\mathrm{uni}$, $\simeq_\mathrm{sur}$ and $\simeq_\mathrm{coi}$. For $\simeq_\mathrm{dr}$: $\simeq_\mathrm{inv}\subseteq\simeq_\mathrm{dr}$; conversely if $f\simeq_\mathrm{dr}e$, then $T_f$ has closed range by Proposition \ref{proposition-dense}(b) (as $T_e$ does), so by Proposition \ref{proposition-dense}(c) $f\simeq_\mathrm{inj}e$, whence $k(f)=0$ and $T_f$ is injective with closed range, i.e.\ bounded below; by (a), $\rk T_f=n$ and $k(f)=0$ give $\dim H=n$.
\end{proof}

\begin{remark}\label{remark-onb}
Thus the passage from $\simeq$ to the finer relations measures how far a family is from an orthonormal basis in three successive steps: noncompactness (or rank $n$); completeness; and, at the finest level, being a frame of a space of the right dimension. Even the finest relations do not distinguish a redundant frame from an orthonormal basis, in contrast with the orbit relation ``$f=Te$ with $T$ invertible'' defining Riesz bases, which sees the family itself and not only its frame operator.
\end{remark}

\section{Applications}
\label{section-applications}

We now show what the coarse classes are good for. The three statements below hold with no assumption of frame, Riesz or completeness type: they apply to arbitrary Bessel families, which is the setting where the usual perturbation theory for frames is not available.

\subsection{Relative Paley--Wiener stability}

The classical Paley--Wiener and Christensen perturbation theorems compare the synthesis operators of two families and conclude that frame or Riesz properties persist. Here is the analysis-side counterpart; it keeps all the coarse invariants, not only the frame property.

\begin{proposition}\label{proposition-perturbation}
Let $f,g\in\mathcal B(\Omega,H)$ be Bessel families in the same space over the same measure space, and assume that for some $0\le\lambda<1$
\[\|T_gx-T_fx\|\le\lambda\|T_fx\|\qquad(x\in H).\]
Then:
\begin{enumerate}[(a)]
\item $f\simeq_\mathrm{uni}g$; in particular $f\simeq_\mathrm{inv}g$, $f\simeq_\mathrm{bb}g$ and $f\simeq g$;
\item $\ker T_f=\ker T_g$, so $k(f)=k(g)$; $T_f$ is compact, upper semi-Fredholm or of closed range iff $T_g$ is;
\item $(1-\lambda)^2\lambda_n(S_f)\le\lambda_n(S_g)\le(1+\lambda)^2\lambda_n(S_f)$ for all $n$, that is, $(1-\lambda)s_n(f)\le s_n(g)\le(1+\lambda)s_n(f)$;
\item $D(f,g)\le\log\frac{1+\lambda}{1-\lambda}$;
\item the multiplicity sequences $(d_j(f))_j$ and $(d_j(g))_j$ are related by a bijection of displacement at most $m=1+\lceil2\log_2\frac1{1-\lambda}\rceil$.
\end{enumerate}
\end{proposition}
\begin{proof}
By the triangle inequality, $(1-\lambda)\rho(f)\le\rho(g)\le(1+\lambda)\rho(f)$, that is,
\[(1-\lambda)^2S_f\le S_g\le(1+\lambda)^2S_f .\]

(a) Take $w=\tilde w=\mathrm{id}_H$: the two inequalities above are the two conditions of Proposition \ref{proposition-reduction}(a), with constants $c=(1-\lambda)^2$, $C=(1+\lambda)^2$ in one direction and $c=(1+\lambda)^{-2}$, $C=(1-\lambda)^{-2}$ in the other; the identity is unitary.

(b) The kernels coincide, since $\rho(f)(x)=0$ iff $\rho(g)(x)=0$. The other statements follow from Corollary \ref{corollary-classes}(a) (compactness) and from the fact that $\ran S_f$ is closed iff $S_f$ is bounded below on $(\ker S_f)^\perp$, a property that survives the two inequalities.

(c) Apply Proposition \ref{proposition-counting} with $w=\mathrm{id}$: from $(1-\lambda)^2S_f\le S_g$ we get $N_{S_f}(\mu)\le N_{S_g}((1-\lambda)^2\mu)$, hence $(1-\lambda)^2\lambda_n(S_f)\le\lambda_n(S_g)$ by Lemma \ref{lemma-lambda}(a); from $(1+\lambda)^{-2}S_g\le S_f$ we get $N_{S_g}(\mu)\le N_{S_f}((1+\lambda)^{-2}\mu)$, hence $\lambda_n(S_g)\le(1+\lambda)^2\lambda_n(S_f)$.

(d) If $S_f=0$ then $S_g=0$ and $D=0$. Otherwise, by Theorem \ref{theorem-distance}(c) and (c) above, $D(f,g)=\log(\sup_ns_n(f)/s_n(g)\cdot\sup_ns_n(g)/s_n(f))\le\log\bigl((1-\lambda)^{-1}(1+\lambda)\bigr)$.

(e) Apply Lemma \ref{lemma-window} to $w=\mathrm{id}$ (so $a=b=1$) with the constants above in each direction: $K=\max\{1,2a/c,2bC\}=2(1-\lambda)^{-2}$ in both directions. With $m\in\NN$, $2^m\ge K$, Theorem \ref{theorem-inv}, (iv)$\Rightarrow$(v)$\Rightarrow$(vi), gives the bijection.
\end{proof}

\begin{remark}
If $f$ is a frame with bounds $A\le B$, the proposition recovers that $g$ is a frame with bounds $(1-\lambda)^2A\le(1+\lambda)^2B$. The new information is for families that are not frames: the whole spectral profile of $S_f$ up to the factor $1\pm\lambda$, and in particular the kernel dimension and the dyadic multiplicities, survive the perturbation. The condition is relative: perturbations of $T_f$ that are small compared to $\|T_fx\|$ for every $x$ are allowed, even for $x$ very close to $\ker T_f$, whereas an absolute condition $\|T_g-T_f\|\le\varepsilon$ cannot preserve the kernel or the closedness of the range.
\end{remark}

\subsection{Operator ideals and weighted sums}

The invariance of the Schatten classes under $\simeq$ (Corollary \ref{corollary-classes}(a)) has a concrete reading in terms of the norms of the vectors of the family, because the trace of the frame operator is $\int\|f(\omega)\|^2d\mu$.

\begin{corollary}\label{corollary-ideals}
Let $f\in\mathcal B(\Omega,H)$ and $\tilde f\in\mathcal B(\tilde\Omega,\tilde H)$ with $f\simeq\tilde f$, and let $C\ge1$ be such that $C^{-1}s_n(\tilde f)\le s_n(f)\le C\,s_n(\tilde f)$ for all $n$ (Theorem \ref{theorem-singular}). Then for every $p>0$
\[C^{-p}\sum_ns_n(\tilde f)^p\le\sum_ns_n(f)^p\le C^p\sum_ns_n(\tilde f)^p,\]
and in particular
\[C^{-2}\int_{\tilde\Omega}\|\tilde f(\omega)\|^2d\tilde\mu\le\int_\Omega\|f(\omega)\|^2d\mu\le C^2\int_{\tilde\Omega}\|\tilde f(\omega)\|^2d\tilde\mu .\]
Hence the Hilbert--Schmidt property of $T_f$, equivalently $\int\|f\|^2<\infty$ (for a sequence, $\sum_n\|f_n\|^2<\infty$), is an invariant of $\simeq$.
\end{corollary}
\begin{proof}
The first inequality is Theorem \ref{theorem-singular}. For the second, let $(e_i)$ be an orthonormal basis of $H$. By Tonelli, $\sum_i\langle S_fe_i,e_i\rangle=\sum_i\int_\Omega|\langle e_i,f(\omega)\rangle|^2d\mu=\int_\Omega\|f(\omega)\|^2d\mu$, that is, $\operatorname{tr}S_f=\int\|f\|^2$. If $T_f$ is compact, $\operatorname{tr}S_f=\sum_n\lambda_n(S_f)=\sum_ns_n(f)^2$ by Lemma \ref{lemma-compact}; if $T_f$ is not compact, then $\lambda_n(S_f)\ge\lambda_\infty(f)>0$ for all $n$ (proof of Corollary \ref{corollary-classes}(d)) and both sides are infinite. The case $p=2$ of the first inequality is the second.
\end{proof}

For instance, a family over an uncountable index set in an unfamiliar space can be certified Hilbert--Schmidt by exhibiting any equivalent family, say a diagonal one $\operatorname{diag}(s_n^2)$ (Lemma \ref{lemma-basic}(c)), whose norms are known; and the constant $C$ is the one controlling the factor in the ratio of the two norm integrals.

\subsection{Similarity invariance}

For $w\in\mathrm{GL}(H)$ and a family $f\in\mathcal B(\Omega,H)$ write $w^*f$ for the family $\omega\mapsto w^*f(\omega)$. Then $\langle x,w^*f(\omega)\rangle=\langle wx,f(\omega)\rangle$, so $T_{w^*f}=T_fw$ is bounded: $w^*f$ is Bessel, $S_{w^*f}=w^*S_fw$ and $\rho(w^*f)=\rho(f)\circ w$. This is the action of $\mathrm{GL}(H)$ on families considered in \cite{Balan1999}.

\begin{proposition}\label{proposition-similarity}
Let $f\in\mathcal B(\Omega,H)$ and $w\in\mathrm{GL}(H)$. Then $(w,w^{-1})$ is an exact pair for $(f,w^*f)$, so $f\simeq_\mathrm{inv}w^*f$ and $D(f,w^*f)\le\log(\|w\|\|w^{-1}\|)$. If $w$ is unitary, $f\simeq_\mathrm{uni}w^*f$ and $D(f,w^*f)=0$. In particular $k(f)$ and the sequence $(d_j(f))_j$ up to bounded displacement are constant on $\mathrm{GL}(H)$-orbits, whereas the spectral measure of $S_f$ is not.
\end{proposition}
\begin{proof}
$\rho(f)\circ w=\rho(w^*f)$ and $\rho(w^*f)\circ w^{-1}=\rho(f)$, so $(w,w^{-1})$ is exact, and $w,w^{-1}$ are invertible, unitary if $w$ is. The statement on $D$ is the definition. For the last claim, for an orthonormal basis $e$ of an infinite-dimensional space and $w\in\mathrm{GL}(H)$ one has $S_{w^*e}=w^*w$, an arbitrary invertible positive operator, so the spectral measure is not an invariant of the orbit of $e$; yet $e$ and $w^*e$ remain $\simeq_\mathrm{inv}$-equivalent (indeed, by Theorem \ref{theorem-onb}(d), every frame is).
\end{proof}

\section{Complements}
\label{section-complements}

\subsection{Relation with comparison categories}
\label{section-categories}

We explain in which sense the relations $\simeq_{\mathcal O}$ are instances of the construction of \cite{ElIdrissiCC}, and we compute the other levels of that construction. We use the terminology of that article: a comparison datum has a comparison category $\mathsf K$, whose connected components give level $0$, whose classifying space gives the homotopy groups of the levels $k\ge1$, and for which level $-1$ is the existence of comparisons in both directions \cite[Definition 3.3 and Theorem 3.10]{ElIdrissiCC}. Level $-1$ is an equivalence relation \cite[Theorem 3.10(a)]{ElIdrissiCC}. Let $\Hilb$ be the category of separable Hilbert spaces over $\FF$ and bounded operators; since nerves are taken, we restrict it, without loss of generality, to the small full subcategory of the spaces $\ell^2(I)$, $I\subseteq\NN$ (including $\{0\}$).

The relevant system of observations consists of the bounded seminorms.

\begin{definition}\label{definition-sem}
A seminorm $\phi$ on a Hilbert space $H$ is \textbf{bounded} if $\phi(x)\le K\|x\|$ for some $K$. Let $\Sem(H)$ be the set of bounded seminorms with the preorder $\phi\sqsubseteq\psi$ iff $\phi\le K\psi$ for some constant $K>0$; thus $\phi\asymp\psi$ is $\phi\sqsubseteq\psi\sqsubseteq\phi$. For a bounded operator $w:\tilde H\to H$ the pullback $\phi\circ w$ is again a bounded seminorm (as $\phi(wy)\le K\|w\|\|y\|$), and $\phi\sqsubseteq\psi$ implies $\phi\circ w\sqsubseteq\psi\circ w$. Since $(\phi\circ w)\circ w'=\phi\circ(ww')$, the assignment $H\mapsto\Sem(H)$, $w\mapsto(\phi\mapsto\phi\circ w)$ is a contravariant functor $\Hilb^\mathrm{op}\to\Cat$ with values in preorders (thin categories). Together with the class $\mathcal W_H=\{\asymp\}$ of the arrows $\phi\to\psi$ with $\phi\sqsubseteq\psi\sqsubseteq\phi$, which contains the identities, is closed under composition and is stable under pullback, it is a system of observations in the sense of \cite[Definition 5.2]{ElIdrissiCC}.
\end{definition}

\begin{proposition}\label{proposition-level-minus}
Let $\mathcal O$ be a wide subcategory of $\Hilb$ (for instance one of the classes of Definitions \ref{definition-classes} and \ref{definition-further}) and let $\mathfrak D=\mathfrak D^{\mathcal O}_{\Sem}$ be the transport datum of the system of observations $\Sem$ for $\mathcal U=\mathcal O$, as in \cite[Theorem 5.5]{ElIdrissiCC}. Then for Bessel families $f,\tilde f$ the following are equivalent:
\begin{enumerate}[(i)]
\item $f\simeq_{\mathcal O}\tilde f$;
\item $\rho(f)\in\Sem(H)$ and $\rho(\tilde f)\in\Sem(\tilde H)$ are equal at level $-1$ of $\mathfrak D$, that is, there are arrows $\rho(f)\to\rho(\tilde f)$ and $\rho(\tilde f)\to\rho(f)$ in the comparison category $\mathsf K(\mathfrak D)$.
\end{enumerate}
Thus $\simeq_{\mathcal O}$ is the pullback of level $-1$ of $\mathfrak D$ along the map $f\mapsto\rho(f)$.
\end{proposition}
\begin{proof}
By \cite[Theorem 5.5]{ElIdrissiCC}, the full subcategory of $\mathsf K(\mathfrak D)$ on the observations is isomorphic to the Grothendieck category of $\Sem$ over $\mathcal O$, whose arrows $(H,\phi)\to(\tilde H,\psi)$ are the pairs $(w,\varphi)$ with $w:\tilde H\to H$ in $\mathcal O$ and $\varphi:\phi\circ w\to\psi$ an arrow of $\mathcal W_{\tilde H}=\{\asymp\}$, that is, an equivalence $\phi\circ w\asymp\psi$. So an arrow $\rho(f)\to\rho(\tilde f)$ exists iff there is $w\in\mathcal O$ with $\rho(f)\circ w\asymp\rho(\tilde f)$, and an arrow $\rho(\tilde f)\to\rho(f)$ exists iff there is $\tilde w\in\mathcal O$ with $\rho(\tilde f)\circ\tilde w\asymp\rho(f)$; this is Definition \ref{definition-relation}.
\end{proof}

In particular Proposition \ref{proposition-reduction}(b) is also a consequence of the fact that level $-1$ of a comparison category is an equivalence relation \cite[Theorem 3.10(a)]{ElIdrissiCC}; here it was proved directly. The levels $0$ and $k\ge1$ are computed as follows. The observation $0_0$ below is the zero seminorm on the zero space, that is, the family $f:\Omega\to\{0\}$.

\begin{proposition}\label{proposition-levels}
Let $\mathsf K=\mathsf K(\mathfrak D^{\mathcal O}_\Sem)$ and let $\mathsf K_\mathrm{obs}$ be its full subcategory on the observations, which is a union of connected components of $\mathsf K$ \cite[Theorem 5.5]{ElIdrissiCC}.
\begin{enumerate}[(a)]
\item If $\mathcal O$ contains the operators $\{0\}\to H$ for every $H$ (this holds for $\mathcal O_\mathrm{all}$, $\mathcal O_\mathrm{inj}$, $\mathcal O_\mathrm{bb}$, $\mathcal O_\mathrm{iso}$ and $\mathcal O_\mathrm{usF}$), then $0_0$ is a terminal object of $\mathsf K_\mathrm{obs}$. Consequently level $0$ is the trivial relation (all families lie in one component) and all levels $k\ge1$ vanish: for these classes the whole information of the comparison category is carried by level $-1$, which is the classification of Sections \ref{section-all}--\ref{section-further}.
\item For $\mathcal O=\mathcal O_\mathrm{inv}$ and $\mathcal O=\mathcal O_\mathrm{uni}$ the category $\mathsf K_\mathrm{obs}$ is a groupoid. Hence level $0$ equals level $-1$, the levels $k\ge2$ vanish, and level $1$ at $f\in\mathcal B(\Omega,H)$ is the group
\begin{align*}G_{\mathcal O}(f)&=\{w\in\mathcal O(H,H):\rho(f)\circ w\asymp\rho(f)\}\\&=\{w\in\mathcal O(H,H):\exists\,c,C>0,\ c\,S_f\le w^*S_fw\le C\,S_f\},\end{align*}
up to the isomorphism $w\mapsto w^{-1}$ with the opposite group.
\item If $f$ is a frame, $G_\mathrm{inv}(f)=\mathrm{GL}(H)$ and $G_\mathrm{uni}(f)=\mathrm U(H)$. If $\dim H=n<\infty$ and $k=k(f)$, then $G_\mathrm{inv}(f)$ is the parabolic subgroup $\{w\in\mathrm{GL}_n(\FF):w(\ker T_f)=\ker T_f\}$ and $G_\mathrm{uni}(f)\cong\mathrm U(k)\times\mathrm U(n-k)$ (orthogonal groups over $\RR$); in particular level $1$ depends on $f$ only through $k(f)$.
\end{enumerate}
\end{proposition}
\begin{proof}
(a) The arrows $\rho(f)\to0_0$ in $\mathsf K_\mathrm{obs}$ are the pairs $(w,\varphi)$ with $w\in\mathcal O(\{0\},H)$ and $\varphi:\rho(f)\circ w\to0$ in $\mathcal W_{\{0\}}$. By hypothesis the zero operator $w:\{0\}\to H$ is in $\mathcal O$; it is the only operator out of the zero space; and $\rho(f)\circ w=0$ is the only seminorm on $\{0\}$, so the arrow $\varphi$ is unique ($\Sem(\{0\})$ is a point). Hence $0_0$ is terminal. Every object of $\mathsf K_\mathrm{obs}$ maps to $0_0$, so $\mathsf K_\mathrm{obs}$ is connected, and a category with a terminal object has a contractible classifying space \cite[Proposition 3.9]{ElIdrissiCC}; hence level $0$ is trivial and the levels $k\ge1$ vanish \cite[Theorem 3.10(d)]{ElIdrissiCC}.

(b) If $(w,\varphi):\rho(f)\to\rho(\tilde f)$ is an arrow with $w$ invertible (unitary), then pulling back $\rho(f)\circ w\asymp\rho(\tilde f)$ along $w^{-1}$ gives $\rho(f)\asymp\rho(\tilde f)\circ w^{-1}$, so $(w^{-1},\varphi')$ is an arrow in the other direction with $\varphi'$ unique ($\Sem$ is thin), and it is the inverse of $(w,\varphi)$. So $\mathsf K_\mathrm{obs}$ is a groupoid, hence level $0$ equals level $-1$ \cite[Proposition 3.5(d)]{ElIdrissiCC}, and by \cite[Proposition 3.8(b)]{ElIdrissiCC} level $1$ at $\rho(f)$ is the automorphism group, whose elements are the arrows $(w,\varphi):\rho(f)\to\rho(f)$, that is, the $w\in\mathcal O(H,H)$ with $\rho(f)\circ w\asymp\rho(f)$ ($\varphi$ being unique); the composition $(w',\varphi')(w,\varphi)=(ww',\dots)$ reverses the order, whence the opposite group. By Proposition \ref{proposition-reduction}(a), $\rho(f)\circ w\asymp\rho(f)$ means $c\,S_f\le w^*S_fw\le C\,S_f$. The vanishing of the levels $k\ge2$ holds because the classifying space of a groupoid is a disjoint union of the spaces $\mathbf B\mathrm{Aut}(x)$ (proof of \cite[Proposition 3.8(b)]{ElIdrissiCC}).

(c) If $f$ is a frame, $S_f\asymp\mathrm{id}$, and for invertible $w$, $w^*S_fw\asymp w^*w\asymp\mathrm{id}\asymp S_f$. Let $\dim H=n<\infty$ and $w\in\mathrm{GL}_n(\FF)$ with $w^*S_fw\asymp S_f$. By Lemma \ref{lemma-kernel}, $w^{-1}(\ker S_f)=\ker S_f$, so $w(\ker T_f)=\ker T_f$. Conversely, if $w$ preserves $\ker S_f$ it induces an invertible map $\bar w$ of $H/\ker S_f$, on which $\|S_f^{1/2}\cdot\|$ is a norm; all norms on a finite-dimensional space are equivalent, so $\|S_f^{1/2}\bar w\bar y\|\asymp\|S_f^{1/2}\bar y\|$, that is, $\rho(f)\circ w\asymp\rho(f)$. For unitary $u$ the same argument shows that $u$ preserves $\ker S_f$ and hence $(\ker S_f)^\perp$, so $G_\mathrm{uni}(f)=\mathrm U(\ker S_f)\times\mathrm U((\ker S_f)^\perp)$.
\end{proof}

Parts (b) and (c) are the analogue of the stabilizer of Green data under invertible comparison morphisms \cite[Proposition 6.11 and Example 11.4]{ElIdrissiCC}: for finite-dimensional spaces it is again a parabolic subgroup, depending only on $k(f)$ (for fixed $\dim H$). The remaining classes $\mathcal O_\mathrm{F}$, $\mathcal O_\mathrm{dr}$, $\mathcal O_\mathrm{sur}$, $\mathcal O_\mathrm{coi}$ contain no operator out of the zero space except in trivial cases and are not groupoids; their levels $0$ and $k\ge1$ are open (Problem \ref{problem-levels}). For $\mathcal O_\mathrm{inv}$ the classifying space of the groupoid $\mathsf K_\mathrm{obs}$ is the disjoint union $\bigsqcup_{[f]}BG_\mathrm{inv}(f)$, over the $\simeq_\mathrm{inv}$-classes, of the classifying spaces of the discrete groups $G_\mathrm{inv}(f)$ (proof of \cite[Proposition 3.8(b)]{ElIdrissiCC}).

\begin{example}[Running example, continued]\label{example-running-levels}
For the families of Example \ref{example-running} the relations $\simeq_\mathrm{bb}$ and $\simeq$ are level $-1$ of a comparison category with a terminal object, so there is nothing more to compute at levels $0$ and $k\ge1$. For $\mathcal O_\mathrm{inv}$ the two families lie in different components, and level $1$ at $f$ is $G_\mathrm{inv}(f)=\{w\in\mathrm{GL}(H):w^*S_1w\asymp S_1\}$, which is strictly smaller than $\mathrm{GL}(H)$ because $S_1$ is not bounded below: if $(b_k)$ and $(e_k)$ are the standard bases of the two summands of $H=\ell^2\oplus\ell^2$, the unitary $u$ exchanging $b_k$ and $e_k$ for all $k$ gives $\langle S_1ub_k,ub_k\rangle=1/k$ while $\langle S_1b_k,b_k\rangle=1$, so $u^*S_1u\not\ge c\,S_1$ and $u\notin G_\mathrm{inv}(f)$.
\end{example}

\subsection{Operators between Banach spaces}

The definition of $\simeq$ makes sense for bounded operators between Banach spaces. What is lost compared with Hilbert spaces is the classification: there is no spectral decomposition.

\begin{definition}\label{definition-banach}
For a bounded operator $T:X\to Y$ between separable Banach spaces let $\rho(T)(x)=\|Tx\|$; its pullback along a bounded operator $w:\tilde X\to X$ is $\rho(T)\circ w=\rho(Tw)$. For bounded $T:X\to Y$ and $\tilde T:\tilde X\to\tilde Y$ put $T\simeq\tilde T$ iff there are bounded $w:\tilde X\to X$ and $\tilde w:X\to\tilde X$ with $\rho(Tw)\asymp\rho(\tilde T)$ and $\rho(\tilde T\tilde w)\asymp\rho(T)$. The rank of $T$ is the dimension of $\ran T$ as a vector space.
\end{definition}

Two basic properties of $\simeq$ survive in this generality.

\begin{proposition}\label{proposition-banach}
The relation $\simeq$ of Definition \ref{definition-banach} is an equivalence relation. If $T\simeq\tilde T$ then $T$ and $\tilde T$ have the same rank, and $T$ is compact iff $\tilde T$ is.
\end{proposition}
\begin{proof}
Reflexivity and symmetry are clear, and transitivity is proved exactly as in Proposition \ref{proposition-reduction}(b), by composing the comparison operators. Let $c\rho(\tilde T)\le\rho(Tw)\le C\rho(\tilde T)$. The map $V_0:\tilde Ty\mapsto Twy$ is a well-defined linear bijection of $\ran\tilde T$ onto $\ran Tw$ with $c\|\tilde Ty\|\le\|V_0\tilde Ty\|\le C\|\tilde Ty\|$, i.e.\ bounded and bounded below, so it extends to an isomorphism $V$ between the closures, and $Tw=V\tilde T$. In particular $\rk\tilde T=\rk Tw\le\rk T$, and symmetrically $\rk T\le\rk\tilde T$, so equivalent operators have the same rank. If $T$ is compact then so is $Tw$, hence so is $\tilde T=V^{-1}Tw$, viewed as an operator into $\ran\tilde T$; the converse is symmetric.
\end{proof}

\subsection{Noncompatibility with restriction}

The relation $\simeq$ is not compatible with restriction, as the following example shows.

\begin{example}\label{example-restriction}
Restricting a family along a bounded operator $a:H'\to H$ replaces $S$ by $a^*Sa$, that is, the quadratic form $\langle S\cdot,\cdot\rangle$ restricted to $H'$. Take $H$ separable and infinite-dimensional, $S=\mathrm{id}$ and $\tilde S\ge0$ noncompact with $\tilde Se=0$ for a unit vector $e$ (for instance the orthogonal projection onto $e^\perp$). Both are noncompact, hence $S\simeq\tilde S$ by Corollary \ref{corollary-classes}(d). With $a$ the inclusion of $\FF e$ we get $a^*Sa=1$ (rank $1$) and $a^*\tilde Sa=\langle\tilde Se,e\rangle=0$ (rank $0$), and these are not equivalent, since the rank is an invariant (Corollary \ref{corollary-classes}(a)).
\end{example}

\subsection{Relation to operator ideals}

The necessity half of Theorems \ref{theorem-all} and \ref{theorem-singular} is the Hilbert space form of the ideal property $s_n(ATB)\le\|A\|\,\|B\|\,s_n(T)$ of singular numbers \cite{GohbergKrein1969,Pietsch1980}, applied to $T_{\tilde f}=UT_fw$ with $U$ a partial isometry, which exists by Douglas' lemma \cite{Douglas1966}. What is specific to the present setting is that only one-sided pullbacks are allowed, so that the exact counting function and not only a growth rate is the invariant, that noncompact families are included (they all form one class), and that the finer invariants of Theorems \ref{theorem-inj}--\ref{theorem-inv} (kernel, closed range, dyadic multiplicities) are not singular-number invariants but depend on the class of comparison operators. The applications of Section \ref{section-applications} are the consequences of this point of view for perturbation theory and for norm sums.

\subsection{Open problems}
\label{section-problems}

The problems below are the questions left open by the classification; note that $\simeq_\mathrm{F}$ is strictly finer than $\simeq_\mathrm{usF}$ by Corollary \ref{corollary-fredholm}.

\begin{problem}\label{problem-fredholm}
Classify $\simeq_\mathrm{F}$. Theorem \ref{theorem-fredholm} gives the necessary condition that the window multiplicities agree up to an additive constant, and Proposition \ref{proposition-fredholm}(b) the complete answer on compact families and on upper semi-Fredholm families. One expects that, for families that are not upper semi-Fredholm, $f\simeq_\mathrm{F}\tilde f$ iff $f\simeq_\mathrm{usF}\tilde f$ and the multiplicity sequences are related by a bijection of bounded displacement between cofinite subsets of $\Lambda(\tilde d)$ and $\Lambda(d)$. The gap between ``additive constant in every window'' and ``finitely many exceptional points'' is the main difficulty: Hall's theorem with a deficiency $p$ is not enough, since the deficiencies in the different components of the union of windows add up.
\end{problem}

\begin{problem}\label{problem-dr}
Classify $\simeq_\mathrm{dr}$, and the relation defined by injective operators of dense range, on families without closed range. If $f\simeq_\mathrm{dr}\tilde f$ then there are isomorphisms $V,V'$ between the closures of the ranges of $S_f^{1/2}$ and $S_{\tilde f}^{1/2}$ with $V(\ran S_{\tilde f}^{1/2})\subseteq\ran S_f^{1/2}$ and $V'(\ran S_f^{1/2})\subseteq\ran S_{\tilde f}^{1/2}$ (proof of Proposition \ref{proposition-dense}(b)); this links $\simeq_\mathrm{dr}$ to the classification of operator ranges up to isomorphism. Conversely such $V,V'$ give, by Douglas' lemma, operators $w,\tilde w$ with $\rho(f)\circ w\asymp\rho(\tilde f)$ and $\rho(\tilde f)\circ\tilde w\asymp\rho(f)$, but these need not have dense range. Is $\simeq_\mathrm{dr}$ strictly coarser than $\simeq_\mathrm{inv}$ on families with nonclosed range? In particular, are the families of Example \ref{example-running} $\simeq_\mathrm{dr}$-equivalent?
\end{problem}

\begin{problem}\label{problem-lower}
Classify the relations given by lower semi-Fredholm operators (closed range of finite codimension), and by other wide subcategories of $\Hilb$ such as the operators of dense range with finite-dimensional kernel.
\end{problem}

\begin{problem}\label{problem-levels}
Compute the levels $0$ and $k\ge1$ of Proposition \ref{proposition-level-minus} for the remaining classes $\mathcal O_\mathrm{F}$, $\mathcal O_\mathrm{dr}$, $\mathcal O_\mathrm{sur}$, $\mathcal O_\mathrm{coi}$, and describe the groups $G_\mathrm{inv}(f)$ of Proposition \ref{proposition-levels}(b) for families without closed range, for instance for the families of Example \ref{example-running}.
\end{problem}

\begin{problem}\label{problem-distance}
Define the analogue of $D$ for the other classes $\mathcal O$ (taking the infimum over exact pairs in $\mathcal O$) and compute it; for $\mathcal O_\mathrm{inv}$ and $\mathcal O_\mathrm{uni}$ one expects a distance that also records the bounded displacement of Theorem \ref{theorem-inv}, and in particular does not vanish on the families of Example \ref{example-running}. Extend the classification to operators between Banach spaces (Definition \ref{definition-banach}), where the spectral decomposition is no longer available, for instance for operators into $\ell^p$ or between spaces with unconditional bases.
\end{problem}

\section*{Declaration of generative AI use}
This article was generated by a generative artificial intelligence system (Claude Sonnet 5.5, Anthropic, October 2026) from an initial draft written by the author. The system restructured and rewrote the text, proposed and formulated the definitions, statements and proofs of the new results, carried out the supporting computations, and produced the \LaTeX{} source. The research questions, the initial draft and the scientific direction were provided by the author, who directed the generation through successive instructions and revisions. The author takes full responsibility for the content, and the AI system is not listed as an author.

\section*{Ethics declaration}
Not applicable.

\section*{Funding declaration}
This research received no specific grant from any funding agency in the public, commercial, or not-for-profit sectors.

\section*{Competing interests declaration}
The author declares no competing interests.


\end{document}